\documentclass[11pt]{article}
\usepackage{amsmath}
\usepackage{graphicx,psfrag,epsf}
\usepackage[dvipsnames]{xcolor}
\usepackage{enumerate}
\usepackage{url} 
\usepackage{appendix}
\RequirePackage{amsthm,amsmath,amsfonts,amssymb,eufrak}
\RequirePackage{natbib}
\RequirePackage[colorlinks,citecolor=blue,urlcolor=blue]{hyperref}
\RequirePackage{graphicx}
\RequirePackage{epsfig,longtable,xcolor,algpseudocode}
\RequirePackage{makecell,array,booktabs}
\usepackage[linesnumbered, ruled, vlined]{algorithm2e}

\newcommand{\blind}{1}
\usepackage[margin=1in]{geometry}

\theoremstyle{plain}

\newtheorem{theorem}{Theorem}[section]
\newtheorem{lemma}{Lemma}[section]
\newtheorem{cor}{Corollary}[section]
\theoremstyle{remark}

\newtheorem{assume}{Assumption}[section]

\theoremstyle{definition}

\newcommand{\bfm}[1]{\ensuremath{\mathbf{#1}}}
\def\ba{\bfm a}     \def\bA{\bfm A}          
\def\bb{\bfm b}

     \def\bE\\
{\bfm E}     \def\cE{{\cal  E}}

\def\bh{\bfm h}     \def\bH{\bfm H}          
     \def\bI{\bfm I}          
          \def\cJ{{\cal  J}}     
               
          \def\cL{{\cal  L}}     
               
          \def\cN{{\cal  N}}     
               
     \def\bP{\bfm P}     \def\cP{{\cal  P}}

          \def\cS{{\cal  S}}     
          \def\cT{{\cal  T}}     
\def\bu{{\bfm u}}               
\def\bv{\bfm v}     \def\bV{\bfm V}          
               
\def\bx{\bfm x}     \def\bX{\bfm X}          
\def\by{\bfm y}               
\def\bz{\bfm z}               
\def\bzero{\bfm 0}

\newcommand{\bfsym}[1]{\ensuremath{\boldsymbol{#1}}}
       \def \bbeta    {\bfsym{\beta}}
\def \bgamma   {\bfsym{\gamma}}       
\def \bepsilon {\bfsym{\epsilon}}     
         \def \btheta   {\bfsym{\theta}}
        
\def \blambda  {\bfsym{\lambda}}      \def \bmu      {\bfsym{\mu}}

\def \bGamma   {\bfsym{\Gamma}}

\def \bSigma   {\bfsym{\Sigma}}

\DeclareMathOperator*{\argmin}{argmin}
\DeclareMathOperator*{\argmax}{argmax}
\DeclareMathOperator{\var}{var}
\DeclareMathOperator{\cov}{cov}
\DeclareMathOperator{\fdp}{FDP}
\DeclareMathOperator{\tpp}{TPP}

\newcommand{\E}{\mathbb{E}}

\DeclareMathOperator{\supp}{supp}

\DeclareMathOperator \col {\mathrm{col}}

\def \bss{\mathrm{best}}
\def \bone   {\bfsym{1}}
\def \AA {\mathbb{A}}
\def \RR	{\mathbb{R}}
\def \PP {\mathbb{P}}
\def \SS {\mathbb{S}}

\def \NN {\mathbb{N}}
\def \EE {\mathbb{E}}

\def \mU {\mathcal{U}}
\def \mV {\mathcal{V}}
\def \mS {\mathcal{S}}

\def \mN {\mathcal{N}}
\def \mL {\mathcal{L}}
\def \mW {\mathcal{W}}
\def \mT {\mathcal{T}}
\def \mG {\mathcal{G}}

\def \fdr {\mathrm{FDR}}
\def \psm {\mathfrak{m}}

\def \wh{\widehat}
\def \wt{\widetilde}

\newcommand{\beq}  {\begin{equation}}
\newcommand{\eeq}  {\end{equation}}
\newcommand{\beqn} {\begin{eqnarray}}
\newcommand{\eeqn} {\end{eqnarray}}
\newcommand{\beqnn}{\begin{eqnarray*}}
\newcommand{\eeqnn}{\end{eqnarray*}}

\newcommand{\lzeronorm}[1]{\lVert#1\rVert_0}
\newcommand{\lonenorm}[1]{\lVert#1\rVert_1}
\newcommand{\ltwonorm}[1]{\lVert#1\rVert_2}

\begin{document}

\if1\blind
{
  \title{\bf On sure early selection of the best subset}
  \author{Ziwei Zhu, Shihao Wu \thanks{
    The authors gratefully acknowledge \textit{NSF DMS 2015366} for the support for this work.}\hspace{.2cm}\\
    Department of Statistics, University of Michigan\\}
  \maketitle
} \fi

\if0\blind
{
  \bigskip
  \bigskip
  \bigskip
  \begin{center}
    {\LARGE\bf On sure early selection of the best subset}
 \end{center}
  \medskip
} \fi




\begin{abstract}
The early solution path, which tracks the first few variables that enter the model of a selection procedure, is of profound importance to scientific discoveries. 
In practice, it is often statistically hopeless to identify all the important features with no false discovery, let alone the intimidating expense of experiments to test their significance.
Such realistic limitation calls for statistical guarantee for the early discoveries of a model selector.
In this paper, we focus on the early solution path of best subset selection (BSS), where the sparsity constraint is set to be lower than {the true sparsity}.
Under a sparse high-dimensional linear model, we establish the sufficient and (near) necessary condition for BSS to achieve sure early selection, or equivalently, zero false discovery throughout its early path.
Essentially, this condition boils down to a lower bound of the \emph{minimum projected signal margin} that characterizes the gap of the captured signal strength between sure selection models and those with spurious discoveries.
Defined through projection operators, this margin is independent of the restricted eigenvalues of the design, suggesting the robustness of BSS against collinearity.
Moreover, our model selection guarantee tolerates reasonable optimization error and thus applies to \emph{near} best subsets. 
Finally, to overcome the computational hurdle of BSS under high dimension,  we propose the ``screen then select'' (STS) strategy to reduce dimension for BSS. Our numerical experiments show that the resulting early path exhibits much lower false discovery rate (FDR) than LASSO, MCP and SCAD, especially in the presence of highly correlated design. We also investigate the early paths of the iterative hard thresholding algorithms, which are greedy computational surrogates for BSS, and which yield comparable FDR as our STS procedure. 
\end{abstract}

\noindent%


\section{Introduction}

High dimensional sparse linear models have been receiving intense theoretical investigation and widely applied in the big data era.
Suppose we have $n$ independent and identically distributed (i.i.d.) observations $\{(\bx_i,y_i)\}_{i=1}^n$ of $(\bx, Y)$ that follows the linear model:
\beq
\label{equ:linear model}
    Y = \bx^\top\bbeta^* + \epsilon. 
\eeq
Here $\bx$ is a design vector valued in $\RR ^ p$, $\bbeta^{\ast}\in\RR^p$ is an unknown sparse coefficient vector, and $\epsilon$ is random noise independent of $\bx$. Write $\bX = (\bx_1, \bx_2, \ldots, \bx_n) ^ \top$, $\by = (y_1, \ldots, y_n) ^ {\top}$ and $\bepsilon = (\epsilon_1, \ldots, \epsilon_n) ^ {\top}$. Then in matrix form, we have that 
\beq
\label{eq:lm_matrix}
\by = \bX\bbeta ^ * + \bepsilon. 
\eeq 

A central problem for high dimensional sparse linear models is the variable selection problem, i.e., to estimate the true support $\cS^*$ of $\bbeta ^ *$. Write $s ^ *  = |\cS ^ *|$. Given a target sparsity $s$, which is not necessarily equal to $s ^ *$, the best subset selection (BSS) solves 
\beq
	\label{eq:best_subset}
	\wh \bbeta^{\bss} (s):= \argmin_{\bbeta \in \RR^p, \|\bbeta\|_0 \le s} \sum_{i = 1} ^ n (y_i -\bx_i ^ \top \bbeta) ^ 2. 
\eeq
In words, BSS seeks for the size-$s$ subset of the available variables that achieves the minimum $L_2$ error of fitting $\by$. Another related type of subset regression approaches, which emerged in the 1970s, is the $\ell_0$-regularized approach, exemplified by Mallow's $C_p$ \cite{Mal73}, Akaike Information Criterion (AIC) \cite{akaike1974new, akaike1998information} and Bayesian Information Criterion (BIC) \citep{schwarz1978estimating}. Instead of directly constraining $\lzeronorm{\bbeta}$, these approaches penalize the loss function by a regularization term that is proportional to $\lzeronorm{\bbeta}$, which can be viewed as an indicator of model complexity.  
Nevertheless,  both the $\ell_0$-constrained and $\ell_0$-regularized methods are notorious for their NP-hardness \citep{natarajan1995sparse} and are thus computationally infeasible under high dimensions. 

For a long period of time, the computational barrier of BSS has been shifting attention away from its exact discrete formulation to its surrogate forms that are amenable to polynomial algorithms. The past three decades or so have witnessed a flurry of profound works on this respect, giving rise to a myriad of variable selection methods with both statistical accuracy and computational efficiency, particularly in high-dimensional regimes. A partial list of them include LASSO \citep{tibshirani1996regression, chen1998atomic}, SCAD \citep{fan2001variable,fan2004nonconcave, loh2013regularized, loh2017support, fan2018lamm}, elastic net \citep{zou2005regularization}, adaptive LASSO \citep{zou2006adaptive}, MCP \citep{zhang2010nearly}  and so forth. The shared spirit of these approaches is to substitute the $\ell_0$-regularization with a surrogate penalty of model complexity as follows: 
\begin{equation}\label{equ_penalizedsol}
	\wh{\bbeta}^{\text{pen}}:=\argmin_{\bbeta\in\RR^p}\sum_{i = 1} ^ n (y_i -\bx_i ^ \top \bbeta) ^ 2 + \rho_{\blambda}(\bbeta), 
\end{equation} 
where $\rho_{\blambda}(\bbeta)$, parameterized by $\blambda$, is a regularizer that encourages parsimony. All these methods are backed up by solid guarantee of model consistency. For example, \cite{zhao2006model} established the well-known irrepresentable conditions for model consistency of LASSO under fixed designs.
\cite{zhang2010nearly} showed that MCP achieves model consistency when the design satisfies a sparse Riesz condition and the minimum signal strength is not too weak.
\cite{fan2011nonconcave} proposed an iterative local adaptive majorize-minimization (I-LAMM) algorithm for general empirical risk minimization with folded concave penalty (e.g., SCAD) and showed that only a \emph{local} Riesz condition suffices to ensure model consistency. 
A recent work \cite{hazimeh2020fast} considered a hybrid of $L_0$ and $L_2$ regularization for problem \eqref{equ_penalizedsol}, which we refer to as L0L2, to pursue sparsity, robustness and computational efficiency. Specifically, they choose $\rho_{\lambda, \gamma}(\bbeta) = \lambda \|\bbeta\|_0 + \gamma \|\bbeta\|_2 ^ 2$ in \eqref{equ_penalizedsol}. They developed fast algorithms for this class of problems based on coordinate descent and local combinatorial search, which exhibited outstanding numerical performance among the state-of-the-art sparse learning algorithms in terms of prediction, estimation and exact recovery probability.

Unlike the previous works focusing on exact support recovery, this paper studies the selection behavior of BSS when the true sparsity is underestimated. It is motivated by the practical situations where exact model recovery is almost always hopeless, in particular in the presence of weak signals.  In Section \ref{sec:early_path}, we introduce the concept of the early solution path, through which we can assess the accuracy of the first few discoveries of a selection procedure. Our goal is to pursue sure early selection, meaning zero false early selection, which we believe as a realistic and desirable goal in modern practice. Besides, given that the exact BSS is often intractable, we accommodate certain optimization error in our statistical analysis. Section \ref{sec:algo_bss} reviews and proposes several approximate algorithms for BSS, which turn out to exhibit superior FDR-TPR tradeoff over penalized methods on their early solution paths in the numerical experiments. Finally, we summarize the major contributions of the paper in Section \ref{sec:contribution}.

\subsection{The early solution path}\label{sec:early_path}

For any two subsets $\cS_1, \cS_2 \subset [p]$, let $\cS_1 \backslash \cS_2$ denote $\cS_1 \cap \cS_2 ^ c$.
Then for any estimate $\wh\cS$ of the true model $\cS ^ *$, the false discovery proportion (FDP) and true positive proportion (TPP) are defined as
\beq
\label{equ:fdptpp}
\fdp(\wh\cS) := \frac{|\wh \cS \backslash \cS ^ {\ast}|}{\max(|\wh\cS|, 1)}\quad\text{and} \quad \text{TPP}(\wh\cS) := \frac{|\wh\cS\cap \cS^*|}{|\cS^*|}. 
\eeq
The false discovery rate (FDR) and true positive rate (TPR) are defined as the expectations of FDP and TPP respectively. In the context of multiple hypothesis testing, FDR and $1 - \mathrm{TPR}$ are essentially type I and type II errors respectively.

A solution path provides a comprehensive view of a model selection procedure: it displays tradeoff between type I and type II errors of the selected models as the regularization {parameter $\lambda$ in \eqref{equ_penalizedsol}} or the model size {$s$ in \eqref{eq:best_subset}} varies. In contrast, model consistency requires oracular knowledge of sparsity, which is often unavailable in practice. Moreover, model consistency is often too ambitious a goal to achieve in real-world problems, where the true sparsity or signal strength rarely satisfies the theoretical requirement. Therefore, the tradeoff between type I and type II error is inevitable, rendering the solution path a more meaningful evaluation criteria for comparing variable selectors. In particular, the early solution path, which tracks the first few selected variables, is of interest to scientific research; after all, it is impossible in practice to assess the causality of too many variables by experiments. Therefore, it is imperative to provide statistical guarantee, say FDR, for the early solution path to guide the subsequent scientific effort.

Formally, define the early solution path of BSS as the set $\bigl\{\wh\bbeta^{\bss}(s)\bigr\}_{s \le  s ^ *}$, i.e., the BSS estimators whose sparsity is not greater than $s ^ *$. In this paper, we explicitly characterize when BSS achieves zero false discovery throughout the early solution path, which we refer to as \emph{sure early selection}.
Regarding related works on solution paths, 
\cite{su2017false} showed that under the regime of linear sparsity, i.e., $s ^ * / p$ tends to a constant, even when the features are independent, false discoveries occur early on the LASSO path with high probability, regardless of the signal strength.
\cite{wang2020complete} further provided a complete FDR-TPR tradeoff diagram of LASSO.  \cite{su2018first} investigated when the first false variable is selected by sequential regression procedures, which include forward stepwise, the LASSO, and least angle regression. Su's setup shares similar flavor with this paper, while the theoretical results therein are based on i.i.d Gaussian design. We instead work with fixed designs with possible collinearity. More details are in Section \ref{sec_ss}.

\subsection{Algorithmic development for BSS}
\label{sec:algo_bss}


Recent advancement in computing hardware and optimization algorithms has made possible the implementation of BSS for real-world problems. 
\cite{bertsimas2016best} recast the BSS problem \eqref{eq:best_subset} as a mixed integer optimization (MIO) problem and showed that for $n$, $p$ in thousands, a MIO algorithm implemented on optimization softwares such as Gurobi can achieve certified optimality within minutes. \cite{bertsimas2020sparse} devised a new cutting plane method that solves to provable optimality the Tikhonov-regularized \cite{Tik43} BSS problem with $n, p$ in the $100,000$s. A more recent work \cite{ZWZ20} proposed an iterative splicing method called ABESS, short for adaptive best subset selection, to solve the BSS problem. They showed that ABESS enjoys both statistical accuracy and polynomial computational complexity when the design satisfies the sparse Riesz conditon and the minimum signal strength is of order $\Omega\{(s^* \log p \log \log n / n) ^ {1 / 2}\}$. 

Regarding the statistical performance in numerical experiments, \cite{bertsimas2016best} and \cite{bertsimas2020sparse} demonstrated that BSS enjoys higher predictive power and lower false discovery rate (FDR) than LASSO. \cite{ZWZ20} presented similar numerical results of ABESS and also showed that ABESS is able to estimate the model sparsity more accurately than LASSO, MCP and SCAD.  \cite{hastie2017extended} conducted extensive numerical experiments on comparison between LASSO, relaxed LASSO and BSS. They covered a wider lower range of signal-to-noise ratios (SNR) than \cite{bertsimas2016best}. The main message therein is that regarding prediction accuracy, LASSO outperforms BSS in the low SNR regime, while the situation is reversed in the high SNR regime. Furthermore, relaxed LASSO \citep{Mei07} is the overall winner, performing similarly or outperforming both LASSO and BSS in nearly all the cases in \cite{hastie2017extended}.

\begin{figure}
    \centering
	\includegraphics[width=7cm]{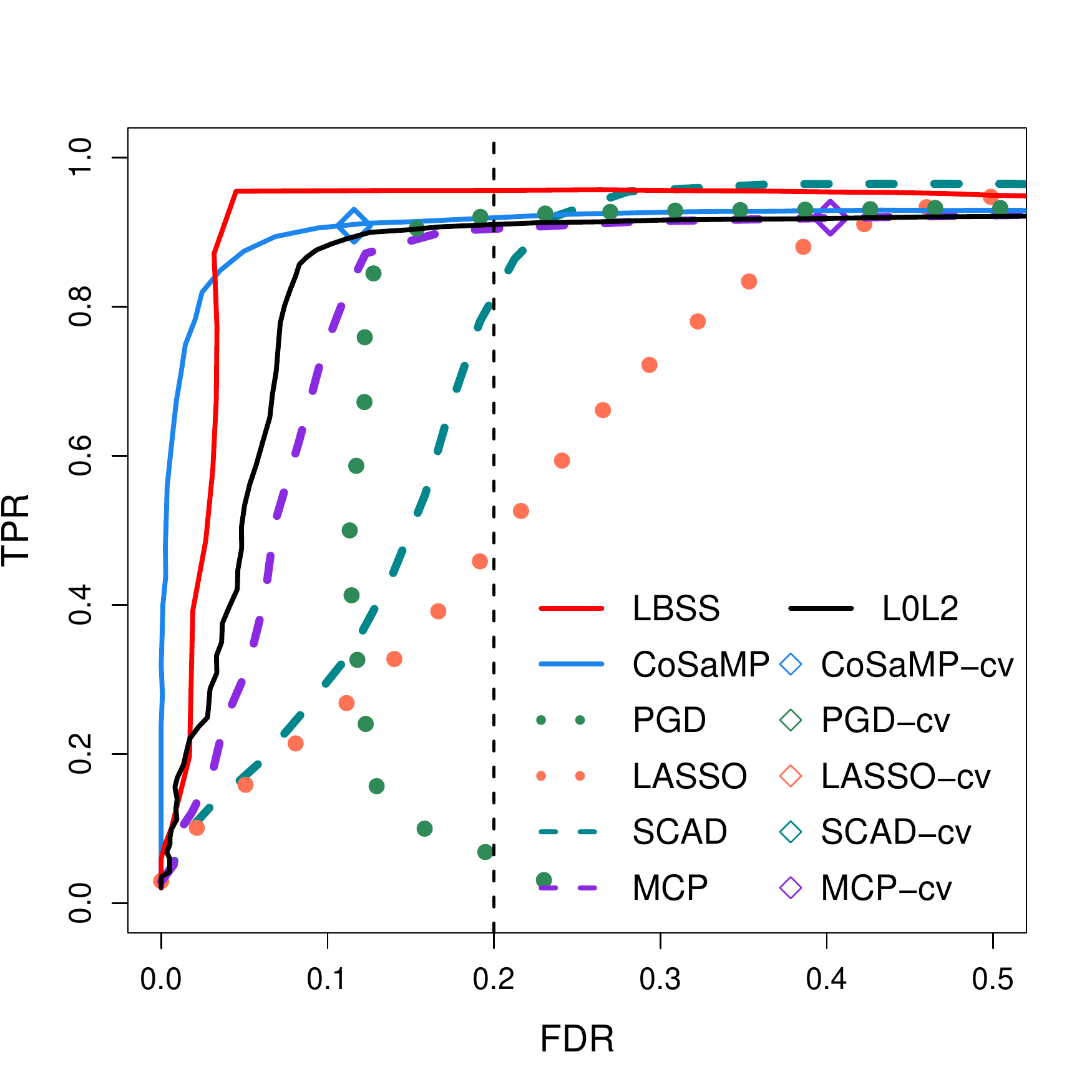}
	\caption{\footnotesize{The FDR-TPR paths of LBSS (red), CoSaMP (blue), L0L2 (black), PGD (dotted green), LASSO (orange), SCAD (dashed green) and MCP (violet) under autoregressive design ($\bx \sim \cN(\textbf{0}, \bSigma)$ with $\Sigma_{ij}=0.8^{|i-j|}$ for $i, j \in [p]$). This is one setup ($\rho = 0.8$ and $\sigma = 0.5$) in Figure \ref{fig:ar}. 
		The FDR and TPR of each point on the solution path is based on averaging FDP and TPP of the solutions over $100$ independent Monte Carlo experiments. More implementation details are in Section \ref{sec:paths}.}}
	\label{fig_example}
\end{figure}

In this paper, we 
introduce a ``screen then select'' (STS) strategy to approximately solve BSS under high-dimensional sparse models and achieve sure early selection. Specifically, we first apply a computationally cheap model selector, say LASSO, to filter out massive useless variables. Then we run a BSS algorithm, say MIO, on the remaining variables to select the final model. The STS strategy is designed to accommodate the computational hardness of the BSS algorithms by first reducing the dimension for them. One particular STS-type algorithm that we investigate is LBSS (LASSO plus BSS), which uses LASSO to screen variables and MIO to select the model (more details are in Section \ref{sec:lbss}). To give some flavor of the statistical performance of LBSS, 
Figure \ref{fig_example} compares the FDR-TPR curves of the solution paths of LBSS, MCP, SCAD, LASSO, PGD and CoSaMP \citep{NTr09} under autoregressive design. There the $x$-axis represents FDR, and the y-axis represents TPR. Note that a perfect model selection procedure yields a ``$\Gamma$''-shaped FDR-TPR path, meaning that it selects no false variable until enforced to select more than $s ^ *$ variables. One can see from Figure \ref{fig_example} that the FDR-TPR paths of LBSS and CoSaMP are the closest to the perfect ``$\Gamma$''-shape, suggesting their superiority over the competing approaches in terms of FDR-TPR tradeoff. More comparisons of this type can be found in Section \ref{sec:paths}, and the overall winner is LBSS.


\subsection{Major contributions}
\label{sec:contribution}
Below we summarize the major contributions of our work: 
\begin{itemize}
	\item[(1)] In Section \ref{sec_ss_suf}, we establish a sufficient condition for near best subsets to achieve sure early selection for any target sparsity $s \le s ^ *$ based on a quantity called \emph{minimum projected signal margin}. This margin 
	is insensitive to high collinearity of the design matrix and even accommodates degenerate covariance structure of the design. 
	
	\item[(2)] In Section \ref{sec_ss_nec}, we show that throughout the early path, i.e., for any $s < s ^ *$, the sufficient condition above is nearly necessary (up to a universal constant) for BSS to achieve sure early selection.
	
	\item[(3)] In Section \ref{sec_margin}, we explicitly derive a non-asymptotic bound of the minimum projected signal margin under random design to allow for verification of the established sufficient and necessary conditions. 
	
	\item[(4)] In Section \ref{sec:lbss}, we introduce the ``screen then select'' (STS) strategy to  efficiently solve 
	BSS. We show that BSS within a sure screening set can achive sure early selection. Our numerical experiments show that the early solution path of the STS strategy enjoys superior FDR-TPR tradeoff over competing approaches. 
\end{itemize}
The proof for all the theorems and the technical lemmas are collected in the appendix.

\subsection{Notation}
We use regular letters, bold regular letters and bold capital letters to denote scalars, vectors and matrices respectively. Given any two sequences $\{a_n\}, \{b_n\}$ valued in $\RR$, we say $a_n \lesssim b_n$ or $b_n\gtrsim a_n$ if there exists a universal constant $C > 0$ such that $a_n \le Cb_n$. We say $a_n\asymp b_n $ if $a_n\lesssim b_n$ and $a_n\gtrsim b_n$.  For any positive integer $a$, we use $[a]$ to denote the set $\{1, 2, \ldots, a\}$. For any vector $\ba$ and matrix $\bA$, we use $\ba^{\top}$ and $\bA ^ {\top}$ to denote the transpose of $\ba$ and $\bA$ respectively. For any matrix $\bX \in \RR ^{n \times p}$, define the projection operator for its column space 
\[
	\bP_{\bX} := \bX(\bX ^ {\top} \bX) ^+\bX ^ {\top}, 
\]
where the superscript symbol $^+$ denotes the Moore-Penrose inverse. We use $\bX_{\cS}$ to denote the submatrix of $\bX$ with columns indexed by $\cS \subset [p]$. 
For any random variable $X$ valued in $\RR$, define
\[
\label{eq:psi_norm}
\|X\|_{\psi_2} := \inf\{t > 0: \E\exp(X ^ 2 / t ^ 2) \le 2\}~\text{and}~\|X\|_{\psi_1} := \inf\{t > 0: \E\exp(|X| / t) \le 2\}.
\]
For any random vector $\bx$ valued in $\RR ^ p$, define 
\[
\label{eq:psi_norm_vec}
\|\bx\|_{\psi_2} := \sup_{\ltwonorm{\bv} = 1}\|\bv ^{\top} \bx\|_{\psi_2}~\text{and}~\|\bx\|_{\psi_1} := \sup_{\ltwonorm{\bv} = 1}\|\bv ^{\top} \bx\|_{\psi_1}.
\]
For any estimator $\wh\bbeta$ of $\bbeta ^ *$, based on \eqref{equ:fdptpp}, we abuse the notation to define 
\[
	\fdp(\wh\bbeta) := \frac{|\mathrm{supp}(\wh\bbeta) \backslash \cS ^ {\ast}|}{\max(|\supp(\wh \bbeta)|, 1)} \quad\text{and} \quad \mathrm{TPP}(\wh\bbeta) := \frac{|\supp(\wh \bbeta)\cap \cS^*|}{|\cS^*|}.
\]

\section{On the early solution path of best subset selection}\label{sec_ss} 

Consider $s \le s ^ *$ in \eqref{eq:best_subset}.
Our focus in this section is to establish sufficient and necessary conditions for BSS to achieve \emph{sure early selection}, i.e., $\fdp\{\wh \bbeta ^{\bss}(s)\} = 0$.  To start with, we introduce a measure called the \emph{minimum projected signal margin} to characterize the fundamental difficulty for BSS to achieve sure early selection. 
Define 
\[
	\AA(s):=\{\mS: \mS\subset[p], |\mS|=s, \mS\setminus \mS^{\ast}\ne \emptyset \}~\text{and}~\AA^{\ast}(s) := \{\mS: \mS\subset\mS^{\ast}, |\mS|=s \}. 
\]
In words, $\AA(s)$ is the set of all size-$s$ models with at least one false variable, while $\AA ^ {\ast}(s)$ is the set of those without any false variable. We further define $\AA_t(s) := \{\cS \in \AA(s): |\cS \backslash \cS ^ *| = t\}$ for $t \in [s]$, which collects all the models in $\AA(s)$ with exactly $t$ false variables.  
For any $\cS_1, \cS_2 \subset [p]$, define the marginal projection operator 
\[
	\bP_{\cS_1 | \cS_2} := \bP_{\bX_{\cS_1}} - \bP_{\bX_{\cS_1 \cap \cS_2}}.
\] 
Note that since $\cS_1 \cap \cS_2 \subset \cS_1$, $\bP_{\cS_1 | \cS_2}$ is a projection operator: it captures the directions in the column space of $\bX_{\cS_1}$ that are orthogonal to the column space of $\bX_{\cS_1 \cap \cS_2}$, representing the margin of fitting power of $\bX_{\cS_1}$ on top of $\bX_{\cS_1 \cap \cS_2}$. 
Writing $\bmu^{\ast} = \bX\bbeta^{\ast}$, 
we define the \emph{optimal feature swap} $\Phi: \AA(s) \to \AA^ *(s)$ as
\begin{equation}
	\Phi(\mS) := \argmax_{\mS^\ddagger\in\AA^{\ast}(s), \cS\cap\cS^*\subset \cS^\ddagger }~\ltwonorm{\bP_{\bX_{\cS^\ddagger}} \bmu ^ *} ^ 2. 
\end{equation}
As the name indicates, $\Phi$ swaps the false variables in $\cS$ for the true ones that maximize the fitting power. 

We are now in position to define the projected signal margin of $\cS \in \AA(s)$ as 
\beq
	\label{eq:psm_def}
	\psm(\cS) := \frac{\ltwonorm{\bP_{\Phi(\cS)|\cS}\bmu^*}  - \ltwonorm{\bP_{\cS|\Phi(\cS)} \bmu^*} }{|\Phi(\cS)\setminus\cS|^{1/2}}
\eeq
and its minimum over $\AA(s)$ as
\[
\psm_*(s) := \min_{\mS\in\AA(s)} \psm(\cS).
\]
Intuitively, $\psm_*(s)$ represents the minimum gain of explanation power from the optimal feature swap, normalized by the square root of the number of false variables. The smaller $\psm_*(s)$, the easier for some $\cS \in \AA(s)$ to outperform $\Phi(\cS)$ in terms of goodness of fit, giving rise to potential false discoveries. Note that if $\beta^ *_j = b, \forall j \in [p]$ and $n ^ {-1} \bX^{\top}\bX = \bI_p$, i.e., homogeneous signal and orthonormal design, we have $\psm(\cS) = bn ^ {1 / 2}$ for any $\cS \in \AA(s)$, implying that $\psm_*(s) = bn ^ {1 / 2}$. Regarding general cases, we refer the readers to Section \ref{sec_margin}, where we establish non-asymptotic bounds for projected signal margin under random design. 

In the sequel, we will see that $\psm_*(s)$ dictates the statistical behavior of the early path of BSS: the sufficient and necessary condition for BSS to achieve sure early selection essentially boils down to a lower bound of $\psm_*(s)$. Throughout this section, we assume that $s^{\ast} < n$ and that $\|\epsilon\|_{\psi_2} \le \sigma$ in \eqref{equ:linear model}.

\subsection{Sufficient conditions}\label{sec_ss_suf}

The goal of this subsection is to explicitly characterize a sufficient condition for $\wh\bbeta ^ {\bss}(s)$ to achieve zero false discovery.  Given the optimization challenge of obtaining $\wh\bbeta ^ {\bss}(s)$ exactly, we extend the scope of our statistical analysis to embrace all  \emph{near} best $s$-subsets, i.e., the size-$s$ subsets that yield comparable goodness of fit as the best subset. Specifically, for any $\cS \subset [p]$, let $\cL_{\cS} = \by^{\top}(\bI - \bP_{\bX_{\cS}})\by$ and $\cL_* = \min_{\cS \subset [p], |\cS| = s} \cL_\cS$. Given any tolerance level $\eta\in[0,1)$, consider the following collection of near best $s$-subsets:
\beq
	\label{eq:near_best}
	\mathbb{S}(s,\eta) := \{\cS: |\cS| = s, \cL_{\mS} \le \cL_* + \eta \psm_*^2(s) \}. 
\eeq
Here $\eta$ is a tolerance threshold that determines the condition for $\cS$ to be a near best $s$-subset: the larger $\eta$, the higher fitting error we can tolerate for a near best subset. Below we introduce a sufficient condition for all sets in $\SS(s,\eta)$ to achieve sure selection.  
\begin{theorem}
	\label{thm1}
	Suppose that $\log p \gtrsim s ^ *$. There exists a universal constant $C \ge 1 $, such that for any $\xi>C$ and $0 \le \eta < 1$, whenever
	\beq
		\label{eq:psm_suff}
		\psm_*(s) \ge \frac{8 \xi \sigma (\log p)^{1/2}}{1-\eta}, 
	\eeq	
	we have that 
	\beq
	       \label{ineq:thm1_result}
		\PP\Bigl\{\fdp(\wh \cS) = 0,~\forall \wh\mS \in \mathbb{S}(s,\eta) \Bigr\} \ge 1-Csp^{-2(C^{-2}\xi^2-1)}.
	\eeq
\end{theorem}

Theorem \ref{thm1} says that whenever the minimum projected signal margin $\psm_*(s)$ satisfies the lower bound \eqref{eq:psm_suff}, all the sets in $\SS(s,\eta)$ are sure selection sets with high probability. Consequently, we have $\tpp(\wh{\cS})=s/s^*$ for all  $\wh{\cS}\in\SS(s,\eta)$ with high probability. As the target sparsity $s$ increases and gets closer to the true sparsity $s^*$, the $\tpp$ guarantee grows higher and finally reaches $1$. Regarding  $\eta$, note that as per \eqref{eq:psm_suff}, the larger $\eta$ is, that is, the more relaxed requirement we impose on near best $s$-subsets, the higher $\psm_*(s)$ we need to guarantee sure selection; this is natural given that $\SS(s,\eta)$ grows as $\eta$ increases. Regarding $\xi$, by assessing its role on the right hand side of \eqref{eq:psm_suff} and \eqref{ineq:thm1_result} respectively, we can deduce that larger minimum projected signal margin leads to higher confidence in sure selection. 

One important message of this theorem is that the covariance structure of the design is non-essential in determining the early path of BSS. Through the marginal projection operators, Theorem \ref{thm1} unveils that it is the gap in capacity of capturing the true signal that determines if any false variable enters the early path of BSS. Given that projection matrices of the features are invariant with respect to linear transformations of them, the minimum projected signal margin is \emph{robust against collinearity}. For instance, given $\ba, \bb \in \RR ^ n$, consider a simple yet illustrative case where $\bX_{\cS ^ *} = \ba \bone^{\top}_{s ^ *}$ and  $\bX_{{\cS ^ *} ^ c} = \bb \bone_{p - s ^ *} ^ {\top}$, i.e, all the columns of $\bX_{\cS ^ *}$ equal $\ba$, and all the columns of $\bX_{{\cS ^ *} ^ c}$ equal $\bb$. This case violates both the irrepresentable condition and the sparse Riesz condition but can satisfy the minimum projected signal margin condition. By applying Theorem \ref{thm1}, one can deduce that as long as $\ba$ and $\bb$ are not too correlated and the signal is not too weak, BSS achieves sure early selection with high probability. Section \ref{sec:paths} presents a wide range of numerical results to demonstrate the robustness of BSS against design correlation.

{Now we discuss the difference between our work and a recent related work \cite{guo2020best} that studies the model consistency of BSS. The main theoretical innovation of our work given \cite{guo2020best} lies in defining the projected signal margin through the optimal feature swap. \cite{guo2020best} proposed the following ``identifiability margin'' to characterize the fundamental difficulty for BSS to achieve model consistency: 
\[
    \begin{aligned}
	    \tau_*(s ^ *) & := \min_{\cS \subset [p], |\cS| = s ^ *, \cS \neq \cS ^ *} \frac{{\bbeta ^ *} ^ {\top}_{\cS^*\setminus\cS}\bigl(\wh\bSigma_{\cS^*\setminus\cS, \cS^*\setminus\cS}-\wh\bSigma_{\cS^*\setminus\cS, \cS}\wh\bSigma_{\cS \cS}^{-1}\wh\bSigma_{\cS, \cS^*\setminus\cS}\bigr) \bbeta_{\cS ^ *\setminus \cS} ^ *}{|\cS ^ * \setminus \cS|} \\
	    & = \min_{\cS \subset [p], |\cS| = s ^ *, \cS \neq \cS ^ *} \frac{\ltwonorm{\bmu ^ *} ^ 2 - \ltwonorm{\bP_{\cS} \bmu ^ *} ^ 2}{n|\cS ^ * \setminus \cS|}. 
	\end{aligned}
\]
The most essential difference between the identifiability margin and our projected signal margin is that for any set $\cS$ with false discoveries, the former compares the projected signal strength of $\cS$ with that of $\cS ^ *$, while the latter compares $\cS$ with its optimal swap set $\Phi(\cS)$. Note that to ensure BSS to recover $\cS ^ *$ with known sparsity $s ^ *$, naturally it suffices to impose a large gap of signal strength between $\cS ^ *$ and the rest of the models of size $s ^ *$. However, on the early path of BSS, there are multiple ($s ^ * \choose s$, more precisely) sure selection sets of size $s$, and it is non-trivial to identify which one to refer to to characterize the fundamental difficulty of achieving sure selection. Instead of fixing one sure selection set as the reference, we propose the optimal feature swap to pick the most \emph{similar and powerful} sure selection set $\Phi(\cS)$ for any set $\cS$ and compute their projected signal strength margin. This turns out to be the measure of fundamental difficulty for BSS to attain sure early selection (see Theorems \ref{thm2} and \ref{thm3}).
}


Another difference between the two margins $\tau_*$ and $m_*$ is that the numerator of $m_*$ is the difference between the $\ell_2$-norms of projected signals instead of the \emph{squared} $\ell_2$-norms of them in $\tau_*$. 
Because of this difference, when $s = s ^ *$, $m_*$ is narrower than $\tau_*$ after normalization: 
\[
	\begin{aligned}
		m ^ 2_*(s ^ *) & = \min_{\cS \in \AA(s ^ *)} \frac{\bigl(\ltwonorm{\bP_{\cS ^ * | \cS}\bmu ^ *} - \ltwonorm{\bP_{\cS | \cS ^ *}\bmu ^ *}\bigr) ^ 2}{|\cS \backslash \cS ^ *|} \le \min_{\cS \in \AA(s ^ *)} \frac{\ltwonorm{\bP_{\cS ^ * | \cS}\bmu ^ *} ^ 2 - \ltwonorm{\bP_{\cS | \cS ^ *}\bmu ^ *} ^ 2}{|\cS \backslash \cS ^ *|} \\
		& = \min_{\cS \in \AA(s ^ *)}\frac{\ltwonorm{\bmu ^ *} ^ 2 - \ltwonorm{\bP_{\cS} \bmu ^ *} ^ 2}{|\cS \backslash \cS ^ *|} = n\tau_*(s ^ *), 
	\end{aligned}
\]
where the inequality is due to the fact that $\ltwonorm{\bP_{\cS ^ * | \cS} \bmu ^ *} \ge \ltwonorm{\bP_{\cS | \cS ^ *} \bmu ^ *}$. Therefore, $m_*(s ^ *) \gtrsim \sigma (\log p) ^ {1 / 2}$ as imposed in \eqref{eq:psm_suff} is a stronger condition than $\tau_*(s ^ *) \gtrsim \sigma ^ 2\log p / n$, which was shown by \cite{guo2020best} as a sufficient and near necessary condition for BSS to recover the true support. As we shall see in Section \ref{sec_ss_nec}, this stronger form of margin is not due to technical limitation but rather the greater difficulty of achieving sure early selection than model consistency. The early solution path underestimates the true sparsity $s ^ *$ and thus requires wider margin of signal strength to compensate for model misspecification and identify true variables. 


Now we discuss a series of literature on the solution path of variable selectors under high dimension:  \cite{su2017false}, {\cite{su2018first} and \cite{wang2020complete}. \cite{su2017false} studied the early solution path of LASSO. Define the LASSO estimator $\wh\bbeta ^ {\mathrm{lasso}}(\lambda)$ as
\[
	\wh\bbeta ^ {\mathrm{lasso}}(\lambda) := \argmin_{\bbeta \in \RR ^ p} \frac{1}{2}\ltwonorm{\by - \bX\bbeta} ^ 2 + \lambda \lonenorm{\bbeta}. 
\]
Theorem 2.1 of \cite{su2017false} says that regardless of $\sigma > 0$, for any arbitrary small constants $\lambda_0 > 0$ and $\eta > 0$,  as $n, s ^ *, p \to \infty$ with $n / p$ and $s ^ * / p$ tend to constants, the probability of the event
\[
	\bigcap_{\lambda \ge \lambda_0} \biggl\{\fdp\Bigl(\wh\bbeta ^ {\mathrm{lasso}}(\lambda)\Bigr) \ge q ^ *\Bigl(\mathrm{TPP}(\wh\bbeta ^ {\mathrm{lasso}}(\lambda))\Bigr) - \eta\biggr\}
\]
tends to one, where $q ^ *$ is a strictly increasing function with $q ^ *(0) = 0$. We refer the readers to \cite{su2017false} for specific definition and visual illustration of $q ^ *$. Simply speaking, this implies that false discoveries occur early on the solution path of LASSO however strong the signal is. While it is tempting to conclude theoretical superiority of BSS over LASSO on the early solution path, we emphasize that our Theorem \ref{thm1} is actually not directly comparable with Theorem 2.1 of \cite{su2017false}, given that we mainly focus on the ultra-high dimensional setup, where $\log p \gtrsim s ^ *$, while \cite{su2017false} focused on linear sparsity regimes. Nevertheless, the simulation experiments in Section \ref{sec:paths} clearly demonstrate the numerical superiority of BSS over LASSO in terms of FDR and TPR on their early solution paths. A follow-up work
\citep{wang2020complete} characterized the complete LASSO tradeoff diagram, 
which shows not only a lower bound on $\fdr$ but also an upper bound on TPR. This further explores the limitation of LASSO in terms of FDR-TPR tradeoff.  Finally, \cite{su2017false} considered the $L_0$-penalized estimator:
\[
   \wh{\bbeta}_{L_0} = \argmin_{\bbeta\in\RR^p}\| \by - \bX \bbeta\|_2^2 + \lambda 
   \|\bbeta\|_0
\]
and shows that under the linear sparsity regime with the signal strength going to infinity, $\wh{\bbeta}_{L_0}$ can achieve $(\fdp,\tpp)=(0,1)$ asymptotically with a proper choice of $\lambda$. Note that this result revolves around exact model recovery and does not characterize the early solution path. It thus remains open how the early solution path of BSS behaves under the linear sparsity setup and how that is compared with LASSO. 
  
\cite{su2018first} considered the rank of the first false variable selected by sequential regression procedures including forward stepwise regression, LASSO, and the least angle regression. 
It assumes that $\bX\in\RR^{n\times p}$ has independent $\cN(0,1/n)$ entries, $\bepsilon\sim\cN(0,\sigma^2\bI)$ and nonzero components in $\bbeta^*$ are equal to some $M\ne 0$. It focuses on the regime of near \emph{linear sparsity} in the sense that $c_1p/(\log p)^{c_2} \le n \le c_3 p$ and that $c_4n\le s^* \le \min\{0.99p, c_5n\log^{0.99}p \}$ for arbitrary positive constants $\{c_i\}_{i \in [5]}$.   Let $T$ denote the rank of the first false variable selected by any of the aforementioned three sequential regression methods. \cite{su2018first} showed that as long as $\sigma/M\to 0$, $\log T = \{1 + o_{\PP}(1)\}\bigl[ \{2n(\log p)/s^*\} ^ {1 / 2}  - n/(2s^*) + \log \{n/(2p\log p)\}\bigr]$. When $n \asymp p \asymp s^*$, it means that each of the sequential methods includes the first false variable after no more than $O\bigl( \exp\{ (\log s^*)^{1/2} \} \bigr)$ steps asymptotically, thereby failing to achieve sure early selection. {In our work, we show that BSS can achieve sure early solution with high probability whenever the projected signal margin exceeds the lower bound as per \eqref{eq:psm_suff}. Although we consider the ultra-high dimensional regime, which is far from the linear sparsity regime and makes our results not directly comparable with the results in \cite{su2018first}, we emphasize that our theory does not require statistical independence between the features and accommodates nearly degenerate designs, while the results in \cite{su2018first} rely on independent Gaussian designs. }}

Beyond the guarantee for absolutely sure selection as in Theorem \ref{thm1}, we can extend our argument to provide any pre-specified  level of FDP guarantee through adjusting the definition of the projected signal margin. Given any $0< q < 1$, consider the following variant of the minimum projected signal margin that corresponds to $q$-level FDP: 
\[ 
	\psm_q(s) := \min_{\cS\in\cup_{t \ge \lceil qs\rceil}\AA_t(s)} \psm(\cS),
\]
as well as the corresponding collection of near best $s$-subsets: 
\[
	\mathbb{S}_q(s,\eta) := \{\cS: |\cS| = s, \cL_{\mS} \le \cL_* + \eta \psm_q^2(s) \}. 
\]
Now we are ready to present the following corollary on $q$-level FDP guarantee. 

\begin{cor}
	\label{cor1}
	Suppose that $\log p \gtrsim s ^ *$. There exists a universal constant $C \ge 1$, such that for any $\xi > C$ and $0\le\eta<1$, if
	\beq
	\label{eq:q_psm}
	\psm_q(s) \ge \frac{8 \xi \sigma (\log p)^{1/2}}{1-\eta}, 
	\eeq	
	then we obtain 
	\beq
		\PP\Bigl\{\fdp(\wh \cS) \le q,~\forall \wh\mS \in \mathbb{S}_q(s,\eta) \Bigr\} \ge 1 - Csp^{-2(C^{-2}\xi^2-1)}. 
	\eeq
\end{cor}
{Corollary \ref{cor1} is a variant of Theorem \ref{thm1} by replacing $\AA(s)$ in Theorem \ref{thm1} with $\cup_{t \ge \lceil qs\rceil}\AA_t(s)$.}

\subsection{Necessary conditions}\label{sec_ss_nec}

In this section, we show that the lower bound \eqref{eq:psm_suff} of the minimum projected signal margin is almost necessary (up to a universal constant) for BSS to achieve sure early selection in general setups (without the assumption that $\log p \gtrsim s^*$). We investigate the necessity in two scenarios seperately: (i) $s = 1$; (ii) $s \ge 2$. 

\subsubsection{Sure first selection} 
\label{sec:sure_first}
We start with the simplest case: $s=1$, i.e., we aim to find only one true predictor. 
Let
\[
    j^{\dagger} := \argmax_{j\in\cS^*}\|\bP_{\bX_j} \bmu ^ *\|_2^2, 
\]
which denotes the best variable in $\cS^*$ in terms of fitting the signal $\bmu^*$. Then for any $j\in(\cS ^*)^c$, $\Phi(\{j\})=\{j^\dagger\}$ 
 and 
\[
	\psm(\{j\}) = \|\bP_{\bX_{j^{\dagger}}}\bmu^*\|_2- \|\bP_{\bX_j}\bmu^*\|_2. 
\]
We say a set $\cP$ is a $\delta$-packing set if and only if for any $\bu, \bv \in \cP$, $\ltwonorm{\bu - \bv} \ge \delta$. Consider the following two assumptions on the problem setup:
\begin{assume}
       \label{ass:packing1} Let $\bar{\bu}_j=\bX_j/\|\bX_j\|_2$ for $j\in[p]\setminus \cS^*$. For some $0<\delta_0<1$, there exists an  index set $\cJ_{\delta_0} \subset (\cS ^ *) ^ c$ of delusive features satisfying:
       \begin{enumerate}
       		\item[(i)] $\{\bar{\bu}_j\}_{j\in\cJ_{\delta_0}}$ is a $\delta_0$-packing set;
       		\item[(ii)] $|\cJ_{\delta_0}| \ge p ^ {c_{\delta_0}}$ for some $0 < c_{\delta_0} < 1$ depending on $\delta_0$ such that $\delta_0 ^ 2c_{\delta_0} \log p > 1$;        		
       		\item[(iii)] Each feature in $\cJ_{\delta_0}$ has small projected signal margin, i.e., 
       		\beq 
       		\label{eq:psm_nec1}
       		0 < 
		\psm(\{j\}) < \frac{\delta_0\sigma( c_{\delta_0}\log p) ^ {1 / 2}}{20}, \forall j \in \cJ_{\delta_0}. 
       		\eeq
       	\end{enumerate}
\end{assume}
\begin{assume}
       \label{ass:best_var} 
       There exists a universal constant $\xi > 8$ such that
       \[
       		\|\bP_{\bX_{j^{\dagger}}}\bmu^*\|_2 \ge \xi\delta_0\sigma (c_{\delta_0}\log p) ^ {1 / 2}.
       \]
\end{assume}
The main idea of Assumption \ref{ass:packing1} is that we can find sufficiently many delusive features with small projected signal margin as per \eqref{eq:psm_nec1}. Furthermore, perceiving Euclidean distance between two vectors within $\{\bar \bu_{j}\}_{j \in [p]}$  as the correlation proxy between the corresponding features, condition (ii) requires the delusive variables to be dissimilar to each other, ensuring that the impact of the dimension is nontrivial. Assumption \ref{ass:best_var} says that the best single variable has sufficient explanation power for the true signal $\bmu ^ *$. This assumption is standard: one can deduce from the well-known $\beta$-$\min$ condition that $\sigma (\log p) ^ {1 / 2}$ is the minimum signal strength we need for $\ltwonorm{\bP_{\bX_{j ^ \dagger}} \bmu ^ *}$ to identify $j ^ {\dagger}$ as a true predictor.

Now we present our theorem on the necessary condition for the very first selection of BSS to be true. 

\begin{theorem}
	\label{thm2}
	Suppose that $\{\epsilon_i \}_{i\in[n]} $ are independent Gaussian noise with variance $\sigma ^ 2$. Under Assumptions \ref{ass:packing1} and \ref{ass:best_var}, we have that
	\[
		\PP\biggl(\min_{j \in \cJ_{\delta_0}} \cL_{\{j\}} < \min_{j ^ * \in \cS ^ *}\cL_{\{ j ^ {*}\}}\biggr) \ge 1 - 4s^ *p ^ {-(c_1\xi \delta ^ 2_0 - 1)c_{\delta_0}} - s ^ *p ^ {- c_1\delta_0 ^ 2 c_{\delta_0}}, 
	\]
	where $c_1$ is a universal constant. 
\end{theorem}

Given Theorem \ref{thm2}, one can deduce by comparing \eqref{eq:psm_suff} and \eqref{eq:psm_nec1} that the lower bound of the minimum projected signal margin in Theorem \ref{thm1} is almost tight: whenever there are sufficiently many spurious features whose projected signal margins violate this lower bound up to a multiplicative constant, the first shot of BSS is false with high probability. 

\subsubsection{Sure early multi-selection}
Now we move on to the general case when $s\ge 2$. 
Define $\mS ^ {\dagger}(s)$ to be the best size-$s$ subset of $\cS ^ *$ in terms of fitting the signal $\bmu ^ *$, i.e., $\mS^{\dagger}(s) := \argmax_{\mS\in\AA^{\ast}(s)}~\ltwonorm{\bP_{\bX_\cS} \bmu ^ *} ^ 2. $
For simplicity, we write $\mS^{\dagger}(s)$ as $\mS^\dagger$ in the sequel. 
Consider the following two assumptions on the problem setup: 
\begin{assume}
        \label{ass:packing2} There exist $j_0 \in \cS ^ \dagger$ and $0 < \delta_0 < 1$ such that if we let $\cS ^ {\dagger}_{0} = \cS ^{\dagger} \backslash \{j_0\}$, $\wt\bu_j:=(\bI-\bP_{\bX_{\mS_0^{\dagger}}})\bX_j$ and $\bar{\bu}_j=\wt\bu_j/ \|\wt\bu_j\|_2$ for $j\in[p]\setminus \mS^{\ast}$, we can find an index set $\cJ_{\delta_0} \subset (\cS ^ *) ^ c$ of delusive features satisfying:
       \begin{enumerate}
       		\item[(i)] $\{\bar{\bu}_j\}_{j\in\cJ_{\delta_0}}$ is a $\delta_0$-packing set;
       		\item[(ii)] $|\cJ_{\delta_0}| \ge p ^ {c_{\delta_0}}$ for some $0 < c_{\delta_0} < 1$ depending on $\delta_0$; 
       		\item[(iii)] Consider all the models of size $s$ that are formed by replacing $j_0$ in $\cS ^ {\dagger}$ with a spurious variable in $\cJ_{\delta_0}$, namely, $\AA_{j_0}:= \bigl\{\cS ^ {\dagger}_0 \cup \{j\}: j \in \cJ_{\delta_0}\bigr\}$. Then each set in $\AA_{j_0}$ has small projected signal margin, i.e.,   
	        \beq
		\label{eq:psm_nec2}
		0 <\psm(\cS)< \frac{\delta_0\sigma(c_{\delta_0} \log p) ^ {1 / 2}}{20}, \forall \cS \in \AA_{j_0}. 
	        \eeq
	\end{enumerate}
\end{assume}

\begin{assume}
	\label{ass:residual_margin_s}	
	There exists a universal constant $ \xi >\max\{ 8 / (\delta_0c^{1 / 2}_{\delta_0} \log p), 2 \}$ such that
	\[
		 \min_{\mS\in\AA ^ *(s), \mS \neq \cS ^ {\dagger}} \frac{ \ltwonorm{\bP_{\cS ^ {\dagger} | \cS} \bmu ^ *} - \ltwonorm{\bP_{\cS | \cS ^ {\dagger}} \bmu ^ *}}{|\cS^{\dagger}\setminus\cS|^{1/2}} \ge \xi\sigma (\log p) ^ {1 / 2}.
	\]
\end{assume}

Assumption \ref{ass:packing2} is similar to Assumption \ref{ass:packing1}; the only difference is that in Assumption \ref{ass:packing2}, each delusive model is not a singleton but a set constructed by replacing $j_0$ in $\cS ^ {\dagger}$ with a spurious variable in $\cJ_{\delta_0}$. Assumption \ref{ass:residual_margin_s} solely involves the true support $\cS ^ *$ and signal $\bmu ^ *$. It guarantees that $\cS ^ {\dagger}$ significantly outperforms all the other size-$s$ sure selection sets in terms of fitting $\bmu ^ *$, thereby being the best $s$-subset within $\cS ^ *$ with high probability.
In this vein, if we can find a size-$s$ model  $\cS$ outside $\AA^ *(s)$ that fits $\by$ better than $\cS ^ {\dagger}$, BSS then favors $\cS$ and selects false variables.

The following theorem shows that with high probability, the best subset involves false selection once there are sufficiently many models in $\AA(s)$ with small projected signal margin as per \eqref{eq:psm_nec2}. Therefore, together with Theorem \ref{thm2}, Theorem \ref{thm3} shows tightness of the lower bound of the minimum projected signal margin \eqref{eq:psm_suff} (up to a constant) uniformly over $s \in [s ^ * - 1]$. 

\begin{theorem}
	\label{thm3}
	Suppose that $\{\epsilon_i \}_{i\in[n]} $ are independent Gaussian noise with variance $\sigma ^ 2$ and $s\ge 2$. Under Assumptions \ref{ass:packing2} and \ref{ass:residual_margin_s}, 
	we have that  
	\[
		\PP\biggl(\min_{\cS \in \AA_{j_0}} \cL_{\cS} < \min_{\cS \in \AA ^ *(s )} \cL_{\cS} \biggr) \ge 1 -4p ^ {-(\xi c_1\delta ^ 2_0 - 1)c_{\delta_0}} - p ^ {- c_1\delta_0 ^ 2 c_{\delta_0}} - 6sp^{-(c_2\xi^2-2)},  
	\]
	where $c_1$ and $c_2$ are universal constants. 
\end{theorem}

\subsection{Projected signal margin under random design}\label{sec_margin}

The previous two subsections demonstrate the pivotal role of the minimum projected signal margin $\psm_*(s)$ in underpinning the sure selection of BSS. In this subsection, we explicitly derive $\psm_*(s)$ under random design, so that we can verify the sufficient condition \eqref{eq:psm_suff} with ease. 
The following theorem gives an explicit lower bound of $\psm_*(s)$ when we have isotropic Gaussian design and $\bbeta ^ *_{\cS ^ *}$ has homogeneous entries. 

\begin{theorem}
	\label{thm:m_lower_bound_ind}
	Consider $n$ independent observations $\{\bx_i\}_{i\in[n]}$ of $\cN(\textbf{0},\bI_p)$. Suppose that $\log p \gtrsim s ^ *$, and that $\beta_j^* = \beta, \forall j\in\cS^*$. Then there exist universal constants $C, c>0$ such that for any  $\kappa > 9C^2$, whenever $n \ge \kappa {s^*}^2 (\log p)^2$, we have 
	\[
	    \PP\bigl\{ \psm_*(s) \ge n^{1/2}c|\beta|, ~\forall s\in[s^*-1] \bigr\} \ge 1-C{s^*}^2p^{-(C^{-1}\kappa ^ {1/2}-3)}. 
	\]
\end{theorem}
Theorem \ref{thm:m_lower_bound_ind} provides an explicit characterization of the projected signal margin under independent Gaussian design. 
Combining Theorem \ref{thm1} and Theorem \ref{thm:m_lower_bound_ind}, we see that $|\beta| \gtrsim \sigma \bigl(\frac{\log p}{n}\bigr)^{1/2}$ is sufficient for BSS to achieve sure early selection with high probability in this case. This lower bound for $|\beta|$ is aligned with the well-known $\beta$-$\min$ condition for support recovery. Specifically, Theorem 3.3 in  \cite{javanmard2019false} shows that when $\min_{j\in\cS^*} |\beta^*_j|/ \bigl\{ \sigma (\log p/n )^{1/2}\bigr\}\to \infty$ as $n,p\to \infty$, one can achieve asymptotic power one through a false discovery rate control procedure based on debiased LASSO. 
Theorem 2 in \cite{wainwright2009information} shows the necessity of the $\beta$-$\min$ condition in identifying true variables: if the $\beta$-$\min$ condition is violated, the failure probability for variable selection consistency can be lower bounded by $1/2$. Note that the novelty of Theorem \ref{thm:m_lower_bound_ind} given the previous works is that it provides selection accuracy guarantee for the entire early solution path rather than only the case when $s = s ^ *$. 

Theorem \ref{thm:m_lower_bound_ind} demonstrates that our lower bound \eqref{eq:psm_suff} for the projected signal margin is achievable under independent Gaussian design. One can actually extend Theorem \ref{thm:m_lower_bound_ind} to any covariance structure of the design; we choose to focus on the independent design for simplicity of presentation and comparison of the results.
The independent Gaussian design has been frequently used to investigate the solution path of variable selectors, e.g., \cite{su2017false}, \cite{su2018first} and \cite{wang2020complete}.



\section{Algorithmic strategies}\label{sec_ipath}

Sure early selection is an appealing property of BSS. Nevertheless, it can be difficult to achieve in practice because of the computational difficulty of exact BSS. 
In this section, we introduce three computational strategies for BSS to pursue sure early selection. 
Section \ref{sec:lbss} proposes a ``screen then select'' strategy that runs exact BSS within only the features that pass a preliminary screening.
Section \ref{sec:aprroxal} introduces the iterative hard thresholding algorithms, exemplified by projected gradient descent \citep{BDa08,BDa09} and CoSaMP \citep{NTr09}, as computational surrogates for BSS.


\subsection{Screen then select}\label{sec:lbss}

Recently, \cite{bertsimas2016best} proposed a MIO (Mixed Integer Optimization) formulation for best subset selection, which is amenable to a group of popular
optimization solvers including CPLEX, GLPK, MOSEK and GUROBI. These solvers can handle BSS problems with $n,p$ in $1000$s within minutes.  Nevetheless, $p$ can go far beyond thousands in ultrahigh-dimensional problems, for which applying the MIO algorithm is again computationally burdensome or even infeasible. 

To resolve the issue of high dimension, we consider the following ``screen then select" (STS) strategy. 
In the screening stage, we generate a screening set $\wt{\cS} \subset [p]$ of size $\wt{s} > s$ through a preliminary feature screening procedure, which can be but is not limited to penalized least squares methods \citep{tibshirani1996regression, fan2001variable, zhang2010nearly} and sure independence screening \citep{fan2008sure}. In the selecting stage, we solve the exact BSS problem with only the screened features, i.e., $\bX_{\wt\cS}$, to finalize the model selection.  
Define the collection of near best $s$-subsets within $\wt\cS$ as 
\[
	\wt{\SS}(s,\eta,\wt{\cS}) := \{\cS\subset \wt{\cS}: |\cS| = s, \cL_{\mS} \le \wt{\cL}_* + \eta  \psm_*^2(s) \}, 
\]
where $\wt{\cL}_* :=\min_{\cS\subset\wt{\cS}, |\cS| = s}\cL_\cS  $ is the optimal objective of the BSS problem on $\wt{\cS}$, and where $\eta$ controls the tolerance level of optimization error.
Define the sure screening event as $\cE := \{\cS^* \subset \wt{\cS}\}$. 
The following theorem shows that under event $\cE$, near best $s$-subsets within $\wt\cS$ can achieve sure early selection with high probability.

\begin{theorem}[Post-screening sure early selection]
    \label{thm:scrbss}
	Suppose that $\log p \gtrsim s ^ *$. There exists a universal constant $C \ge 1$, such that for any $\xi>C$ and $0 \le \eta < 1$, whenever
	\beq
		\label{eq:psm_suff_sts}
		\psm_*(s) \ge \frac{8 \xi \sigma (\log p)^{1/2}}{1-\eta}, 
	\eeq	
	we have that 
	\beq
	       \label{ineq:thm1_result_sts}
		\PP\bigl\{\fdp(\wh{\cS}   ) = 0 , \forall \wh{\cS} \in \wt{\SS}(s,\eta,\wt{\cS})
		\bigr\} \ge 1-Csp^{-2(C^{-2}\xi^2-1)} - \PP(\cE^c).
	\eeq
\end{theorem}

Theorem 7.21 in \cite{wainwright2019high} says that LASSO enjoys sure screening when the design satisfies the sparse Riesz condition and the irreprensentable condition and when the signals are not too weak. This means that under these conditions, if we use LASSO as the feature screener in the STS strategy, LASSO can generate an $\wt \cS$ such that $\PP(\cE ^ c) \to 0$ as $n, p \to \infty$. Combining this with the theorem above then implies that any near best $s$-subset of $\wt \cS$ achieves sure early selection. 
This motivates LBSS (LASSO plus BSS) that uses LASSO as the feature screener and then runs exact BSS for further selection.
Section \ref{sec:paths} numerically shows that LBSS achieves superior FDR-TPR tradeoff on the early solution paths over other approaches. 

\subsection{Iterative hard thresholding}\label{sec:aprroxal}

Another algorithmic strategy to implement BSS is through iterative hard thresholding (IHT), which can be viewed as greedy approximation for BSS. This section introduces two IHT algorithms. The first one is the vanilla IHT algorithm due to \cite{BDa08} and \cite{BDa09}, which is basically a projected gradient descent (PGD) algorithm that projects the iterate to an $\ell_0$-ball after a gradient descent update. The other one is CoSaMP (compressive sampling matching pursuit), an iterative two-stage hard thresholding algorithm proposed by \cite{NTr09}. We put the pseudocode of the two algorithms in Section A of the appendix for readers' reference. 

$\mathrm{PGD}$ enforces the sparsity of the solutions by directly projecting them to an $l_0$-ball. \cite{BDa08} showed that PGD enjoys near-optimal error guarantee within few iterations whose number depends only on the logarithm of a form of the signal-to-noise ratio of the problem. 
CoSaMP, instead, performs two rounds of hard thresholding in each iteration: it first expands the model by recruiting the largest coordinates of the gradient (first hard thresholding) and then reduces the model by discarding the smallest components of the refitted signal on the expanded model (second hard thresholding).
Under restricted isometry conditions, \cite{NTr09} showed that CoSaMP achieves optimal sample complexity and optimization guarantees, as well as high computational and memory efficiency, for recovering sparse signals. Later, \cite{jain2014iterative} established similar statistical and optimization guarantees for a generalized version of CoSaMP that serves any objective satisfying restricted strong convexity (RSC) and restricted strong smoothness (RSS). 
One particularly interesting result therein that connects CoSaMP with BSS is that given the true sparsity $s ^ *$, CoSaMP can find a model whose sparsity is of the same order of $s ^ *$ that achieves superior goodness of fit over the exact best size-$s$ subset. 
Based on this optimization guarantee, a more recent work \citep{FGZ20} discovers that CoSaMP (referred to as IHT therein) can achieve sure screening properties when enforced to select more than $s ^ *$ variables. These positive results motivate us to investigate the performance of PGD and CoSaMP on their early solution paths.



On top of these two existing methods, we propose Algorithm \ref{alg_ihtpath} to efficiently compute their solution paths.
Specifically, let $\texttt{A}$ denote either PGD or CoSaMP, and let $\Xi_{\texttt{A}}$ denote the collection of all the input of algorithm $\texttt{A}$ except initial value $\wh{\bbeta}_0$ and projection size $\pi$.
Consider a set of projection sizes $\Pi = \{\pi_1, \ldots, \pi_M\}$ with $\pi_1 < \ldots < \pi_M$. 
For any $i \in [M]$, let $\wh \bbeta ^ {\texttt{A}}(i)$ denote the solution of algorithm $\texttt{A}$ with $\pi = \pi_i$. Instead of computing $\wh \bbeta ^ {\texttt{A}}(i)$ for each $i \in [M]$ separately, we propose to use $\wh \bbeta ^ {\texttt{A}}(i)$ as a warm initializer to compute $\wh \bbeta ^ {\texttt{A}}(i + 1)$ (see line 3 of Algorithm \ref{alg_ihtpath}), which turns out to substantially accelerate the convergence in our numerical experiments.

\begin{center}
        \begin{algorithm}[H]
            \caption{\texttt{A} $-$ Path($\Xi_{\texttt{A}} , \Pi$)}\label{alg_ihtpath}
            \KwIn{$\Xi_{\texttt{A}}$, projection size set $\Pi = \{\pi_1, \ldots, \pi_M\}$ with $\pi_1 < \ldots < \pi_M$.}
            \begin{algorithmic}[1]
               \State  $\wh\bbeta^{\texttt{A}}(0) = \textbf{0}$
                \State \textbf{for} i = 1:$M$
                \State $~~\wh{\bbeta}^{\texttt{A}}(i)\gets\texttt{A}\bigl(\Xi_{\texttt{A}}, \wh{\bbeta}_0 =  \wh\bbeta^{\texttt{A}}(i - 1), \pi = \pi_i)$
                \State \textbf{end for}
            \end{algorithmic}
        	\KwOut{$\bigl(\wh\bbeta^{\texttt{A}}(1), \wh\bbeta^{\texttt{A}}(2), ..., \wh\bbeta^{\texttt{A}}(M)\bigr)$}
        \end{algorithm}
\end{center}

In Section \ref{sec:paths}, we see that $\mathrm{LBSS}$ performs overall the best in terms of the FDR-TPR tradeoff on the early solution path among all the investigated methods including PGD, CoSaMP, LASSO, etc.

\section{Numerical experiments}\label{sec:paths}

This section numerically investigates the FDR and TPR of the early paths of LBSS, PGD, COSAMP as well as penalized methods including LASSO, SCAD, MCP and $L_0L_2$. 
Generally speaking, given a model selection procedure with tuning parameter $\lambda \in \RR$, the FDR-TPR path is formed by the set $\bigl\{\bigl(\mathrm{FDR(\wh\cS(\lambda))}, \mathrm{TPR}(\wh\cS(\lambda))\bigr)\bigr\}_{\lambda> 0}$, where $\wh\cS(\lambda)$ is the output of the model selector with tuning parameter $\lambda$. 
For LBSS, the tuning parameter $\lambda$ is the target sparsity of best subset selection, while the screening size is pre-specified ($300$ variables in Section \ref{sec:simu} and $40$ variables in Section \ref{sec:realpath}).  For both PGD and CoSaMP, $\lambda$ is the projection size $\pi$. Regarding the expansion size $l$ in  CoSaMP, we set it to be the size of the model selected by MCP that is tuned by $10$-fold cross validation (CV) in terms of the mean squared error (MSE). 
For L0L2, $\lambda$ is the tuning parameter for $L_0$ regularization; the tuning parameter for $L_2$ regularization, which we donte by $\gamma$ after display \eqref{equ_penalizedsol}, is chosen oracularly: we found that $\gamma = 10 ^ {-5}$ yields the best FDR-TPR curve. 
Note that a perfect FDR-TPR path is in the ``$\Gamma$'' shape, meaning that the method keeps recruiting (nearly) true predictors until having them all. 

In the following, we evaluate the solution paths of all the approaches on both synthetic data and a skin cutaneous melanoma dataset. 
We use the R package \texttt{bestsubset} to solve exact BSS, the R package \texttt{picasso} \cite{GLJ19} to implement the LASSO, SCAD and MCP, and the R package \texttt{L0Learn} \citep{hazimeh2022l0learn} to implement L0L2.

\subsection{Simulated data}\label{sec:simu}


In this section, we investigate two different designs for numerical comparison: the autoregressive design and the equicorrelated design. Specifically, we set $p$ = $5,000$, $s^{\ast}$ = $50$ and $n=\lceil 2s^*\log p \rceil$. In each Monte Carlo experiment, we generate $\mS^{\ast}$ $\subset$ $[p]$ by randomly selecting $s^{\ast}$ locations from $[p]$ in a uniform manner and generate $\bbeta^*$ such that $\beta_j^{\ast}$ = $0$ for $j\in(\mS^{\ast})^c$ and that $\{(\beta_j^{\ast}/0.1)-1\}_{j\in\mS^{\ast}}\overset{\text{i.i.d.}}{\sim}\chi_1^2$.
Rows of the design matrix $\bX$ are independently generated from $\mN(\mathbf{0},\bSigma)$, where $\bSigma$ is either from autoregressive (Case 1) or equicorrelated (Case 2) models. 
We generate $\by$ as $\by=\bX\bbeta^{\ast}+\bepsilon$ with $\bepsilon\sim\mN(\mathbf{0},\sigma^2\bI)$, where $\sigma ^ 2$ will be specified later. Each point on the presented FDR-TPR curve is based on the average of the FDPs and TPPs of the estimators obtained in 100 independent Monte Carlo experiments, using the same tuning parameter. In some challenging setups  such as those in Case 2, many methods can yield very low TPR. Besides, it is often meaningless to compare the early solution path when FDR is overwhelmingly high; therefore, for certain plots, we zoom in to show only the low-FDR part to facilitate comparison. \\


\noindent
\textsc{Case 1: Autoregressive design.} We set $\Sigma_{jk}=\rho^{|j-k|}$ for correlation parameter $\rho\in\{0, 0.5, 0.8 \}, \forall j, k \in [p]$ and $\sigma \in \{0.5,1\}$. 
The FDR-TPR paths, together with the FDR and TPR of the solutions chosen by 10-fold cross validation, are presented in Figure \ref{fig:ar}.
We have the following observations:
\begin{figure}
      \centering
      \setlength\tabcolsep{0pt}
	\renewcommand{\arraystretch}{-5}      
      \begin{tabular}{ccc}
      \includegraphics[width=4.8cm]{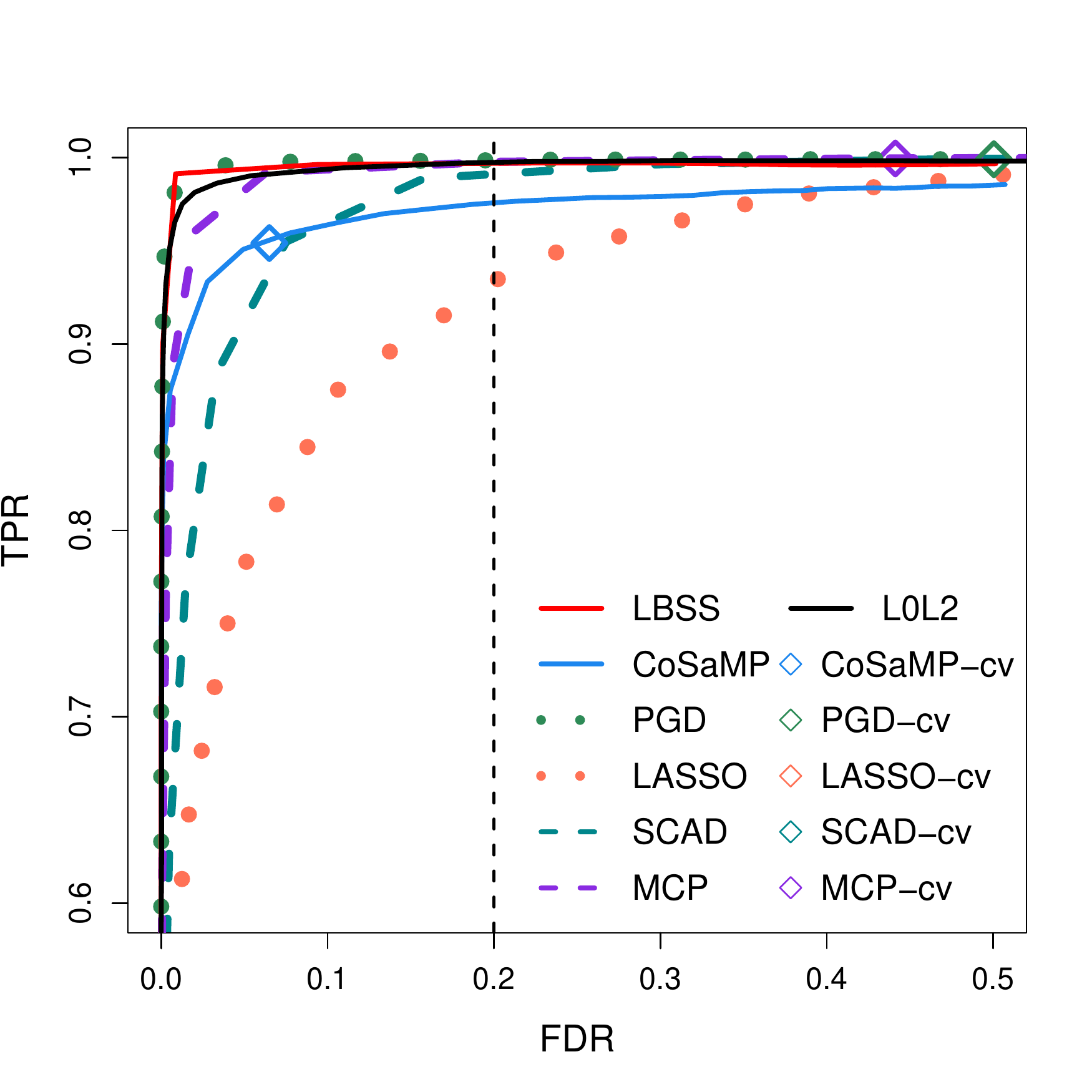} & \includegraphics[width=4.8cm]{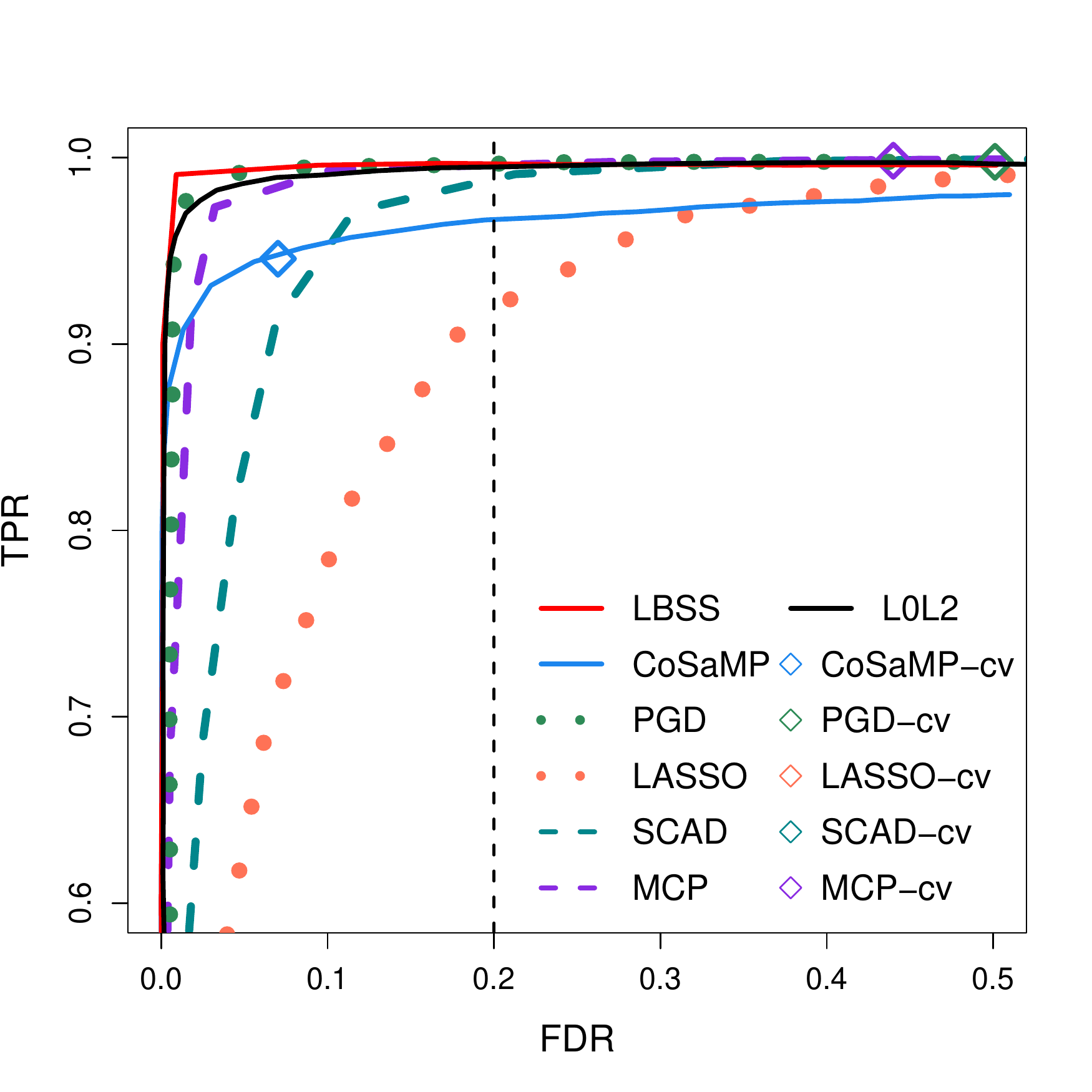} & \includegraphics[width=4.8cm]{11c} \\  
    $\rho = 0, \sigma = 0.5$ & $\rho = 0.5, \sigma = 0.5$ & $\rho = 0.8, \sigma = 0.5$ \\
      \includegraphics[width=4.8cm]{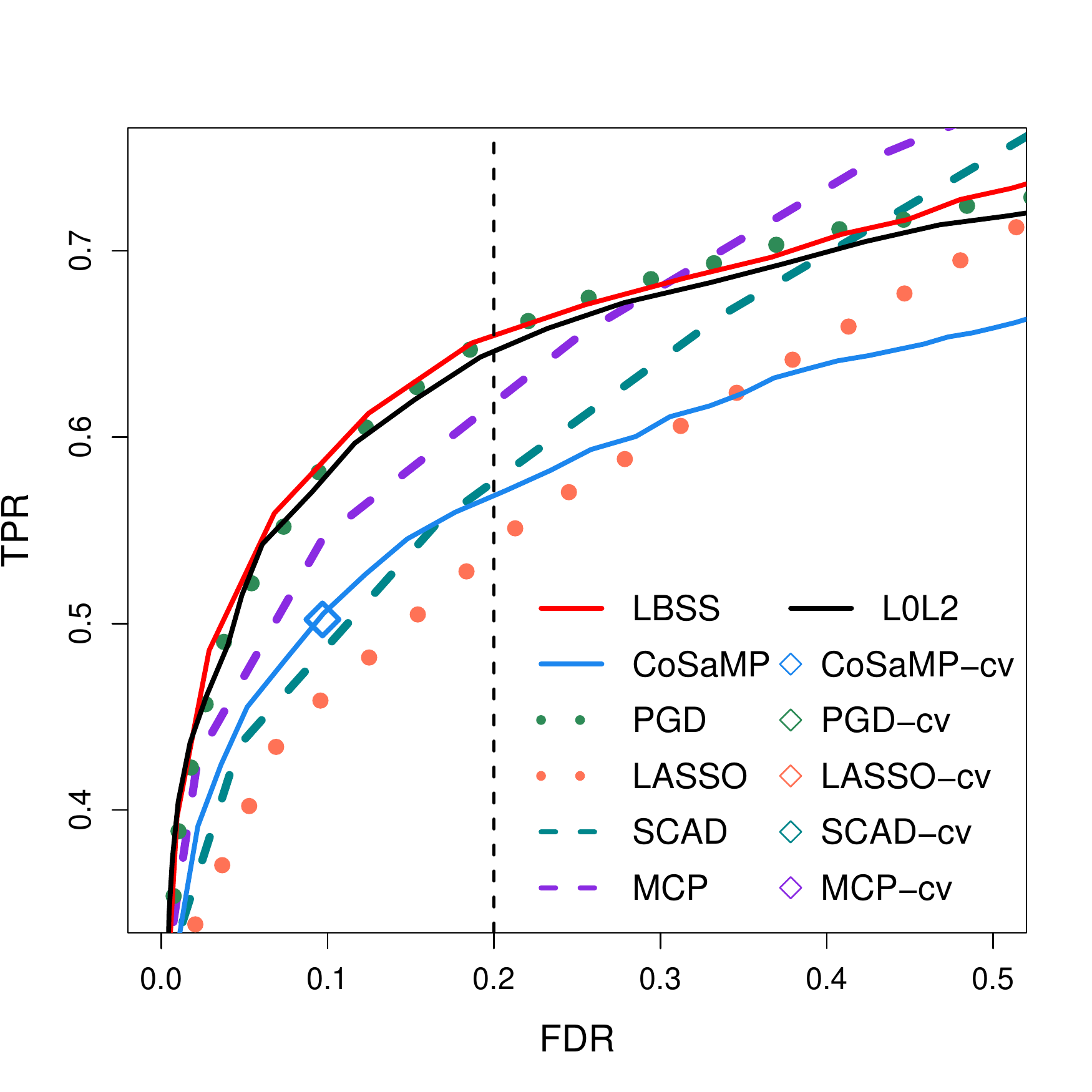} &  \includegraphics[width=4.8cm]{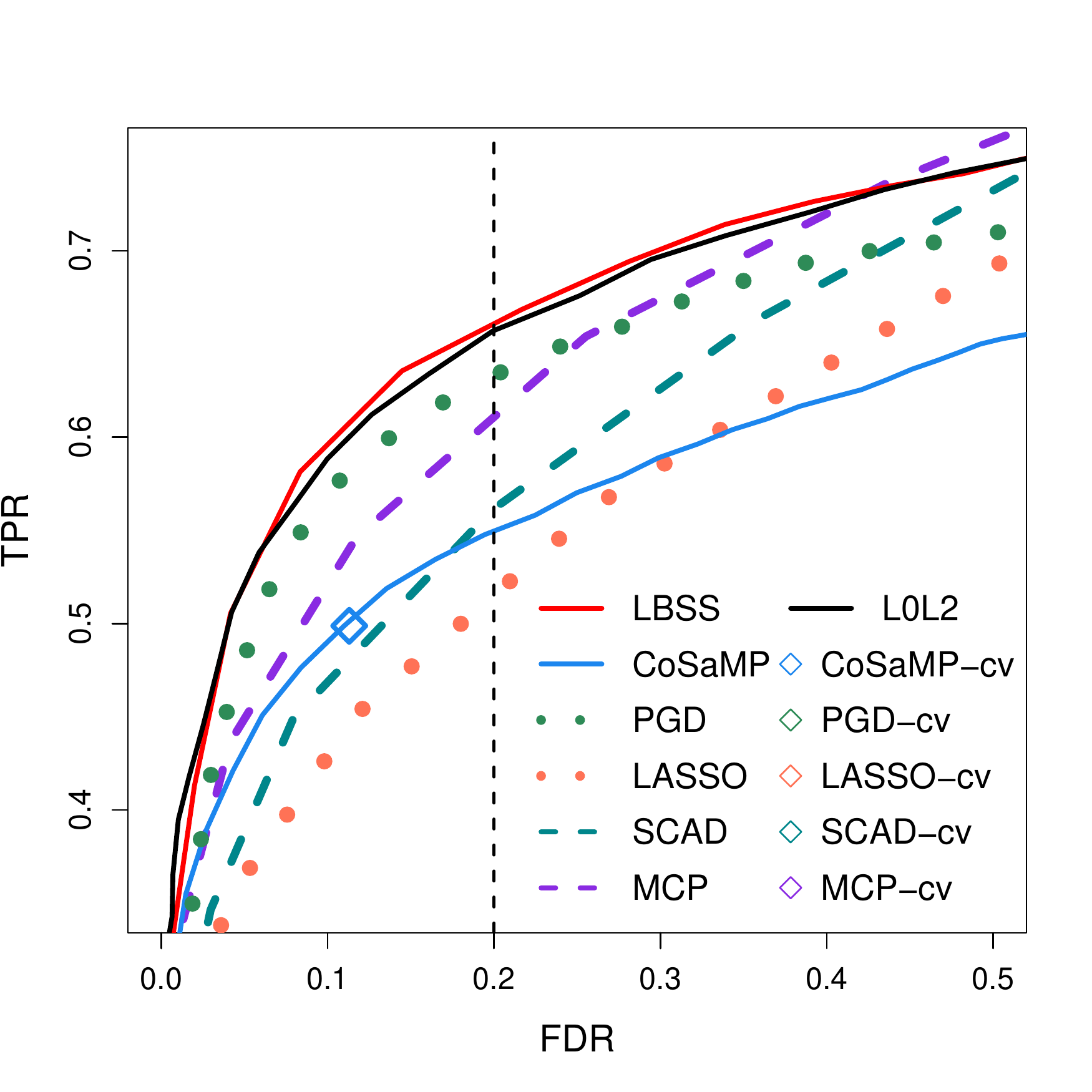} & \includegraphics[width=4.8cm]{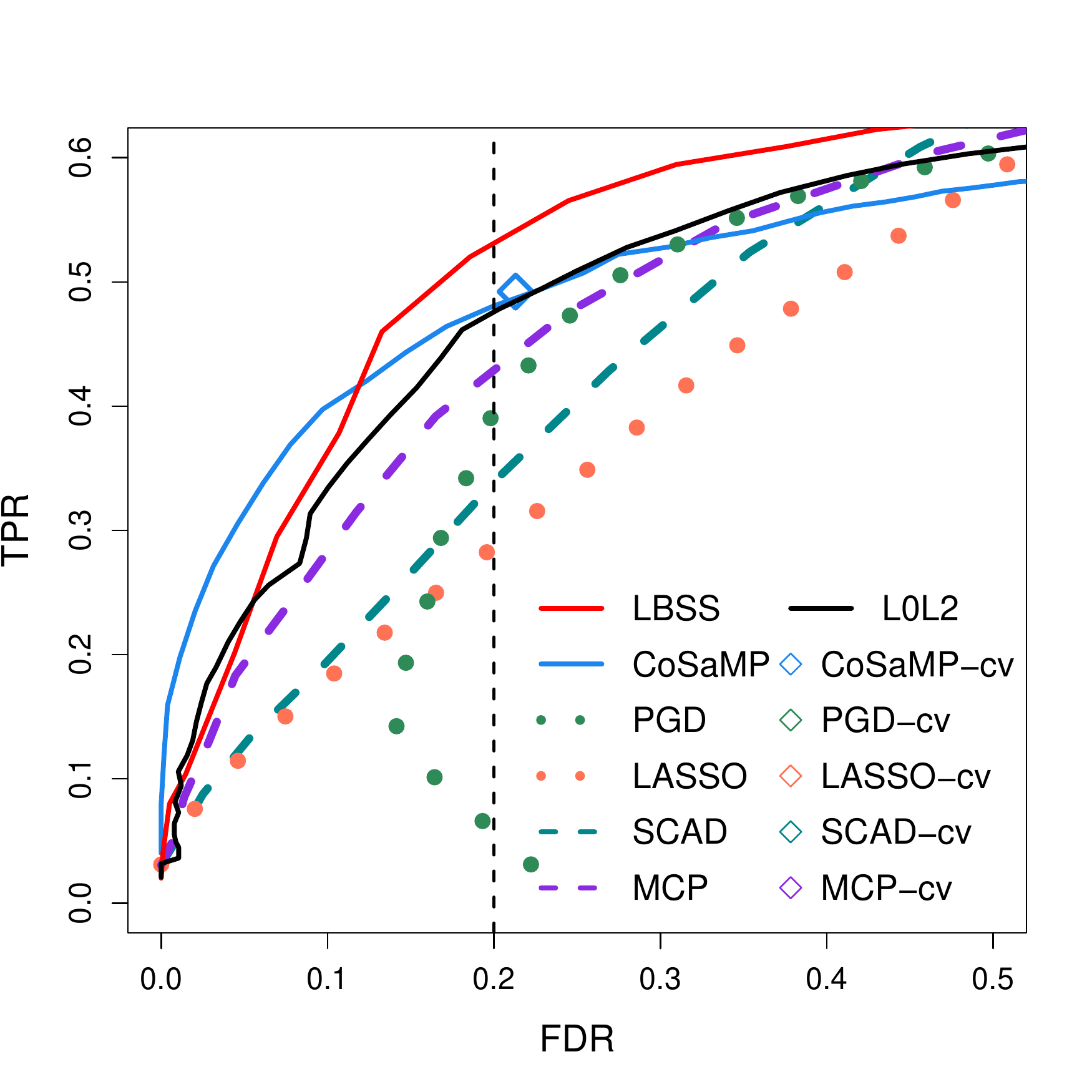} \\
    $\rho = 0, \sigma = 1$ & $\rho = 0.5, \sigma = 1$ & $\rho = 0.8, \sigma = 1$ \\
      \end{tabular}
      \caption{\footnotesize{FDR-TPR paths under autoregressive design. We set $\Sigma_{jk} = \rho^{|j-k|}$ for $\rho\in\{0,0.5,0.8\}$ and $\sigma \in \{0.5, 1\}$. The diamonds represent the solutions of different methods that are tuned by $10$-fold CV. The results are obtained over $100$ Monte Carlo repetitions by randomly generating the design $\bX$, the coefficient vector $\bbeta^*$, the noise $\bepsilon$ and computing $\by = \bX\bbeta^* + \bepsilon$.}}
      \label{fig:ar}
\end{figure}
\begin{itemize}
      \item[(i)] In the independent settings ($\rho=0$), LBSS, PGD and L0L2 perform similarly and outperform the other approaches on the early solution path. In particular, when $\rho = 0$ and $\sigma = 0.5$, both LBSS and PGD exhibit ``$\Gamma$"-shaped FDR-TPR paths. CoSaMP has no false discovery until TPR exceeds $80\%$ but then recruits true variables in a relatively slow manner that
      prohibits it from selecting all the true variables even when the FDR is as high as $50\%$. LASSO instead has false discoveries relatively early on the solution path.

      \item[(ii)] As the correlation between features increases, i.e., when $\rho\in\{0.5,0.8\}$, PGD performs worse. CoSaMP and L0L2, in contrast, still have comparable early selection performance as LBSS. When $\rho = 0.8$ and $\sigma=0.5$, CoSaMP appears to be more robust against strong design dependence than LBSS, but still cannot achieve as high power as LBSS.  Overall, LBSS performs the best in terms of the FDR-TPR tradeoff on the early solution path.

      \item[(iii)] Regarding the TPR performance, when FDR is controlled at $20\%$, which is a widely used level, LBSS always enjoys the highest TPR.  
      
\end{itemize}

\noindent
\textsc{Case 2: Equicorrelated design.} 
\begin{figure}[t]
      \centering
      \setlength\tabcolsep{0pt}
	\renewcommand{\arraystretch}{0}        
      \begin{tabular}[b]{cc}
              \multicolumn{2}{c}{\includegraphics[width=8cm]{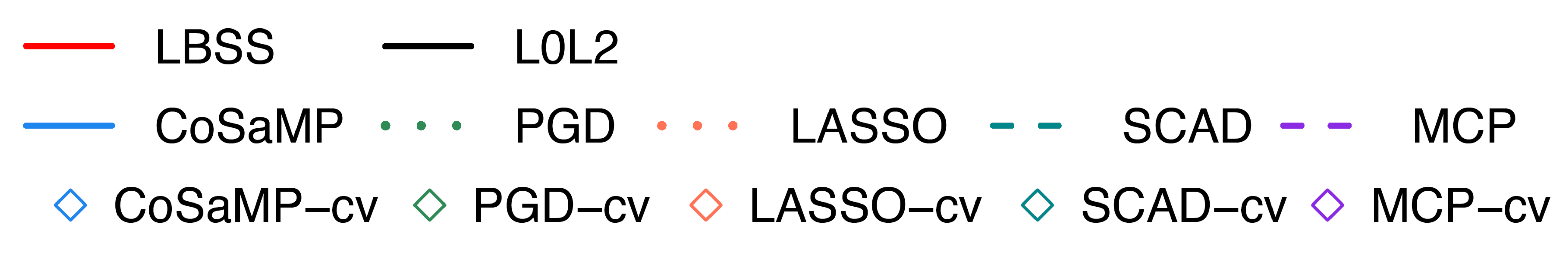}} \\
      \includegraphics[height=5cm]{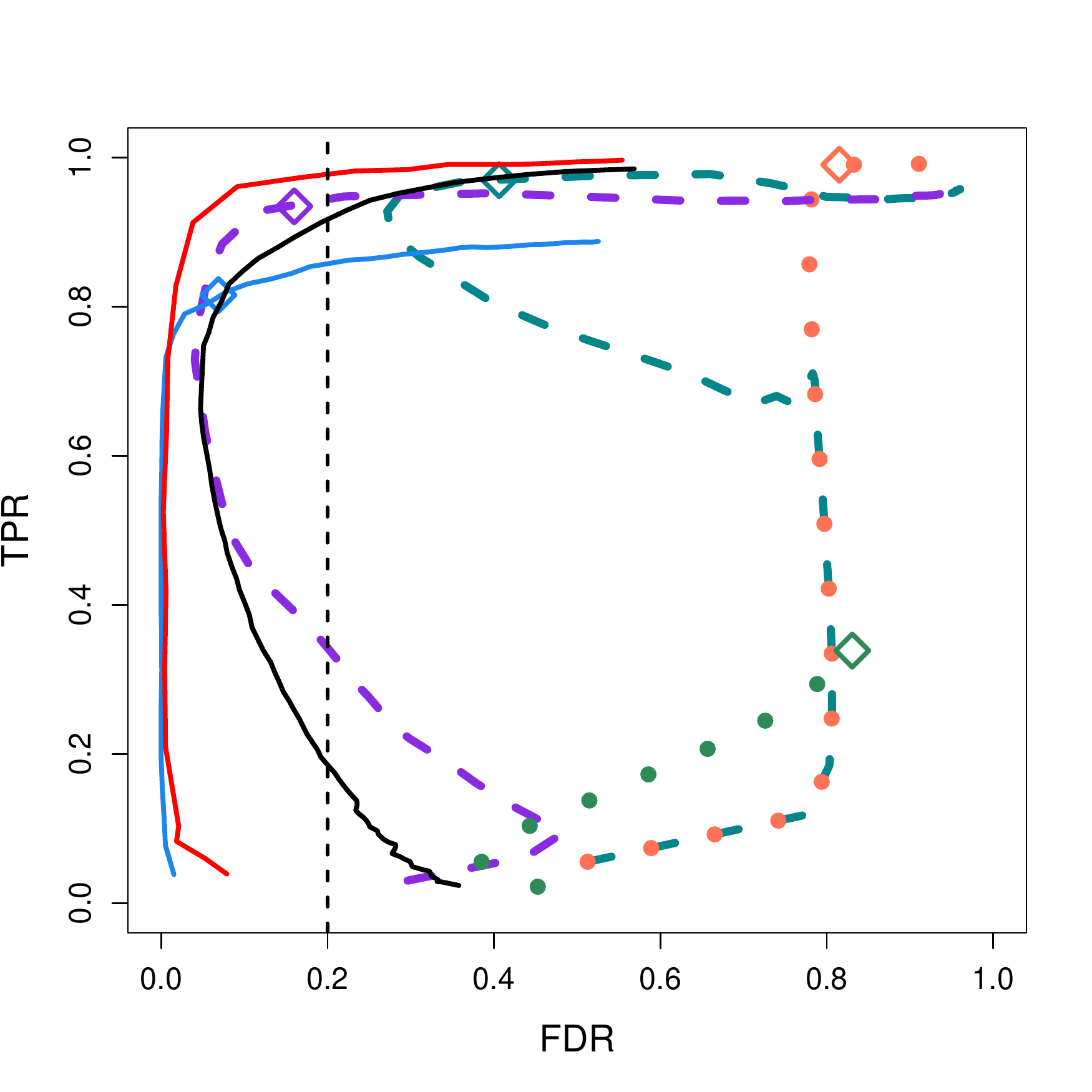} &  \includegraphics[height=5cm]{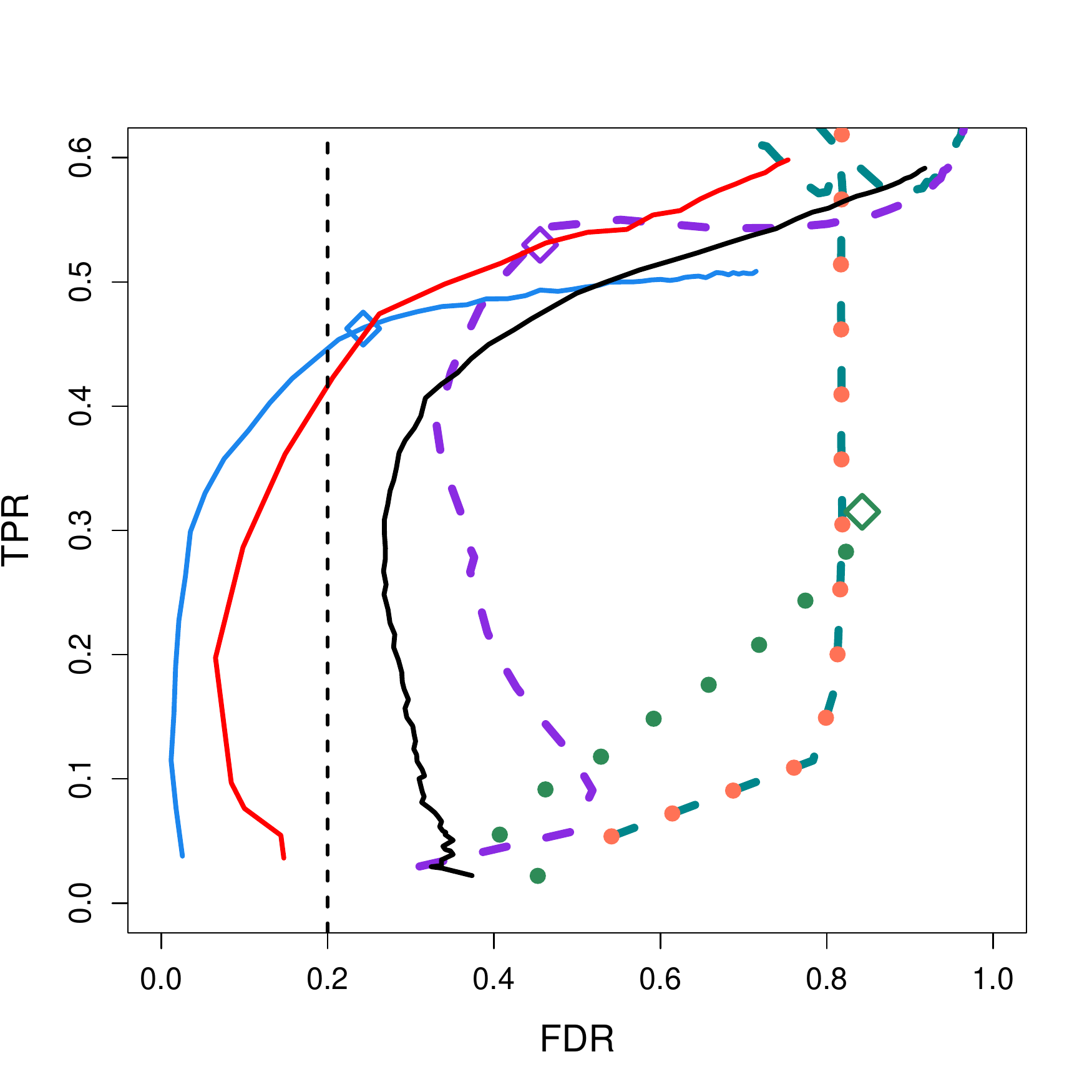} \\
      \\ 
      $\rho = 0.5, \sigma = 0.5$ &  $\rho = 0.5, \sigma = 1$ \\
      \\
      \includegraphics[height=5cm]{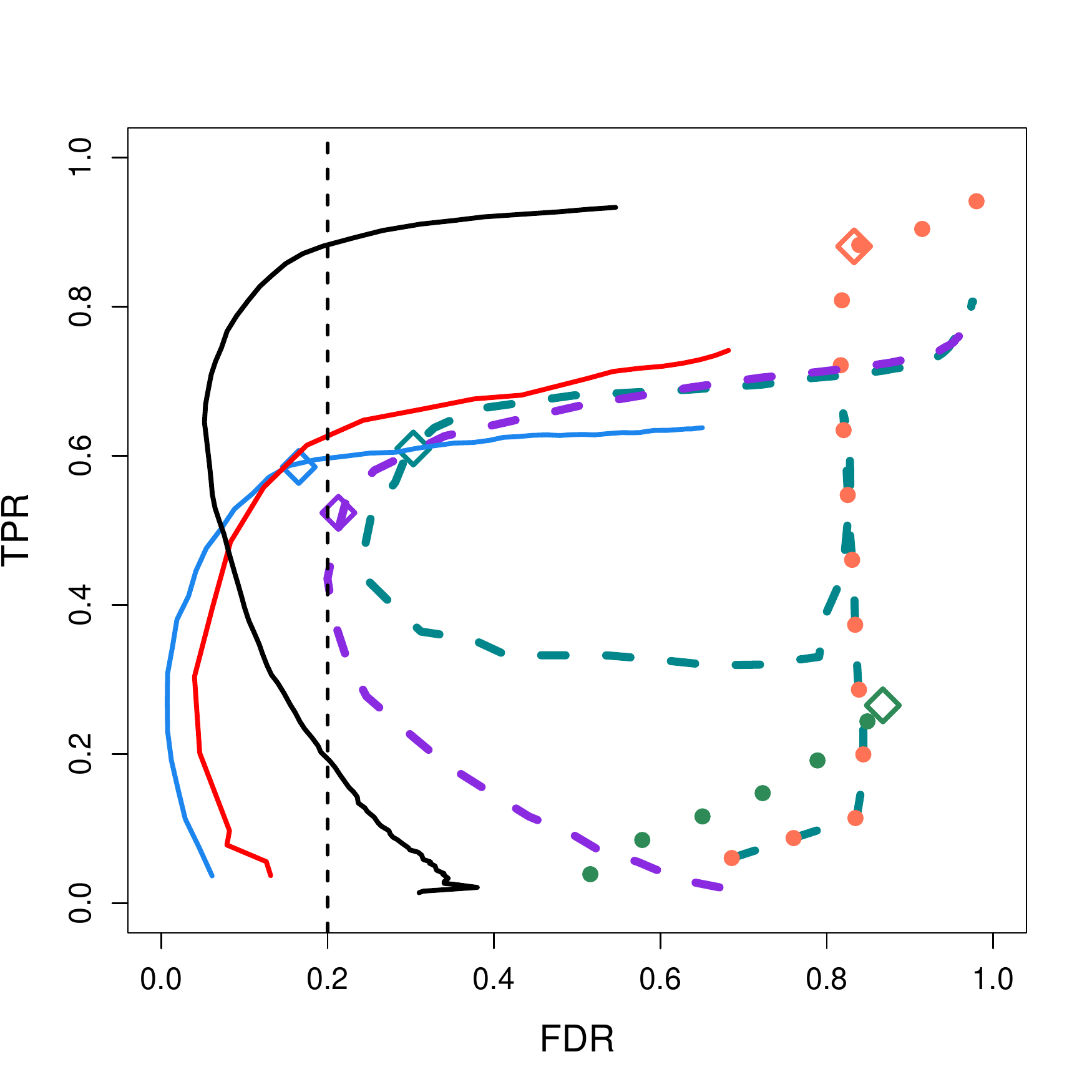} &  \includegraphics[height=5cm]{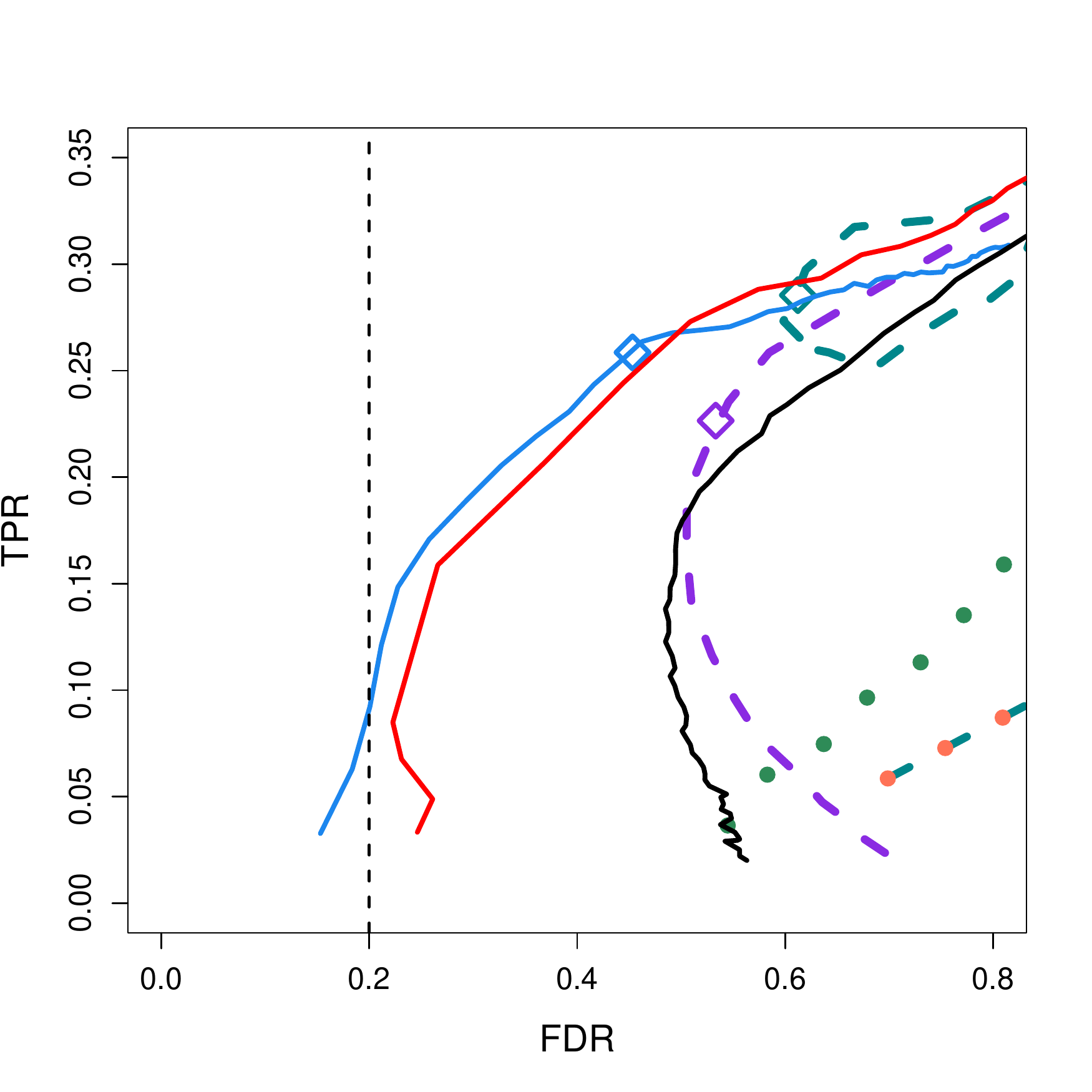}  \\
      \\
      $\rho = 0.8, \sigma = 0.5$ &  $\rho = 0.8, \sigma = 1$ \\
      \end{tabular}
      \caption{\footnotesize{FDR-TPR paths under equicorrelated design. We set $\bSigma = \rho\textbf{1}\textbf{1}^\top + (1-\rho)\bI$ for $\rho \in\{0.5,0.8\}$ and $\sigma \in \{0.5, 1\}$. 
      The diamond points represent the $10$-fold CV solutions.  }}
      \label{fig:equi}
\end{figure}
We consider $\bSigma = \rho\mathbf{1}\mathbf{1}^{\top}+(1-\rho)\bI$ with $\rho \in \{0.5, 0.8\}$ and two noise levels $\sigma\in\{0.5, 1\}$.
This is a much more challenging case than the previous one because of constant correlation. 
The FDR-TPR paths are presented in Figure \ref{fig:equi}.
We have the following observations:
\begin{itemize}
      \item[(i)]  {CoSaMP} and LBSS maintain to yield a nearly vertical FDR-TPR path at the early stage when $\rho = 0.5$ and $\sigma = 0.5$, meaning that most of their early discoveries are genuine. However, with higher design correlation or stronger noise, their performance deteriorates.
      
      \item[(ii)] Compared with the case of autoregressive design (Case 1), the strengthened collinearity here hurts the overall performance of all the methods. For example, fixing $\rho = 0.8$ and $\sigma = 0.5$ and requiring FDR below $20\%$, one can see that CoSaMP can hit over $80\%$ TPR in Case 1 but less than $60\%$ TPR in Case 2, and that LBSS hits over $90\%$ TPR in Case 1 but only a little over $60\%$ in Case 2. 
    
      \item[(iii)] LASSO, SCAD and MCP select much more false variables than LBSS and CoSaMP on their early solution path. However, interestingly, at a latter stage, the nonconvex penalty in SCAD and MCP corrects the models by replacing the false discoveries with true variables, leading to a north-west pivot of the path as  $\lambda$ decreases. LASSO instead does not exhibit such correction: its FDR-TPR path keep moving in the north-east direction as $\lambda$ decreases, meaning that its FDR keeps increasing. PGD totally breaks down in such challenging cases with strongly dependent designs. 
      
      \item[(iv)] The performance of L0L2 is outstanding when $\rho = 0.8$ and $\sigma = 0.5$. In spite of the false discoveries on the early path, L0L2 gradually eliminates them and recruits true variables as $\lambda$ decreases, so that in the latter part of the path, L0L2 can achieve more than $80\%$ TPR given FDR is below $20\%$. We attribute this advantage to the $L_2$ regularization of L0L2 that guards against strong correlation among the features.
\end{itemize}

\subsection{Semi-simulation with the skin cutaneous melanoma dataset}\label{sec:realpath}
In this section, we take our design matrix $\bX$ from the skin cutaneous melanoma dataset in the Cancer Genome Atlas 
(\href{https://cancergenome.nih.gov/}{https://cancergenome.nih.gov/}), which provides comprehensive profiling data on more than thirty cancer types \citep{sun2019counting}. The dataset contains $p = 20,351$ items of mRNA expression data of $n = 469$ patients. To choose reasonable locations of true variables, we adopt the top 20 genes that are found highly associated with cutaneous melanoma according to meta-analysis of 145 papers \citep{chatzinasiou2011comprehensive}. One can find a list of these genes at \href{http://bioserver-3.bioacademy.gr/Bioserver/melGene/}{http://bioserver-3.bioacademy.gr/Bioserver/melGene/}.
We conduct a semi-simulation using this real design matrix.
Fixing $\mS^{\ast}$ to be the locations of the top 20 genes, we generate $\bbeta^{\ast}$ by letting $\beta_j^{\ast}$ = $0$ for $j\in(\mS^{\ast})^c$ and  $\{(\beta_j^{\ast}/\beta_{\min})-1\}_{j\in\mS^{\ast}}\overset{\text{i.i.d.}}{\sim}\chi_1^2$, where $\beta_{\min} \in\{1,0.5,0.25\}$. Then we generate the response vector by letting $\by = \wt\bX\bbeta^{\ast}+\bepsilon$, where $\wt \bX$ is the standardized $\bX$, and where $\bepsilon\sim\cN(\mathbf{0}, \bI)$. 

Figure \ref{fig:real} presents the solution paths of all the methods investigated in Section \ref{sec:simu}. We have the following observations:

\begin{figure}
      \centering
      \setlength\tabcolsep{-2pt}
		\renewcommand{\arraystretch}{0}      
      \begin{tabular}{ccc}
      \multicolumn{3}{c}{\includegraphics[width=8cm]{legend}} \\
      \includegraphics[height=4.9cm]{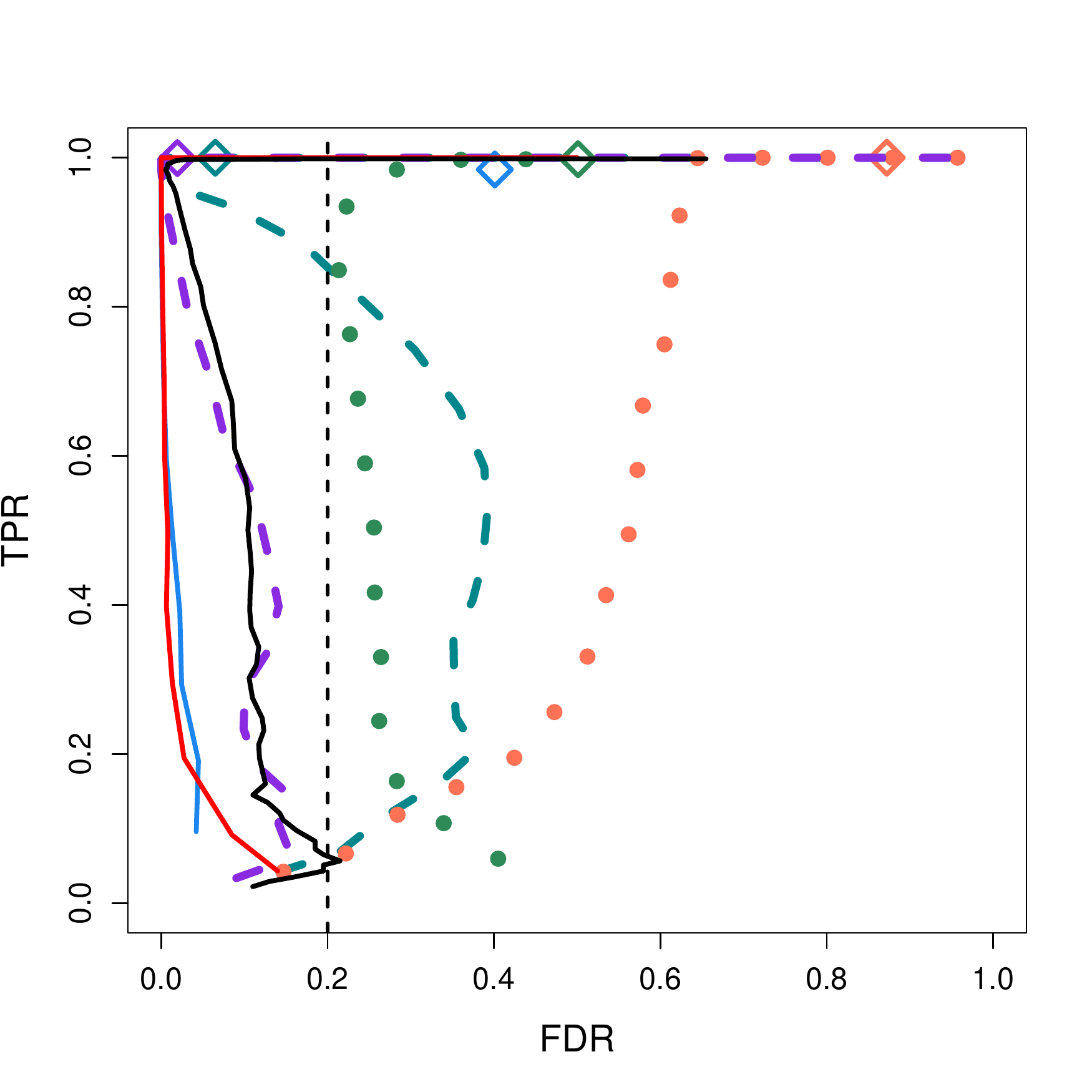} & \includegraphics[height=4.9cm]{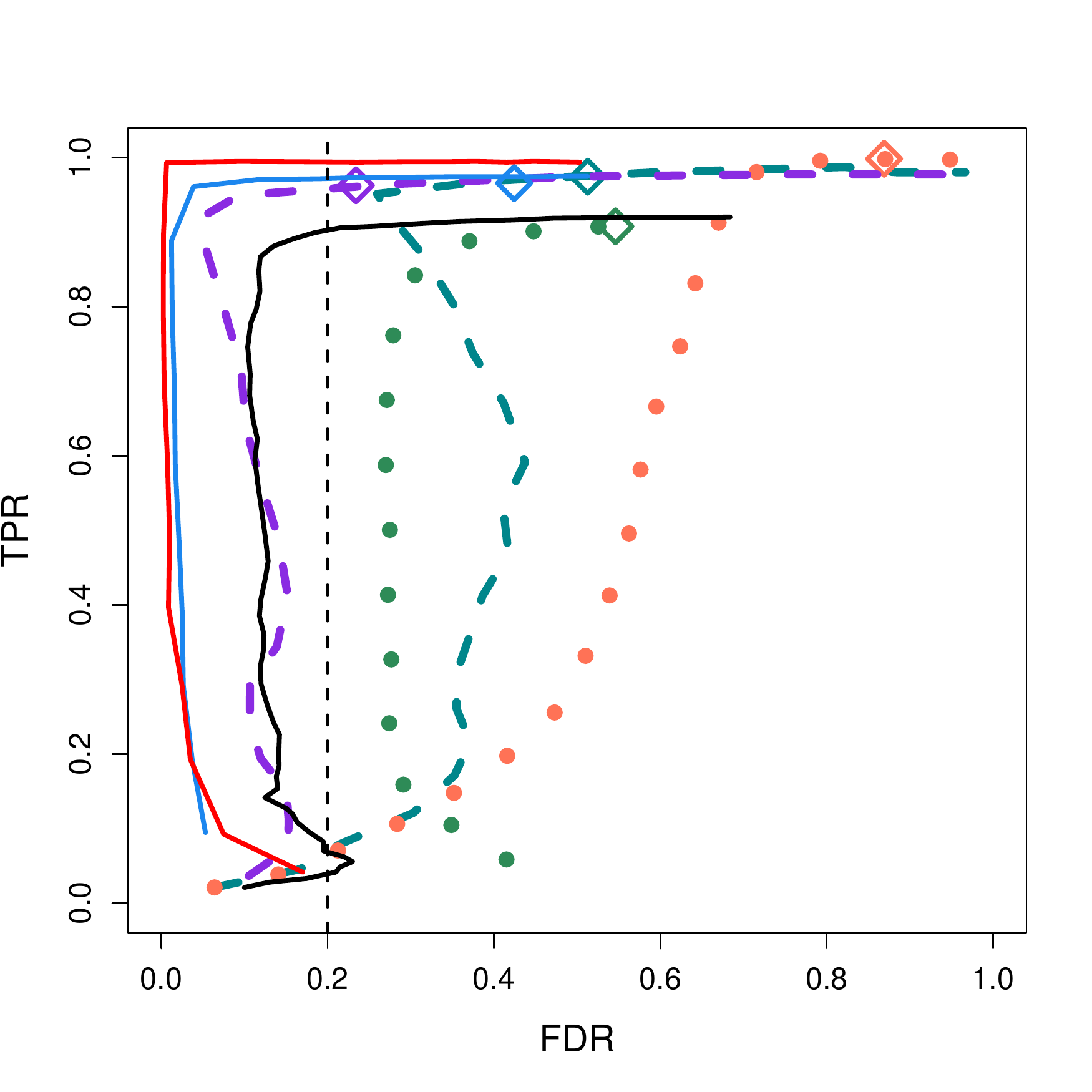} & \includegraphics[height=4.9cm]{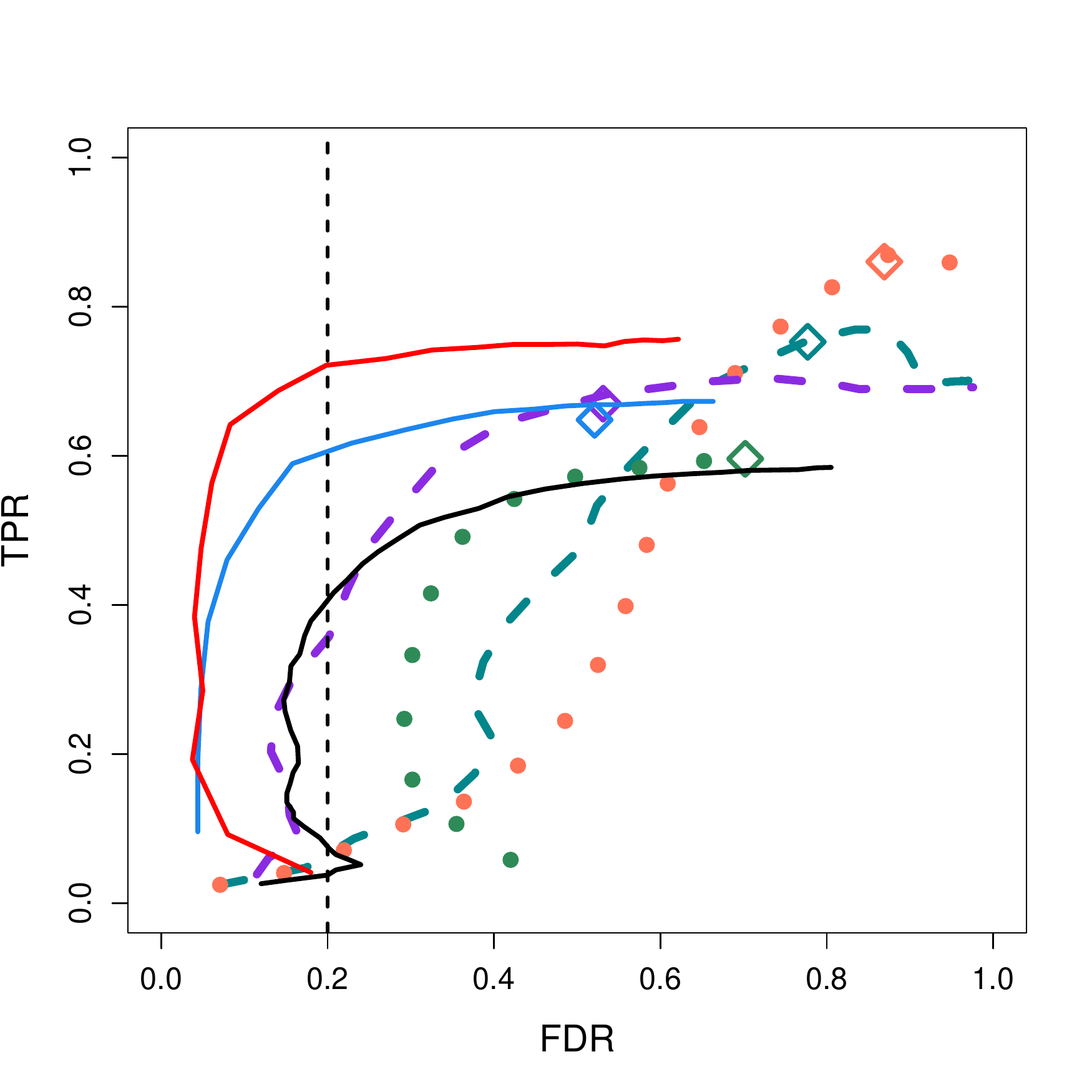}  \\
      $\beta_{\min} = 1$ & $\beta_{\min} = 0.5$ & $\beta_{\min} = 0.25$
      \end{tabular}
      \caption{FDR-TPR paths under real data design. We set the noise level $\sigma = 1$ and choose $\beta_{\min} \in \{1,0.5,0.25\}$.  The diamond points represent the $10$-fold CV solutions. }
      \label{fig:real}
\end{figure}
\begin{itemize}
      \item[(i)] Under the setups of relatively high signal-to-noise ratios ($\beta_{\min}=1,0.5$), the FDR-TPR paths of both LBSS and {CoSaMP} are nearly in the   ``$\Gamma$'' shape, while the other methods have false discoveries on their early solution paths.  In particular, when FDR is controlled at $20\%$, which is a widely used level, LBSS and CoSaMP always have the highest TPR. Similarly to Case 2 of Section \ref{sec:simu}, SCAD, MCP and L0L2 correct their models at a latter stage by substituting false variables with true ones. The FDR of LASSO keeps increasing as $\lambda$ decreases. 
      
     \item[(ii)]  Regarding the $10$-fold CV solutions, we see that all the methods with visible CV points have a non-negligible number of false variables in order to improve prediction power. The CV points of  {SCAD} and  {MCP} sometimes have lower FDR. {LASSO} tends to select false variables more than the other approaches.   
\end{itemize}


\section{Discussion} 

In this paper, we establish the sufficient and (near) necessary conditions for BSS to achieve sure selection throughout its early path. We show that the underpinning quantity is the minimum projected signal margin that characterizes the fundamental gap of fitting power between sure selection models and spurious ones. This margin is robust against collinearity of the design, justifying the low FDP of the early path of LBSS that we observed under highly correlated designs. The current results motivate the following two questions that we wish to answer in our future research: 
\begin{enumerate}
	\item Previous works \cite{NTr09} and \cite{jain2014iterative} on IHT rely on restricted strong convexity and smoothness (or their variants) of the loss function to develop its optimization guarantee. Given that these conditions are dispensable for the statistical properties of BSS, and that IHT is shown to be closely related with BSS, it is natural to ask if these conditions are necessary for IHT to yield desirable statistical properties, say sure early selection. In particular, can we achieve any FDP guarantee for the early path of CoSaMP with only a lower bound of the minimum projected signal margin? Answering these questions should entail more delicate optimization analysis. 
	
	\item What are the sufficient and necessary conditions for BSS to achieve sure early selection under linear sparsity, i.e., $s ^ * / p$ tends to a constant? 

    \item Under more general contexts (say generalized linear models), what are the counterpart sufficient and necessary conditions for $\ell_0$-constrained methods to achieve sure early selection?
\end{enumerate}

\section{Acknowledgement}

We appreciate the suggestions and guidance from Prof. Dimitris Bertsimas at MIT that substanitally improved the quality of this paper. We also appreciate the funding support of NSF DMS-2015366.


\bibliographystyle{ims.bst} 
\bibliography{ref.bib}  

\newpage
\begin{appendix}

\section{Pseudocode of the IHT algorithms}

For any $\bv\in\RR^p$ and $r\in\NN$, define
$\cT_{\text{abs}}(\bv, r) := \{j : |v_j| \text{ is among the top $r$ largest values}$ of $\{|v_k|\}_{k=1}^p  \}$.
For any $\bbeta\in\RR^p$, let $\cL(\bbeta)=\sum_{i=1}^n(y_i-\bx_i^\top\bbeta)^2$. We present the pseudocode of PGD and CoSaMP in Algorithms \ref{alg_iht1} and \ref{alg_iht}.

\begin{center}
        \begin{algorithm}
            \caption{PGD($\bX, \by, \wh\bbeta_0, \pi, \eta, \tau$)}\label{alg_iht1}
            \KwIn{Design matrix $\bX$, response $\by$, initial value $\wh\bbeta_0$, projection size $\pi$, step size $\eta$, convergence threshold $\tau>0$.}
            \begin{algorithmic}[1]
                \State ~~ $t$ $\gets$ $0$
                \State ~~ \textbf{repeat}
                \State ~~~~ \hspace*{0.02in} $\wh{\bbeta}_t^{\dagger}
                \gets 
                \wh{\bbeta}_t - \eta \nabla\mL(\wh\bbeta_t) $
                \State ~~~~ \hspace*{0.02in}
                $\mG_t  \gets \mT_{\text{abs}}(\wh{\bbeta}_t^{\dagger},\pi)$
                \State ~~~~ \hspace*{0.02in} $\wh\bbeta_t\gets(\bX_{\mG_t }^{\top}\bX_{\mG_t} )^{+}\bX_{\mG_t }^{\top}\by$
                \State ~~~~ \hspace*{0.02in} $t\gets t+1$
                \State ~~ \textbf{until}
                $\|\wh\bbeta_t-\wh\bbeta_{t-1}\|_2<\tau$
                \State~~ $\wh\bbeta^{\text{pg}}\gets\wh\bbeta_t$  
         \end{algorithmic}
         \KwOut{$\wh\bbeta^{\mathrm{pg}}$ }
        \end{algorithm}
\end{center}

\begin{center}
        \begin{algorithm}
            \caption{CoSaMP($\bX, \by, \wh\bbeta_0, \pi, l, \tau$)}\label{alg_iht}
            \KwIn{Design matrix $\bX$, response $\by$, initial value $\wh\bbeta_0$, projection size $\pi$, expansion size $l$, convergence threshold $\tau>0$.}
            \begin{algorithmic}[1]
                \State ~~ $t$ $\gets$ $0$
                \State ~~ \textbf{repeat}
                \State ~~~~ \hspace*{0.02in} $\mG_t \gets \mT_{\text{abs}}(\nabla\mL(\wh\bbeta_t),l)$
                \State ~~~~ \hspace*{0.02in} $\mS_t^{\dagger}\gets$supp$(\wh\bbeta_t)\cup\mG_t$
                \State ~~~~ \hspace*{0.02in} $\wh\bbeta^{\dagger}_t\gets(\bX_{\mS_t^\dagger}^{\top}\bX_{\mS_t^\dagger})^{+}\bX_{\mS_t^\dagger}^{\top}\by$
                \State ~~~~ \hspace*{0.02in} $\mS_t\gets\mT_{\text{abs}}(\wh\bbeta_t^{\dagger},\pi)$
                \State ~~~~ \hspace*{0.02in} $\wh\bbeta_{t+1}\gets (\bX_{\mS_t}^{\top}\bX_{\mS_t})^{+}\bX_{\mS_t}^{\top}\by$
                \State ~~~~ \hspace*{0.02in} $t\gets t+1$
                \State ~~ \textbf{until}
                $\|\wh\bbeta_t-\wh\bbeta_{t-1}\|_2<\tau$
                \State~~ $\wh\bbeta^{\text{cs}}\gets\wh\bbeta_t$            
         \end{algorithmic}
         \KwOut{$\wh\bbeta^{\mathrm{cs}}$ }
        \end{algorithm}
\end{center}

\section{Proofs of technical results}
\subsection{Proof of Theorem 2.1}
For any $\cS \in [p]$, define $\bgamma_{\cS} := n ^{-1 / 2} (\bI - \bP_{\bX_{\cS}})\bmu ^ *$. 
	For any $\cS \in \AA(s)$, we have 
	\beq
		\label{eq:rss}
		\begin{aligned}
		n^{-1} \cL_\mS & = n^{-1}\{\by^{\top}(\bI-\bP_{\bX_\mS})\by  \} = n^{-1} (\bbeta^{\ast\top}\bX^{\top}+\bepsilon^{\top})(\bI-\bP_{\bX_\mS})(\bX\bbeta^{\ast}+\bepsilon) \\
		& =\|\bgamma_\mS\|_2^2+\frac{2\bgamma_{\mS}^{\top}\bepsilon}{n^{1/2}}+\frac{1}{n}\bepsilon^{\top}(\bI-\bP_{\bX_\mS})\bepsilon.
		\end{aligned}
	\eeq
	Similarly, we have that 
	\beq
		\label{eq:rss_phis}
		\begin{aligned}
			n^{-1} \cL_{\Phi(\cS)} & =\|\bgamma_{\Phi(\cS)}\|_2^2+\frac{2\bgamma_{\Phi(\cS)}^{\top}\bepsilon}{n^{1/2}}+\frac{1}{n}\bepsilon^{\top}(\bI-\bP_{\bX_{\Phi(\cS)}})\bepsilon 
		\end{aligned}
	\eeq
	Combining the two displays above yields that 
	\beq
		\label{eq:loss_gap}
		\begin{aligned}
			n^{-1}(\cL_\mS-\cL_{\Phi(\cS)})  = &  \|\bgamma_\cS\|_2^2-   \|\bgamma_{\Phi(\cS) }\|_2^2  + \frac{2
			 (\bgamma_{\cS} - \bgamma_{\Phi(\cS)}) ^ {\top} \bepsilon}{n ^ {1 / 2}} - \frac{1}{n}\bepsilon^{\top}(\bP_{\bX_{\mS}}-\bP_{\bX_{\Phi(\cS)}})\bepsilon \\
                  = &  \eta (\|\bgamma_\cS\|_2^2-   \|\bgamma_{\Phi(\cS) }\|_2^2) \\
                   &+ 2^{-1}(1-\eta) (\|\bgamma_\cS\|_2^2-   \|\bgamma_{\Phi(\cS) }\|_2^2) + \frac{2
			 (\bgamma_{\cS} - \bgamma_{\Phi(\cS)}) ^ {\top} \bepsilon}{n ^ {1 / 2}}  \\
	           &+ 2^{-1}(1-\eta) (\|\bgamma_\cS\|_2^2-   \|\bgamma_{\Phi(\cS) }\|_2^2) - \frac{1}{n}\bepsilon^{\top}(\bP_{\bX_{\mS}}-\bP_{\bX_{\Phi(\cS)}})\bepsilon .
		 \end{aligned}
	\eeq
	We wish to show that the following two bounds hold with high probability: 
	\beq
		\label{eq:proof1_target1}		
		\inf_{\cS \in \AA(s)} \biggl\{2^{-1}(1-\eta) (\|\bgamma_\cS\|_2^2-   \|\bgamma_{\Phi(\cS) }\|_2^2) + \frac{2(\bgamma_{\cS} - \bgamma_{\Phi(\cS)}) ^ {\top} \bepsilon}{n ^ {1 / 2}}\biggr\}  > 0 
	\eeq
	and 
	\beq
		\label{eq:proof1_target2}		
		\inf_{\cS \in \AA(s)} \biggl\{2^{-1}(1-\eta) (\|\bgamma_\cS\|_2^2-   \|\bgamma_{\Phi(\cS) }\|_2^2) - \frac{1}{n}\bepsilon^{\top}(\bP_{\bX_{\mS}}-\bP_{\bX_{\Phi(\cS)}})\bepsilon \biggr\}> 0. 
	\eeq
	
	~\\
	
	We first consider \eqref{eq:proof1_target1}. Given that $\|\bepsilon\|_{\psi_2} \le \sigma$, applying Hoeffding's inequality yields that for any $x > 0$, 
	\beq
		\PP\bigl\{\bigl|(\bgamma_{\cS} - \bgamma_{\Phi(\cS)}) ^ {\top} \bepsilon\bigr| > x\sigma \ltwonorm{\bgamma_{\cS} - \bgamma_{\Phi(\cS)}}  \bigr\}\le 2e^{-x^2/2}.
	\eeq
	Note that $\AA_t(s) := \{\cS \in \AA(s): |\cS \backslash \Phi(\cS)| = t\}$ for $t \in [s]$. Then $\AA(s)=\cup_{t\in[s]}\AA_t(s)$.
	A union bound over $\cS \in \AA_t(s)$ yields that
	\beq
	       \label{equ:hoeffding_thm1}
			\PP\biggl\{\exists \cS \in \AA_{t}(s) \text{ s.t. }  \frac{2 |(\bgamma_{\cS} - \bgamma_{\Phi(\cS)}) ^ {\top} \bepsilon|}{n ^ {1 / 2}} \ge \frac{2x\sigma \ltwonorm{\bgamma_{\cS}-\bgamma_{\Phi(\cS)}}}{n ^ {1 / 2}}  \biggr\}
		     \le  2|\AA_t(s)|e ^ {-x ^ 2 / 2}.
	\eeq
	Writing $\xi_0 = \frac{(1-\eta)\psm_*(s)}{8\sigma(\log p)^{1/2}}$ and substituting $x=2\xi_0(t\log p)^{1/2}$ into \eqref{equ:hoeffding_thm1}, we have that 
	\beq
	\begin{aligned}
		\label{ineq:linear_term_bound}
	     &	\PP\biggl\{\exists \cS \in \AA_{t}(s)\text{ s.t. } \frac{2|(\bgamma_{\cS} - \bgamma_{\Phi(\cS)}) ^ {\top} \bepsilon|}{n ^ {1 / 2}} \ge 4\xi_0  \sigma\ltwonorm{\bgamma_{\cS}-\bgamma_{\Phi(\cS)}}\biggl(\frac{t \log p}{n}\biggr) ^ {1 / 2}  \biggr\}   \le  2|\AA_t(s)|e^{-2\xi_0^2t\log p}.
	 \end{aligned}
	\eeq
	Note that
	\[
	      \bgamma_{\cS} - \bgamma_{\Phi(\cS)} = n^{-1/2}(\bP_{\bX_{\Phi(\cS)}}-\bP_{\bX_{\cS}})\bmu^{\ast} = n^{-1/2}(\bP_{\Phi(\cS)|\cS}-\bP_{\cS|\Phi(\cS)})\bmu^{\ast}, 
	\]
	which implies that
	\begin{equation}\label{eq:gamma_linear1}
	    \| \bgamma_{\cS} - \bgamma_{\Phi(\cS)}\|_2 \le n^{-1/2} \bigl(\|\bP_{\Phi(\cS)|\cS}\bmu^{\ast}\|_2+\|\bP_{\cS|\Phi(\cS)}\bmu^{\ast}\|_2\bigr).
	\end{equation}
 	Consequently,
	\beq
	\begin{aligned}\label{eq:gamma_square1}
	      \|\bgamma_\cS\|_2^2 - \|\bgamma_{\Phi(\cS)}\|_2^2& = n^{-1}\bmu^{\ast\top}(\bP_{\bX_{\Phi(\cS)}}-\bP_{\bX_{\cS}})\bmu^{\ast}  = n^{-1} \bmu^{\ast\top}(\bP_{\Phi(\cS)|\cS}-\bP_{\cS|\Phi(\cS)})\bmu^{\ast}  \\
	      & = \frac{(\|\bP_{\Phi(\cS)|\cS}\bmu^{\ast}\|_2+\|\bP_{\cS|\Phi(\cS)}\bmu^{\ast}\|_2)}{n^{1/2}} \frac{(\|\bP_{\Phi(\cS)|\cS}\bmu^{\ast}\|_2-\|\bP_{\cS|\Phi(\cS)}\bmu^{\ast}\|_2)}{n^{1/2}} \\
	      & \ge \|\bgamma_\cS - \bgamma_{\Phi(\cS)} \|_2 \frac{(\|\bP_{\Phi(\cS)|\cS}\bmu^{\ast}\|_2-\|\bP_{\cS|\Phi(\cS)}\bmu^{\ast}\|_2)}{n^{1/2}}.
	\end{aligned}
	\eeq
	Combining \eqref{eq:gamma_square1} with the definition of $\xi_0$ and $\psm_*(s)$ 
	yields that
	\[
	     \frac{1-\eta}{2} (\|\bgamma_\cS\|_2^2 - \|\bgamma_{\Phi(\cS)}\|_2^2) \ge 4\xi_0  \sigma\ltwonorm{\bgamma_{\cS}-\bgamma_{\Phi(\cS)}}\biggl(\frac{t \log p}{n}\biggr) ^ {1 / 2} .
	\]
	Therefore, we deduce from \eqref{ineq:linear_term_bound} that 
	\[
	   \PP\left\{ \exists \cS\in\AA_t(s) \text{ s.t. } \frac{2 |(\bgamma_{\cS} - \bgamma_{\Phi(\cS)}) ^ {\top} \bepsilon|}{n ^ {1 / 2}} \ge  \frac{1-\eta}{2} (\|\bgamma_\cS\|_2^2 - \|\bgamma_{\Phi(\cS)}\|_2^2)   \right\} \le 2|\AA_t(s)|e^{-2\xi_0^2t\log p}.
	\]
	Note that 
	\[
	      |\AA_t(s)| = \binom{p-s ^ *}{t}\binom{s ^ *}{s -t} = \binom{p-s ^ *}{t}\binom{s ^ *}{s^* - s +t}.
	\]
	By Stirling's formula and the fact that $\log p \gtrsim s ^ *$, we have
	\[
	      \log \biggl\{\binom{s ^ *}{s^* - s +t}\biggr\} \lesssim \log \biggl\{\binom{s ^ *}{\lfloor s^* / 2\rfloor}\biggr\} \lesssim s ^ *\lesssim \log p. 
	\]
	Hence, we have
	\[
	\begin{aligned}
	      \PP\biggl\{ \exists \cS\in\AA_t(s) \text{ s.t. } \frac{2 |(\bgamma_{\cS} - \bgamma_{\Phi(\cS)}) ^ {\top} \bepsilon|}{n ^ {1 / 2}} & \ge  \frac{1-\eta}{2} (\|\bgamma_\cS\|_2^2 - \|\bgamma_{\Phi(\cS)}\|_2^2)   \biggr\} \\
	      & \le 2C_1p^{-(2\xi_0^2t - t -1)} \le 2C_1p^{-2t(\xi_0^2 -1)}.
	\end{aligned}
	\]
	A further union bound over $t \in[s]$ yields that
	\begin{equation}\label{eq:thm1_T1_1}
	    \PP\left\{ \exists \cS\in\AA(s) \text{ s.t. } \frac{2 |(\bgamma_{\cS} - \bgamma_{\Phi(\cS)}) ^ {\top} \bepsilon|}{n ^ {1 / 2}} \ge  \frac{1-\eta}{2} (\|\bgamma_\cS\|_2^2 - \|\bgamma_{\Phi(\cS)}\|_2^2)   \right\} \le 2C_1sp^{-2(\xi_0^{2}-1)}.
	\end{equation}
	
	Next we aim to show that (\ref{eq:proof1_target2}) is a high-probability event.
	Fix any $t\in[s]$. For any $\mS\in\AA_t(s)$, let $\mU,\mV$ be the orthogonal complement of $\mW:=\col(\bX_{\mS\cap{\Phi(\cS)}})$ as a subspace of $\mathrm{col}(\bX_\mS)$ and $\col(\bX_{{\Phi(\cS)}})$ respectively. Then dim$(\mU)=$dim$(\mV)=t$. We have
	\begin{align}
	\frac{1}{n}\bepsilon^{\top}(\bP_{\bX_\mS}-\bP_{\bX_{\Phi(\cS)}})\bepsilon & = \frac{1}{n}\bepsilon^{\top}(\bP_\mW+\bP_\mU)\bepsilon-\frac{1}{n}\bepsilon^{\top}(\bP_\mW+\bP_\mV)\bepsilon = \frac{1}{n}\bepsilon^{\top}(\bP_\mU-\bP_\mV)\bepsilon.
	\end{align}
	By Theorem 1.1 in \cite{rudelson2013hanson}, there exists a universal constant $c>0$ such that for any $x>0$,
	\begin{equation}
	\PP\left(|\bepsilon^{\top}\bP_\mU\bepsilon-\EE(\bepsilon^{\top}\bP_\mU\bepsilon)|>\sigma^2 x   \right) \le 2e^{-c\min(x^2/\|\bP_\mU\|_F^2,x/\|\bP_\mU\|_2)} = 2e^{-c\min(x^2/t,x)}.
	\end{equation}
	Similarly,
	\begin{equation}
		\PP\left(|\bepsilon^{\top}\bP_\mV\bepsilon-\EE(\bepsilon^{\top}\bP_\mV\bepsilon)|>\sigma^2 x   \right) \le  2e^{-c\min(x^2/t,x)}.
	\end{equation}
	Note that $\EE(\bepsilon^{\top}\bP_\mV\bepsilon)=\EE\text{tr}(\bP_\mV\bepsilon\bepsilon^{\top}) = \var(\epsilon_1)\text{tr}(\bP_\mV) = t\var(\epsilon_1)=\EE(\bepsilon^{\top}\bP_\mU\bepsilon)$. Combining the above two inequalities yields
	\begin{equation}
		\label{eq:quad}	
		\PP(|\bepsilon^{\top}\bP_\mU\bepsilon-\bepsilon^{\top}\bP_\mV\bepsilon| >2\sigma^2x )\le 4e^{-c\min(x^2/t,x)}.
	\end{equation}
	If we have $\xi_0\ge (16 \log 2)^{-1/2}$, then applying a union bound over $\mS\in\AA_t(s)$ and taking $x =16\xi_0^2 t \log p $ yields that
	\beq
		\begin{aligned}
			\PP\biggl\{{\exists\mS \in \AA_t(s)} & \text{ s.t. }   \frac{|\bepsilon^{\top}\bP_{\bX_\mS}\bepsilon-\bepsilon^{\top}\bP_{\bX_{\Phi(\cS)}}\bepsilon|}{n} >32\xi_0^2\sigma^2 t\left( \frac{\log p}{n}\right) \biggr\} \\
			&\le  4C_1p^{-(16c\xi_0^{2}t-t-1)}  \le 4C_1p^{-2t(8c\xi_0^{2}-1)} .
		\end{aligned}
	\eeq
	Note that
	\[
	\begin{aligned}
	    \frac{1-\eta}{2} (\|\bgamma_\cS\|_2^2-\|\bgamma_{\Phi(\cS)}\|_2^2) & \ge \frac{1-\eta}{2n} (\|\bP_{\Phi(\cS)|\cS}\bmu^*\|_2-\|\bP_{\cS|\Phi(\cS)}\bmu^*\|_2)(\|\bP_{\Phi(\cS)|\cS}\bmu^*\|_2+\|\bP_{\cS|\Phi(\cS)}\bmu^*\|_2) \\
	    & \ge \frac{1-\eta}{2n}(\|\bP_{\Phi(\cS)|\cS}\bmu^*\|_2-\|\bP_{\cS|\Phi(\cS)}\bmu^*\|_2)^2 \ge 32\xi_0^2 \sigma^2 t\left(\frac{\log p}{n}\right).
	\end{aligned}
	\]
A further union bound over $t \in [s]$ yields
	\beq
		\label{eq:thm1_T2_1}
		\begin{aligned}
			\PP\biggl\{\exists\cS\in\AA(s) \text{ s.t. }  \frac{|\bepsilon^{\top}\bP_{\bX_\mS}\bepsilon-\bepsilon^{\top}\bP_{\bX_{\Phi(\cS)}}\bepsilon|}{n}> \frac{1-\eta}{2} (\|\bgamma_\cS\|_2^2-\|\bgamma_{\Phi(\cS)}\|_2^2)\biggr\}  \le 4C_1sp^{-2(8c\xi_0^2 -1)}. 
		\end{aligned}
	\eeq
	
	Finally, let $C=\max\{1,(8c)^{-1/2}, 6C_1\}$. Combining \eqref{eq:loss_gap}, \eqref{eq:thm1_T1_1} and \eqref{eq:thm1_T2_1} yields that 
	\[
		\PP\biggl[\inf_{\cS \in \AA(s)}\biggl\{\frac{1}{n}(\cL_{\cS} - \cL_{\Phi(\cS)}) - \eta(\ltwonorm{\bgamma_{\cS}} ^ 2 - \ltwonorm{\bgamma_{\Phi(\cS)}} ^ 2)\biggr\} \ge 0\biggr] \ge 1 - Csp^{-2(C^{-2}\xi_0^2 -1)}.
	\]
	Note that
	\[
		\eta(\|\bgamma_\cS\|_2^2-\|\bgamma_{\Phi(\cS)}\|_2^2) \ge \frac{\eta t \psm_*^2(s)}{n}\quad \text{ and }\quad \cL_{\cS}-\cL_* \ge \cL_{\cS}-\cL_{\Phi(\cS)}
	\]
	for all $\cS\in\AA(s)$. For any $\xi>C$, the conclusion then follows immediately if $\xi_0 \ge \xi$.


\subsection{Proof of Theorem 2.2}

	We wish to show that if (13) in the article is satisfied, then there exists $j\in \cJ_{\delta_0}$ such that $ \cL_{\{j\}} < \hat \cL_{\dagger} := \min_{j ^ * \in \cS ^ *}{\cL_{\{j ^ *\}}} $ with high probability. To see this, for any $j ^ * \in \cS ^ *$ and any $j \in \cJ_{\delta_0}$, we have that
	\beq
		\label{eq:rs_margin}
		\begin{aligned}
				n^{-1}(\cL_{\{j\}}-\cL_{\{j ^ *\}}) & = \|\bgamma_{\{j\}}\|_2^2 - \ltwonorm{\bgamma_{\{j ^ *\}}} ^ 2 + \frac{2(\bgamma_{\{j\}} - \bgamma_{\{j ^ *\}}) ^ {\top} \bepsilon}{n ^ {1 / 2}} - \frac{1}{n}\bepsilon^{\top}(\bP_{\bX_{j}}-\bP_{\bX_{j^{*}}})\bepsilon \\
				& \le \|\bgamma_{\{j\}}\|_2^2 - \ltwonorm{\bgamma_{\{j ^ \dagger\}}} ^ 2 + \frac{2(\bgamma_{\{j\}} - \bgamma_{\{j ^ *\}}) ^ {\top} \bepsilon}{n ^ {1 / 2}} - \frac{1}{n}\bepsilon^{\top}(\bP_{\bX_{j}}-\bP_{\bX_{j^{*}}})\bepsilon. 
		\end{aligned}
	\eeq
	Recall that  $\bar \bu_j := \bX_j / \ltwonorm{\bX_j}, \forall j \in [p]$. 
	For convenience, write 
	\[
		\Delta = \max_{j \in \cJ_{\delta_0}} \{\ltwonorm{\bP_{\bX_{j^{\dagger}}} \bmu ^ *} - \ltwonorm{\bP_{\bX_j} \bmu ^ *}\} = \max_{j \in \cJ_{\delta_0}} \{|\bar \bu_{j^{\dagger}} ^ {\top} \bmu ^ *| - |\bar \bu_{j} ^ {\top} \bmu ^ *|\}. 
	\]
	We then have that 
	\beq
		\label{eq:gamma_gap1}
		\bgamma_{\{j\}} - \bgamma_{\{j ^*\}} = n ^ {- 1 / 2}(\bP_{\bX_{j^{*}}} - \bP_{\bX_{j}})\bmu ^ * = n ^ {-1 / 2} (\bar\bu_{j^{*}}\bar\bu_{j^{*}} ^ {\top} - \bar\bu_{j}\bar\bu^{\top}_{j}) \bmu ^ *, 
	\eeq
	and that 
	\beq
		\label{eq:quad_gamma_gap11}
		\begin{aligned}
			\ltwonorm{\bgamma_{\{j\}}} ^ 2 - \ltwonorm{\bgamma_{\{j ^ \dagger\}}} ^ 2 & = \frac{1}{n} \bmu ^ {*\top}(\bP_{\bX_{j ^ \dagger}} - \bP_{\bX_{j}})\bmu ^ * = \frac{1}{n}\bmu ^ {\ast \top}(\bar\bu_{j^\dagger}\bar\bu_{j^{\dagger}} ^ {\top} - \bar\bu_{j}\bar\bu^{\top}_{j})\bmu ^ * \\
			& = \frac{1}{n}(|\bar \bu_{j^{\dagger}} ^ {\top} \bmu ^ *| - |\bar \bu_{j} ^ {\top} \bmu ^ *|)(|\bar \bu_{j^\dagger} ^ {\top} \bmu ^ *| + |\bar \bu_{j} ^ {\top} \bmu ^ *|) \le \frac{2}{n}\Delta |\bar \bu_{j^{\dagger}} ^ {\top} \bmu ^ *|. 
		\end{aligned}
	\eeq
	Note that Assumption 2.2 implies that $|\bar \bu^{\top}_{j ^ \dagger}{\bmu ^ *}| \ge 4 \Delta$; we thus deduce from \eqref{eq:quad_gamma_gap11} that 
	\beq
		\label{eq:uj0_lower_bound1}
		\begin{aligned}
			(\bar\bu_j ^ {\top}\bmu ^ *) ^ 2 & \ge (\bar\bu_{j^{\dagger}} ^ {\top}\bmu ^ *) ^ 2 - 2\Delta|\bar \bu_{j^{\dagger}} ^ {\top} \bmu ^ *| \ge \frac{(\bar\bu_{j^{\dagger}} ^ {\top}\bmu ^ *) ^ 2}{2}, \forall j \in \cJ_{\delta_0}. 
		\end{aligned}
	\eeq  
	By Lemma 6.1 of \cite{guo2020best}, \eqref{eq:uj0_lower_bound1} and then (13 in the article), we have for any $j, k \in \cJ_{\delta_0}$ and $j \neq k$ that
	\[
		\begin{aligned}
			\ltwonorm{\bgamma_{\{j\}} - \bgamma_{ \{k\}}}  = n ^ {- 1/ 2}\ltwonorm{\bar\bu_j\bar\bu_j ^{\top} \bmu ^ * - \bar\bu_k\bar\bu_k ^{\top} \bmu ^ *} & \ge n ^ {-1 / 2}\min(|\bar \bu_j ^ {\top}\bmu ^ *|, |\bar \bu_k ^ {\top}\bmu ^ *|)\delta_0 \\
			& \ge \delta_0 (2n) ^ {- 1 /2}|\bmu ^ {\ast \top} \bar \bu_{j^{\dagger}}| =: \delta_0'. 
		\end{aligned}
	\]
	Therefore, $\{\bgamma_{\{j\}} - \bgamma_{\{j ^ *\}}\}_{j \in \cJ_{\delta_0}}$ is a $\delta'_0$-packing set of itself. By Sudakov's lower bound \citep[][Theorem 5.30]{Wai2019}, we deduce that
	\beq
		\label{ineq:sudakov_1}
		\E\biggl\{ \min_{j \in \cJ_{\delta_0}} \frac{2\bigl(\bgamma_{\{j\}} - \bgamma_{\{j ^ *\}}\bigr) ^ {\top} \bepsilon}{n ^ {1 / 2}}\biggr\} \le - \delta_0' \sigma \biggl(\frac{c_{\delta_0}\log p}{n}\biggr) ^ {1 / 2}. 
	\eeq
	Furthermore, \eqref{eq:gamma_gap1} and (13) in the article imply that
	\[
		\ltwonorm{\bgamma_{\{j\}} - \bgamma_{\{j^ {*}\}}} \le n ^ {- 1 / 2}(|\bar \bu_{j^*} ^ {\top} \bmu ^ *| + |\bar \bu_j ^ {\top} \bmu ^ *|) \le 2n ^ {-1 / 2}|\bar\bu^{\top}_{j^\dagger}\bmu ^ *|. 
	\] 
	Therefore, by Lemma \ref{lem:var_sup_process}, 
	\beq
		\label{ineq:var_sup_process_1}
		\biggl\|\min_{j \in \cJ_{\delta_0}} \frac{2(\bgamma_{\{j\}} - \bgamma_{\{j ^ *\}}) ^ {\top} \bepsilon}{n ^ {1 / 2}}\biggr\|_{\psi_2} \lesssim \frac{\sigma ^ 2\max_{j \in \cJ_{\delta_0}}\ltwonorm{\bgamma_{\{j\}} - \bgamma_{\{j ^ *\}}} ^ 2}{n} \lesssim \frac{\sigma ^ 2({\bar \bu_{j^{\dagger}} ^ {\top} \bmu ^ *})^ 2}{n ^ 2}. 
	\eeq
	Combining \eqref{ineq:sudakov_1} and \eqref{ineq:var_sup_process_1}, we deduce that there exists a universal constant $C > 0$, such that for any $t > 0$, 
	\[
		\PP \biggl\{\min_{j \in \cJ_{\delta_0}} \frac{2(\bgamma_{\{j \}} - \bgamma_{\{j ^ *\}}) ^ {\top} \bepsilon}{n ^ {1 / 2}} \ge  - \delta_0' \sigma \biggl(\frac{c_{\delta_0}\log p}{n}\biggr) ^ {1 / 2} + \frac{C t\sigma|{\bar \bu_{j^{\dagger}} ^ {\top} \bmu ^ *}|}{n}\biggr\} \le \exp(-t ^ 2). 
	\]
 	Choosing $t = 2 ^ {-3 / 2} C ^ {-1} \delta_0c ^ {1 / 2}_{\delta_0} \log ^ {1 /2} p$, we then reduce the bound above to 
	\beq
		\label{eq:nec_c1_t2_1}
		\PP \biggl\{\min_{j \in \cJ_{\delta_0}} \frac{2(\bgamma_{\{ j \}} - \bgamma_{\{ j ^ *\}}) ^ {\top} \bepsilon}{n ^ {1 / 2}} \ge  - \frac{\delta_0 c^ {1/ 2}_{\delta_0}\sigma |{\bar \bu_{j^{\dagger}} ^ {\top} \bmu ^ *}|}{2 ^ {3 / 2}}\frac{\log ^ {1 / 2}p}{n}\biggr\} \le p ^ {- \delta_0 ^ 2c_{\delta_0} / (8C^2)}. 
	\eeq
	Besides, combining Assumption 2.2 with the fact that $\delta_0 ^ 2 c_{\delta_0} \log p > 1$ yields that
	\[
		|\bar\bu^{\top}_{j ^ \dagger}\bmu ^ *| > 8\sigma /\{\delta_0c^{1 / 2}_{\delta_0}(\log p) ^ {1 / 2}\}. 
	\] 
	Therefore, applying a union bound over $j \in \cJ_{\delta_0}$ to \eqref{eq:quad} with $x = (8\sigma) ^ {-1}\delta_0c^{1 / 2}_{\delta_0}|\bar\bu_{j ^ {\dagger}}^\top\bmu ^ *|\log^{1 / 2} p$, we deduce by Assumption 2.2 that 
	\beq
		\label{eq:nec_c1_t3_1}	
		\begin{aligned}
			\PP\biggl\{\max_{j \in \cJ_{\delta_0}} \frac{|\bepsilon^{\top}\bP_{\bX_j}\bepsilon-\bepsilon^{\top}\bP_{\bX_{j^{*}}}\bepsilon|}{n}> \frac{\delta_0 c^ {1/ 2}_{\delta_0}\sigma |{\bar \bu_{j^{\dagger}} ^ {\top} \bmu ^ *}|\log ^ {1 / 2}p}{4n} \biggr\}  \le 4p ^ {-(\xi c\delta ^ 2_0 / 8 - 1)c_{\delta_0}}. 
		\end{aligned}
	\eeq
	Finally, note that 	
	\beq
		\label{eq:case1_1}
		\begin{aligned}
			\min_{j \in \cJ_{\delta_0}} & n^{-1}(\cL_{\{j\}}-\cL_{\{j^{*}\}}) \\
			& \le \max_{j\in \cJ_{\delta_0} } \bigl(\|\bgamma_{\{j\}}\|_2^2 - \ltwonorm{\bgamma_{\{j ^ \dagger\}}} ^ 2\bigr) + \min_{j \in \cJ_{\delta_0} } \frac{2(\bgamma_{\{j\}} - \bgamma_{\{j^{*}\}}) ^ {\top} \bepsilon}{n ^ {1 / 2}} + \max_{j \in \cJ_{\delta_0}}\frac{1}{n}\bepsilon^{\top}(\bP_{\bX_{j}}-\bP_{\bX_{j^{*}}})\bepsilon. 
		\end{aligned}
	\eeq	
	Combining \eqref{eq:case1_1}, \eqref{eq:nec_c1_t2_1}, \eqref{eq:nec_c1_t3_1} and \eqref{eq:quad_gamma_gap11} yields that when $\Delta < \delta_0 \sigma (c_{\delta_0}\log p) ^ {1 / 2} / 20$,
	\beq
		\PP\biggl(\min_{j \in \cJ_{\delta_0}}  \cL_{\{j\}} < \cL_{\{j^{*}\}} \biggr) \ge 1 - 4p ^ {-(\xi c\delta ^ 2_0 / 8 - 1)c_{\delta_0}} - p ^ {- \delta_0 ^ 2 c_{\delta_0} / (8C^2)}. 
	\eeq
	The conclusion immediately follows once we apply a union bound over $j ^ * \in \cS ^ *$.


\subsection{Proof of Theorem 2.3}
	An important observation is that $\Phi(\cS) = \cS ^ \dagger$ and 
	\[
	     \psm(\cS) = \ltwonorm{\bP_{{\cS ^ {\dagger} | \cS}} \bmu ^ *} - \ltwonorm{\bP_{{\cS | \cS ^ {\dagger}}} \bmu ^ *}
	\]
	for any $\cS \in \AA_{j_0}$. This motivates us to take the following two main steps to establish the theorem: (i) we show that $\cL_{\cS ^ {\dagger}}$ is the smallest among $\{\cL_{\cS}\}_{\cS \in \AA ^ *(s )}$ with high probability; (ii) we show that $\min_{\cS \in \AA_{j_0}} \cL_{\cS}< \cL_{\cS ^ \dagger}$, which implies that $\cS ^ \dagger$ is not the best subset any more, and thus that the best subset must have false discoveries. 
	
	{\bf Step (i).} We aim to show that 
    \beq
    \label{eq:s_dagger_optimal}
    \PP\biggl(\min_{\cS \in \AA ^ *(s)} \cL_{\cS} - \cL_{\cS ^ {\dagger}} \le 0\biggr) \le {4sp^{-(c\xi^2/4-2)} + 2sp^ {-(\xi ^ 2/32-2) }}. 
    \eeq
	This step follows closely the proof strategy of Theorem 2.1. 
    For any $\cS\in\AA^*(s)$, 
    \beq
    \begin{aligned}
    	n^{-1}(\cL_\cS-\cL_{\cS^{\dagger}})   = & \|\bgamma_{\mS}\|_2^2 - \ltwonorm{\bgamma_{\mS ^ {\dagger}}} ^ 2 + \frac{2(\bgamma_{\cS} - \bgamma_{\cS ^ {\dagger}}) ^ {\top} \bepsilon}{n ^ {1 / 2}} - \frac{1}{n}\bepsilon^{\top}(\bP_{\bX_{\mS}}-\bP_{\bX_{\mS^{\dagger}}})\bepsilon \\
    	=& \frac{\|\bgamma_{\mS}\|_2^2 - \ltwonorm{\bgamma_{\mS ^ {\dagger}}} ^ 2 }{2}  +  \frac{2(\bgamma_{\cS} - \bgamma_{\cS ^ {\dagger}}) ^ {\top} \bepsilon}{n ^ {1 / 2}} \\
    	& + \frac{\|\bgamma_{\mS}\|_2^2 - \ltwonorm{\bgamma_{\mS ^ {\dagger}}} ^ 2 }{2}   - \frac{1}{n}\bepsilon^{\top}(\bP_{\bX_{\mS}}-\bP_{\bX_{\mS^{\dagger}}})\bepsilon.
    \end{aligned}
    \eeq
    We wish to show that
	\beq
    	\inf_{\cS \in \AA^ *(s )}\biggl\{\frac{\|\bgamma_{\mS}\|_2^2 - \ltwonorm{\bgamma_{\mS ^ {\dagger}}} ^ 2 }{2}  +  \frac{2(\bgamma_{\cS} - \bgamma_{\cS ^ {\dagger}}) ^ {\top} \bepsilon}{n ^ {1 / 2}}\biggr\} >0 \label{eq:theorem2_T1} 
    \eeq
    and
    \beq
    	\inf_{\cS \in \AA ^ *(s )}\biggl\{\frac{\|\bgamma_{\mS}\|_2^2 - \ltwonorm{\bgamma_{\mS ^ {\dagger}}} ^ 2 }{2}   - \frac{1}{n}\bepsilon^{\top}(\bP_{\bX_{\mS}}-\bP_{\bX_{\mS^{\dagger}}})\bepsilon\biggr\} >0 \label{eq:theorem2_T2}
	\eeq
    with high probability in the sequel. Define $\AA_t^*(s) = \{\cS\in\AA^*(s): |\cS\setminus\cS^\dagger| = t\}$. 
 	To prove (\ref{eq:theorem2_T1}), first fix some $\cS\in\AA_t^*(s)$. By Hoeffding's inequality, we have for any $x>0$ that
    \[
    \PP\left\{ |(\bgamma_\cS-\bgamma_{\cS^{\dagger}})^{\top}\bepsilon| > x\sigma\|\bgamma_\cS-\bgamma_{\cS^{\dagger}} \|_2 \right\} \le 2e^{-x^2/2}.
    \]
    A union bound over $\cS\in\AA_t^*(s)$ yields that
    \[
    \begin{aligned}
    	\PP\biggl\{\exists \cS \in \AA_t^*(s) & \text{ s.t. } \frac{2 |(\bgamma_{\cS} - \bgamma_{\cS^{\dagger}}) ^ {\top} \bepsilon|}{n ^ {1 / 2}} \ge \frac{2x\sigma \ltwonorm{\bgamma_{\cS}-\bgamma_{\cS^{\dagger}}}}{n ^ {1 / 2}}  \biggr\} \\
    &   \le  2|\AA_t^*(s)|e ^ {-x ^ 2 / 2} = 2 \binom{s}{t}\binom{s^*-s}{t} \le  2pe ^ {-(x ^ 2 / 2-2t)},
    \end{aligned}
    \]
    Let $x=\xi(t\log p)^{1/2}/4$. Then we have that
    \[
    \begin{aligned}
    	\PP\biggl\{ \exists \cS\in\AA_t^*(s) \text{ s.t. } \frac{2 |(\bgamma_{\cS} - \bgamma_{\cS^{\dagger}}) ^ {\top} \bepsilon|}{n ^ {1 / 2}} \ge \frac{\xi\sigma \|\bgamma_\cS - \bgamma_{\cS^{\dagger}} \|_2 }{2}\left(\frac{t\log p}{n}\right)^{1/2}  \biggr\}  \le 2p^{-t(\xi^2/32-2)}.
    \end{aligned}
    \]
    Note that
    \beq
    \begin{aligned}\label{eq:thm2_gamma_square}
    	\|\bgamma_\cS\|_2^2 - \|\bgamma_{\cS^{\dagger}}\|_2^2& = n^{-1}\bmu^{\ast\top}(\bP_{\bX_{\cS^{\dagger}}}-\bP_{\bX_{\cS}})\bmu^{\ast}  \\
    	& \ge \|\bgamma_\cS - \bgamma_{\cS^{\dagger}} \|_2 \frac{\|\bP_{\cS^{\dagger}|\cS}\bmu^{\ast}\|_2-\|\bP_{\cS|\cS^{\dagger}}\bmu^{\ast}\|_2}{n^{1/2}} \\
    	& \ge \|\bgamma_\cS - \bgamma_{\cS^{\dagger}} \|_2 \xi \sigma\left(\frac{t\log p}{n}\right)^{1/2}.
    \end{aligned} 	
    \eeq
    Further apply a union bound for $t\in[s]$. We thus have that
    \begin{equation}\label{eq:theorem2_T1_result}
    	\PP\biggl\{ \exists \cS\in\AA^*(s) \text{ s.t. } \frac{2 |(\bgamma_{\cS} - \bgamma_{\cS^{\dagger}}) ^ {\top} \bepsilon|}{n ^ {1 / 2}} \ge  \frac{ \|\bgamma_\cS\|_2^2 - \|\bgamma_{\cS^{\dagger}}\|_2^2  }{2} \biggr\} 
    	\le  2sp^{-(\xi^2/32-2)}.
    \end{equation}

    Next we show that (\ref{eq:theorem2_T2}) holds with high probability.
    Similarly to \eqref{eq:quad}, we can obtain that 
    \begin{equation}
    	\label{eq:quad2}
    	\PP(|\bepsilon^{\top}\bP_\mU\bepsilon-\bepsilon^{\top}\bP_\mV\bepsilon| >2\sigma^2x )\le 4e^{-c\min(x^2/t,x)}.
    \end{equation}
    Note that $\log p>1$ and $\xi > 2$. By taking $x =t\xi^2 \log p/4 $, applying a union bound over $\mS\in\AA_t^*(s)$ yields that 
    \beq
    \begin{aligned}
    	\PP\biggl\{{\exists\mS \in \AA_t^*(s)} \text{ s.t. }  \frac{|\bepsilon^{\top}\bP_{\bX_\mS}\bepsilon-\bepsilon^{\top}\bP_{\bX_{\cS^{\dagger}}}\bepsilon|}{n}>\frac{t \xi^2\sigma^2\log p}{2n} \biggr\}
    	\le 4p^{-t(c\xi^{2}/4-2)}  .
    \end{aligned}
    \eeq
    Note that
    \[
	    \begin{aligned}
	    	\frac{\|\bgamma_\cS\|_2^2-\|\bgamma_{\cS^{\dagger}}\|_2^2}{2} & =\frac{(\ltwonorm{\bP_{\cS^{\dagger}|\cS}\bmu ^ *}-\ltwonorm{\bP_{\cS|\cS^{\dagger}}\bmu ^ *})(\ltwonorm{\bP_{\cS^{\dagger}|\cS}\bmu ^ *}+\ltwonorm{\bP_{\cS|\cS^{\dagger}}\bmu ^ *})}{2n}\\ 
	    	&  \ge \frac{(\ltwonorm{\bP_{\cS^{\dagger}|\cS}\bmu ^ *}-\ltwonorm{\bP_{\cS|\cS^{\dagger}}\bmu ^ *})^2}{2n}  \ge \frac{t\xi^2\sigma^2\log p}{2n}.
	    \end{aligned}
    \]
    Then with a union bound over $t\in[s]$, it follows that 
    \beq
    \label{eq:theorem2_T2_result}
    \begin{aligned}
    	\PP\biggl(\exists\cS\in\AA^*(s) \text{ s.t. }  \frac{|\bepsilon^{\top}\bP_{\bX_\mS}\bepsilon-\bepsilon^{\top}\bP_{\bX_{\mS^{\dagger}}}\bepsilon|}{n}> \frac{\|\bgamma_\cS\|_2^2-\|\bgamma_{\cS^{\dagger}}\|_2^2}{2}\biggr)  \le 4sp^{-(c\xi^{2}/4-2)} . 
    \end{aligned}
    \eeq
    Combining (\ref{eq:theorem2_T1_result}) and (\ref{eq:theorem2_T2_result}) yields (\ref{eq:s_dagger_optimal}). \\
    
    {\bf Step (ii).} We wish to show that if (14) in the article  is satisfied, then $\min_{\cS\in \AA_{j_0}}\cL_{\cS} < \cL_{\cS ^ {\dagger}}$. Note that for any $\cS \in \AA_{j_0}$, 
	\beq
		\begin{aligned}
				n^{-1}(\cL_\mS-\cL_{\mS^{\dagger}}) & = \|\bgamma_{\mS}\|_2^2 - \ltwonorm{\bgamma_{\mS ^ {\dagger}}} ^ 2 + \frac{2(\bgamma_{\cS} - \bgamma_{\cS ^ {\dagger}}) ^ {\top} \bepsilon}{n ^ {1 / 2}} - \frac{1}{n}\bepsilon^{\top}(\bP_{\bX_{\mS}}-\bP_{\bX_{\mS^{\dagger}}})\bepsilon. 
		\end{aligned}
	\eeq
	Denote the only element of $\cS \backslash \cS ^{\dagger}_0$ by $j$. We then have that 
	\beq
		\label{eq:gamma_gap2}
		\bgamma_{\cS} - \bgamma_{\cS ^ {\dagger}} = n ^ {- 1 / 2}(\bP_{\bX_{\cS ^ \dagger}} - \bP_{\bX_{\cS}})\bmu ^ * = n ^ {-1 / 2} (\bar\bu_{j_0}\bar\bu_{j_0} ^ {\top} - \bar\bu_{j}\bar\bu^{\top}_{j}) \bmu ^ *, 
	\eeq
	and that 
	\beq
		\label{eq:quad_gamma_gap2}
		\begin{aligned}
			\ltwonorm{\bgamma_{\cS}} ^ 2 - \ltwonorm{\bgamma_{\cS ^ {\dagger}}} ^ 2 & = \frac{1}{n} \bmu ^ {*\top}(\bP_{\bX_{\cS ^ {\dagger}}} - \bP_{\bX_{\cS}})\bmu ^ * = \frac{1}{n}\bmu ^ {\ast \top}(\bar\bu_{j_0}\bar\bu_{j_0} ^ {\top} - \bar\bu_{j}\bar\bu^{\top}_{j})\bmu ^ * \\
			& = \frac{1}{n}(|\bar \bu_{j_0} ^ {\top} \bmu ^ *| - |\bar \bu_{j} ^ {\top} \bmu ^ *|)(|\bar \bu_{j_0} ^ {\top} \bmu ^ *| + |\bar \bu_{j} ^ {\top} \bmu ^ *|) \le \frac{2}{n}\Delta |\bar \bu_{j_0} ^ {\top} \bmu ^ *|, 
		\end{aligned}
	\eeq
	where $\Delta := \max_{\cS \in \AA_{j_0}} \{\ltwonorm{\bP_{{\cS ^ {\dagger} | \cS}} \bmu ^ *} - \ltwonorm{\bP_{{\cS | \cS ^ {\dagger}}} \bmu ^ *}\} = \max_{k \in \cJ_{\delta_0}} \{|\bar \bu_{j_0} ^ {\top} \bmu ^ *| - |\bar \bu_k ^ {\top} \bmu ^ *|\}$.	
	Assumption \ref{ass:residual_margin_s} implies that $|{\bar \bu_{j_0} ^ {\top} \bmu ^ *}| >4\Delta$. 
	Therefore, 
	we have that 
	\beq
		\label{eq:uj0_lower_bound}
		\begin{aligned}
			(\bar\bu_j ^ {\top}\bmu ^ *) ^ 2 & > (\bar\bu_{j_0} ^ {\top}\bmu ^ *) ^ 2 - 2\Delta|\bar \bu_{j_0} ^ {\top} \bmu ^ *| > \frac{(\bar\bu_{j_0} ^ {\top}\bmu ^ *) ^ 2}{2}, \forall j \in \cJ_{\delta_0}. 
		\end{aligned}
	\eeq  
	By Lemma 6.1 of \cite{guo2020best} and then \eqref{eq:uj0_lower_bound}, we have for any $j, k \in \cJ_{\delta_0}$ and $j \neq k$ that
	\[
		\begin{aligned}
			\ltwonorm{\bgamma_{\cS ^ {\dagger}_0 \cup \{j\}} - \bgamma_{\cS ^ {\dagger}_0 \cup \{k\}}}  = n ^ {- 1/ 2}\ltwonorm{\bar\bu_j\bar\bu_j ^\top \bmu ^ * - \bar\bu_k\bar\bu_k ^\top \bmu ^ *} & \ge n ^ {-1 / 2}\min(|\bar \bu_j ^ {\top}\bmu ^ *|, |\bar \bu_k ^ {\top}\bmu ^ *|)\delta_0 \\
			& \ge \delta_0 (2n) ^ {- 1 /2}|\bmu ^ {\ast \top} \bar \bu_{j_0}| =: \delta_0'. 
		\end{aligned}
	\]
	Therefore, $\{\bgamma_{\cS} - \bgamma_{\cS ^ {\dagger}}\}_{\cS \in \AA_{j_0}}$ is a $\delta'_0$-packing set of itself. By Sudakov's lower bound \citep[][Theorem 5.30]{Wai2019}, we deduce that
	\beq
		\label{ineq:sudakov_2}
		\E\biggl\{ \min_{\cS \in \AA_{j_0}} \frac{2(\bgamma_{\cS} - \bgamma_{\cS ^ {\dagger}}) ^ {\top} \bepsilon}{n ^ {1 / 2}}\biggr\} \le - \delta_0' \sigma \biggl(\frac{c_{\delta_0}\log p}{n}\biggr) ^ {1 / 2}. 
	\eeq
	Furthermore, \eqref{eq:gamma_gap2} and (14) in the article imply that
	\[
		\ltwonorm{\bgamma_{\cS} - \bgamma_{\cS ^ {\dagger}}} \le n ^ {- 1 / 2}(|\bar \bu_{j_0} ^ {\top} \bmu ^ *| + |\bar \bu_j ^ {\top} \bmu ^ *|) \le 2n ^ {-1 / 2}|\bar\bu^{\top}_{j_0}\bmu ^ *|. 
	\] 
	Therefore, applying Lemma \ref{lem:var_sup_process} yields that 
	\beq
		\label{ineq:var_sup_process_2}
		\biggl\|\min_{\cS \in \AA_{j_0}} \frac{2(\bgamma_{\cS} - \bgamma_{\cS ^ {\dagger}}) ^ {\top} \bepsilon}{n ^ {1 / 2}}\biggr\|_{\psi_2} \lesssim \frac{\sigma ^ 2\max_{\cS \in \AA_{j_0}}\ltwonorm{\bgamma_{\cS} - \bgamma_{\cS ^{\dagger}}} ^ 2}{n} \lesssim \frac{\sigma ^ 2({\bar \bu_{j_0} ^ {\top} \bmu ^ *})^ 2}{n ^ 2}. 
	\eeq
	Combining \eqref{ineq:sudakov_2} and \eqref{ineq:var_sup_process_2}, we deduce that there exists a universal constant $C > 0$, such that for any $t > 0$, 
	\[	
		\PP \biggl\{\min_{\cS \in \AA_{j_0}} \frac{2(\bgamma_{\cS} - \bgamma_{\cS ^ {\dagger}}) ^ {\top} \bepsilon}{n ^ {1 / 2}} \ge  - \delta_0' \sigma \biggl(\frac{c_{\delta_0}\log p}{n}\biggr) ^ {1 / 2} + \frac{C t\sigma|{\bar \bu_{j_0} ^ {\top} \bmu ^ *}|}{n}\biggr\} \le \exp(-t ^ 2). 
	\]
	Choosing $t = 2 ^ {-3 / 2} C ^ {-1} \delta_0c ^ {1 / 2}_{\delta_0} \log ^ {1 /2} p$, we then reduce the bound above to 
	\beq
		\label{ineq:epsilon_linear_gap2}
		\PP \biggl\{\min_{\cS \in \AA_{j_0}} \frac{2(\bgamma_{\cS} - \bgamma_{\cS ^ \dagger}) ^ {\top} \bepsilon}{n ^ {1 / 2}} \ge  - \frac{\delta_0 c^ {1/ 2}_{\delta_0}\sigma |{\bar \bu_{j_0} ^ {\top} \bmu ^ *}|}{2 ^ {3 / 2}}\frac{\log ^ {1 / 2}p}{n}\biggr\} \le p ^ {- \delta_0 ^ 2c_{\delta_0} / (8C^2)}. 
	\eeq
	Besides, Assumption 2.4 yields that $|\bar\bu^{\top}_{j_0}\bmu ^ *| > 8\sigma /\{\delta_0c^{1 / 2}_{\delta_0}(\log p) ^ {1 / 2}\}$. Therefore, applying a union bound over $\cS \in \AA_{j_0}$ to \eqref{eq:quad} with $x = (8\sigma) ^ {-1}\delta_0c^{1 / 2}_{\delta_0}|\bar\bu_{j_0}^\top\bmu ^ *|\log^{1 / 2} p$, we obtain that
	\beq
	\label{ineq:epsilon_quad_gap2}
	\begin{aligned}
		\PP\biggl\{\max_{\cS \in \AA_{j_0}} \frac{|\bepsilon^{\top}\bP_{\bX_{\cS}}\bepsilon-\bepsilon^{\top}\bP_{\bX_{\cS ^ {\dagger}}}\bepsilon|}{n}> \frac{\delta_0 c^ {1/ 2}_{\delta_0}\sigma |{\bar \bu_{j_0} ^ {\top} \bmu ^ *}|\log ^ {1 / 2}p}{4n} \biggr\}  \le 4p ^ {-(\xi c\delta ^ 2_0 / 8 - 1)c_{\delta_0}}. 
	\end{aligned}
	\eeq
Finally, note that 	
\beq
	\begin{aligned}
		\min_{\cS \in \AA_{j_0}} & n^{-1}(\cL_{\cS}-\cL_{\cS ^ \dagger}) \\
		& \le \max_{\cS \in \AA_{j_0}} \bigl(\|\bgamma_{\cS}\|_2^2 - \ltwonorm{\bgamma_{\cS ^ \dagger}} ^ 2\bigr) + \min_{\cS \in \AA_{j_0}} \frac{2(\bgamma_{\cS} - \bgamma_{\cS ^ \dagger}) ^ {\top} \bepsilon}{n ^ {1 / 2}} + \max_{\cS \in \AA_{j_0}}\frac{1}{n}\bepsilon^{\top}(\bP_{\bX_{\cS}}-\bP_{\bX_{\cS ^ {\dagger}}})\bepsilon. 
	\end{aligned}
\eeq	
When $\Delta < \delta_0 \sigma (c_{\delta_0}\log p) ^ {1 / 2} / 20$, we reach the conclusion once we combine the bound above with \eqref{ineq:epsilon_linear_gap2}, \eqref{ineq:epsilon_quad_gap2} and \eqref{eq:quad_gamma_gap2}. 
\beq
\PP\biggl(\min_{\cS \in \AA_{j_0}}  \cL_{\cS} < \cL_{\cS ^ {\dagger}} \biggr) \ge 1 -4p ^ {-(\xi c\delta ^ 2_0 / 8 - 1)c_{\delta_0}} - p ^ {- \delta_0 ^ 2 c_{\delta_0} / (8C^2)}. 
\eeq


\subsection{Proof of Corollary \ref{cor1}}
The proof of Corollary \ref{cor1} is analogous to that of Theorem 2.1 by simply replacing $\AA(s)$ with $\AA_q(s)$. We omit the details for less redundancy.


\subsection{Lemma \ref{lem:dist_project} and its proof}

To start with, for any two size-$s$ sets $\cS \in \AA(s)$ and $\cS ^ {\ddagger} \in \AA ^ *(s)$ such that $\cS \cap \cS ^ * \subset \cS ^ {\ddagger}$, we analyze two crucial components of the projection signal margin,  ${\bmu^*}^{\top}\bP_{\cS^{\ddagger}|\cS}\bmu^*$ and ${\bmu^*}^{\top}\bP_{\cS|\cS^{\ddagger}}\bmu^*$, under Gaussian design.  
Suppose $\{\bx_i\}_{i\in[n]}$ are $n$ independent and identically distributed observations of $\bx\sim\cN(\textbf{0},\bSigma)$. 
For notational simplicity, let $\cS_1 = \cS \cap\cS^{\ddagger}, \cS_2 = \cS^{\ddagger} \setminus \cS_1 \text{ and } \cS_3 = \cS\setminus \cS_1$.	Figure \ref{fig:ms_set} shows the relationship between the sets.
\begin{figure}[h]        
      \centering
      \includegraphics[width=7cm]{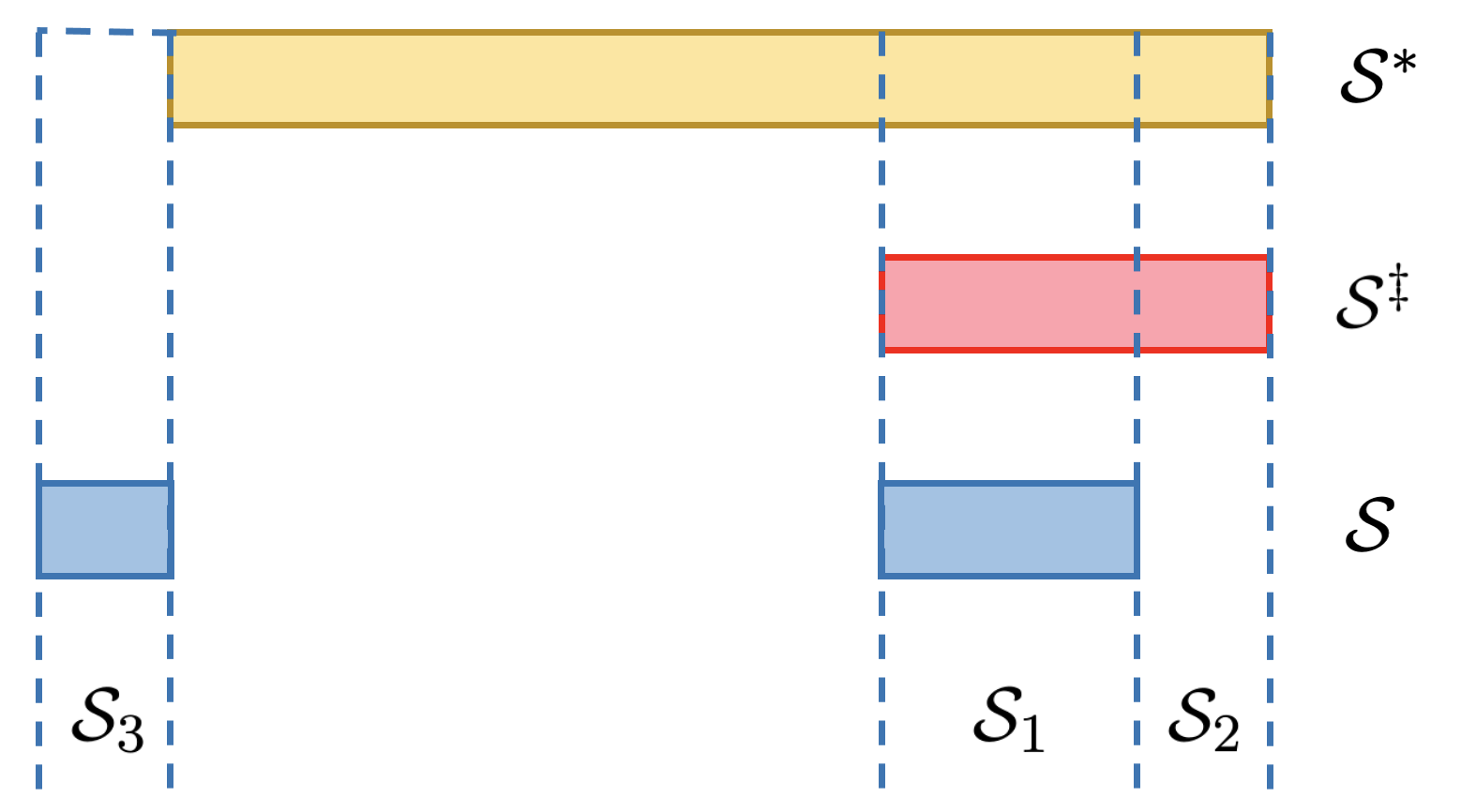} 
      \caption{The relationship between $\cS_1$, $\cS_2$ and $\cS_3$.}
      \label{fig:ms_set}      		      
\end{figure}

Next we introduce some quantities that are involved in the concentration bounds we establish. Define 
	\[
		\bh := \cov(\bx_{\cS_2}, \mu ^ *| \bx_{\cS_1}) = (\bSigma_{\cS_2\cS ^ *} - \bSigma_{\cS_2\cS_1}\bSigma^{-1}_{\cS_1\cS_1}\bSigma_{\cS_1\cS ^ *})\bbeta ^ * 
	\]
	and
	\[
		\bH := \cov(\bx_{\cS_2}, \bx_{\cS_2}| \bx_{\cS_1}) = \bSigma_{\cS_2\cS_2} - \bSigma_{\cS_2\cS_1}\bSigma^{-1}_{\cS_1\cS_1}\bSigma_{\cS_1\cS_2}. 
	\]
	Then write 
	\[
		\nu_1 := \bh^ \top \bH ^ {-1}\bh ~~\text{and}~~ \nu_2 := \var(\mu ^ *|\bx_{\cS_1}) - \nu_1 = \bbeta ^ {*\top} (\bSigma_{\cS ^ *\cS ^ *} - \bSigma_{\cS ^ *\cS_1} \bSigma_{\cS_1\cS_1} ^ {-1}\bSigma_{\cS_1 \cS ^ *})\bbeta ^ * - \nu_1.
	\]
	Similarly, define 
	\[
		\bh' := \cov(\bx_{\cS_3}, \mu ^ *| \bx_{\cS_1}) = (\bSigma_{\cS_3\cS ^ *} - \bSigma_{\cS_3\cS_1}\bSigma^{-1}_{\cS_1\cS_1}\bSigma_{\cS_1\cS ^ *})\bbeta ^ * 
	\]
	and
	\[
		\bH':= \cov(\bx_{\cS_3}, \bx_{\cS_3}| \bx_{\cS_1}) = \bSigma_{\cS_3\cS_3} - \bSigma_{\cS_3\cS_1}\bSigma^{-1}_{\cS_1\cS_1}\bSigma_{\cS_1\cS_3}. 
	\]
   Then write 
	\[
		\nu'_1 := \bh'^ \top {\bH'} ^ {-1}\bh' ~~\text{and}~~ \nu'_2 := \bbeta ^ {*\top} (\bSigma_{\cS ^ *\cS ^ *} - \bSigma_{\cS ^ *\cS_1} \bSigma_{\cS_1\cS_1} ^ {-1}\bSigma_{\cS_1 \cS ^ *})\bbeta ^ * - \nu'_1. 
	\] 
	Now we are in position to present the concentration bounds for ${\bmu^*}^{\top}\bP_{\cS^{\ddagger}|\cS}\bmu^*$ and ${\bmu^*}^{\top}\bP_{\cS|\cS^{\ddagger}}\bmu^*$. 

\begin{lemma}
     \label{lem:dist_project}
     For any two size-$s$ sets $\cS \in \AA(s)$ and $\cS ^ {\ddagger} \in \AA ^ *(s)$ such that $\cS \cap \cS ^ * \subset \cS ^ {\ddagger}$ and any $\xi  > 0$,  we have 
	\[
		\begin{aligned}
			\PP\biggl[\Bigl|{\bmu ^ *} ^ {\top}&  \bP_{\cS^{\ddagger} | \cS }\bmu ^ * - \{(n - s + t)\nu_1 + t \nu_2\}\Bigr|>  \\
			& \xi\Bigl\{3\nu_1(n - s + t)^{1 / 2} + 6(\nu_1\nu_2) ^ {1 / 2}(n - s + t) ^ {1 / 2}  + 3\nu_2t ^ {1 /2 }\Bigr\}  \biggr]   \le 6e ^ {-c\min(\xi ^ 2, \xi)} ~\text{ and }\\
			\PP\biggl[\Bigl|{\bmu ^ *} ^ {\top}&  \bP_{\cS | \cS^{\ddagger} }\bmu ^ * - \{(n - s + t)\nu'_1 + t \nu'_2\}\Bigr|>  \\
			& \xi\Bigl\{3\nu'_1(n - s + t)^{1 / 2} + 6(\nu'_1\nu'_2) ^ {1 / 2}(n - s + t) ^ {1 / 2}  + 3\nu'_2t ^ {1 /2 }\Bigr\}  \biggr] \le 6e ^ {-c\min(\xi ^ 2, \xi)},
		\end{aligned}
	\]
	where $t=|\cS^{\ddagger}\setminus\cS|$, and where $c$ is a universal constant.
\end{lemma}

\begin{proof}[Proof of Lemma \ref{lem:dist_project}]

Recall that $\AA_t(s):=\{\cS\in\AA(s):|\cS\setminus\cS^{*}|=t \}$ for $ t \in [s]$. For any $\cS\in\AA_t(s)$, we start with analyzing $\|\bP_{\cS^{\ddagger}|\cS}\bmu^*\|^2_2$. 	Note that $|\cS_2|=t$ and $|\cS_1|=s-t$.
	For each $j \notin \cS_1$, regressing $\bX_j$ on $\bX_{\cS_1}$ yields that
	\[
		\bX_{j} = \bX_{\cS_1} \bSigma_{\cS_1 \cS_1} ^ {-1} \bSigma_{\cS_1 j} + (\bX_j - \bX_{\cS_1} ^ {\top} \bSigma_{\cS_1 \cS_1} ^ {-1} \bSigma_{\cS_1 j}) =: \bX_{\cS_1}\btheta_j + \bGamma_j. 
	\]	
	Consider the conditional distribution of $\|\bP_{\cS^{\ddagger}|\cS}\bmu ^ *\|_2^2$ given $\bX_{\cS_1}$. 
	Let 
	\[
	\bI - \bP_{\bX_{\cS_1}} = \bV\bV ^ {\top} = \sum_{j = 1} ^ {n - (s - t)} \bv_j\bv_j ^ {\top}
	\]
	 be an eigen-decomposition of $\bI - \bP_{\bX_{\cS_1}}$.
	Note that $\bV$ is independent of $\bGamma_{\cS_1 ^ c}$, because $\bX_{\cS_1} $ is independent of $\bGamma_{\cS_1 ^ c}$. Then we have 
	\[
	    \begin{aligned}
			{\bmu ^ *} ^ {\top} \bP_{\cS^{\ddagger}| \cS}\bmu ^ * & =  {\bmu ^ *} ^ {\top}\bV\bV ^ {\top}\bX_{\cS_2}(\bX_{\cS_2} ^ {\top}\bV\bV ^ {\top}\bX_{\cS_2}) ^ {-1} \bX_{\cS_2} ^ {\top} \bV \bV ^{\top} \bmu ^ * \\
			& = {\bmu ^ *} ^ {\top} \bV\bV ^ {\top}\bGamma_{\cS_2}(\bGamma_{\cS_2} ^ {\top}\bV\bV ^ {\top}\bGamma_{\cS_2}) ^ {-1}\bGamma_{\cS_2} ^ {\top}\bV \bV ^{\top} \bmu ^ * \\
			& = {\tilde\bmu} ^ {* \top} \tilde \bGamma_{\cS_2}(\tilde \bGamma_{\cS_2} ^ {\top} \tilde \bGamma_{\cS_2}) ^ {-1} \tilde \bGamma_{\cS_2} ^ {\top} \tilde \bmu ^ *, 
		\end{aligned}
	\]
	where $\tilde\bGamma_{\cS_1 ^ c} := \bV ^ {\top} \bGamma_{\cS_1 ^ c}$ and $\tilde \bmu ^ * = \bV ^{\top} \bmu ^ *$. Besides, conditional on $\bX_{\cS_1}$, $\tilde\bGamma_{\cS_1 ^ c}$ and $\tilde \bmu ^ *$ have $(n - s + t)$ independent rows because of Gaussianity of $\bGamma_{\cS_1 ^ c}$ and $\bmu ^ *$ and orthogonality of $\bV$. Applying Lemma \ref{lem:projected_gaussian},  we obtain that for any $\xi  > 0$, 
	\beq
	\label{equ:applied_b1}
		\begin{aligned}
			\PP\biggl[\Bigl|{\bmu ^ *} ^ {\top} & \bP_{\cS ^ {\ddagger} | \cS }\bmu ^ * - \{(n - s + t)\nu_1 + t \nu_2\}\Bigr| \\
			& > \xi\Bigl\{3\nu_1(n - s + t)^{1 / 2} + 6(\nu_1\nu_2) ^ {1 / 2}(n - s + t) ^ {1 / 2}  + 3\nu_2t ^ {1 /2 }\Bigr\} ~\bigg|~ \bX_{\cS_1}\biggr] \le 6e ^ {-c\min(\xi ^ 2, \xi)}.  
		\end{aligned}
	\eeq
	Taking expectation with respect to $\bX_{\cS_1}$ on both sides of (\ref{equ:applied_b1}), we deduce that 
	\beq
		\label{eq:psm_bound}
			\begin{aligned}
			\PP\biggl[\Bigl|{\bmu ^ *} ^ {\top} & \bP_{\cS ^ {\ddagger} | \cS }\bmu ^ * - \{(n - s + t)\nu_1 + t \nu_2\}\Bigr| \\
			& > \xi\Bigl\{3\nu_1(n - s + t)^{1 / 2} + 6(\nu_1\nu_2) ^ {1 / 2}(n - s + t) ^ {1 / 2}  + 3\nu_2t ^ {1 /2 }\Bigr\} \biggr] \le 6e ^ {-c\min(\xi ^ 2, \xi)}.  
		\end{aligned}
	\eeq	
	For $\|\bP_{\cS|\cS^{\ddagger}}\bmu^*\|^2_2$, we can reach the conclusion by simply replacing $\cS_2$ with $\cS_3$.
	
\end{proof}


\subsection{Proof of Theorem 2.4}
For any $\cS\in\AA_t(s)$, we first consider any $\cS^\ddagger\in\AA^*(s)$ such that $\cS\cap\cS^* \subset \cS^\ddagger$.
Note that $\bSigma = \bI_p$, simple algebra yields that $\nu_1 = t\beta^2$, $\nu_2=(s^*-s)\beta^{2}$, $\nu_1{'} = 0$ and $\nu_2{'} = (s^*-s+t)\beta^2$. Applying Lemma \ref{lem:dist_project} yields that for any $\xi>0$, we have
\[
		\begin{aligned}
			\PP\biggl[\Bigl|{\bmu ^ *} ^ {\top}&  \bP_{\cS^{\ddagger} | \cS }\bmu ^ * - \{(n - s + t)t\beta^2 + t (s^*-s)\beta^2\}\Bigr|>  \\
			& \xi\Bigl\{3\nu_1(n - s + t)^{1 / 2} + 6(\nu_1\nu_2) ^ {1 / 2}(n - s + t) ^ {1 / 2}  + 3\nu_2t ^ {1 /2 }\Bigr\}  \biggr]   \le 6e ^ {-c\min(\xi ^ 2, \xi)} ~\text{ and }\\
			\PP\biggl[\Bigl|{\bmu ^ *} ^ {\top}&  \bP_{\cS | \cS^{\ddagger} }\bmu ^ * - t (s^*-s+t)\beta^2\}\Bigr|>  \xi\Bigl\{3(s^*-s+t)t ^ {1 /2 }\Bigr\}  \biggr] \le 6e ^ {-c\min(\xi ^ 2, \xi)},
		\end{aligned}
	\]
	where $c$ is a universal constant.

       For simplicity, let $\Delta:= 3\nu_1(n - s + t)^{1 / 2} + 6(\nu_1\nu_2) ^ {1 / 2}(n - s + t) ^ {1 / 2}  + 3\nu_2t ^ {1 /2 }$ and $\Delta' := 3(s^*-s+t)t ^ {1 /2 }$. Whenever $n > s ^ *$, there exists a universal constant $C_0 > 0$ such that $\max\{\Delta , \Delta'\}  \le C_0(ns ^ * t) ^ {1 / 2}\beta ^ 2 =: M\beta ^ 2$. 
       Writing $A = \{t(n+s^*-2s+t) - \xi M\} ^{1/2}|\beta|$ and $B = \{t(s^*-s+t) + \xi M\}^{1/2}|\beta|$, we deduce from the previous two bounds that
        \[
                 \PP\left(  \frac{ \|\bP_{\cS^\ddagger|\cS}\bmu^*\|_2 -   \|\bP_{\cS|\cS^\ddagger}\bmu^*\|_2}{t^{1/2}} < \frac{A-B}{t^{1/2}} \right)\le 12e ^ {-c\min(\xi ^ 2, \xi)}.
        \]
       Now we derive a lower bound on $(A-B)/t^{1/2}$. We have
       \[
	       \begin{aligned}
	               & \frac{A-B}{t^{1/2}} \ge \frac{A^2-B^2}{2t^{1/2}A} = \frac{\{(n-s)t - 2\xi M\}|\beta|}{2t^{1/2}\{t(n+s^*-2s+t) + \xi M \}^{1/2}}.
	       \end{aligned}
       \]
       If we have $n\ge 2s$, choosing $\xi = \xi_0 := (n - s)t / (4M)$ then yields that
       \[
       		\frac{A-B}{t^{1/2}} \ge \frac{n ^ {1 / 2}|\beta|}{24}.
       \]
       Therefore,
       \[
               \PP\left(  \frac{ \|\bP_{\cS^\ddagger|\cS}\bmu^*\|_2 -   \|\bP_{\cS|\cS^\ddagger}\bmu^*\|_2}{t^{1/2}} < \frac{n^{1/2}|\beta|}{24} \right)\le 12e ^ {-c\min(\xi_0 ^ 2, \xi_0)}.
       \]
        Define $\mathbb{F}(\cS):=\{\cS^{\ddagger} \in\AA^*(s): \cS\cap\cS^* \subset \cS^{\ddagger} \}$. Applying a union bound over $\mathbb{F}(\cS)$ yields that
        \[
               \PP\left(\psm(\cS)< \frac{n^{1/2}|\beta|}{24} \right)\le 12|\mathbb{F}(\cS)|e ^ {-c\min(\xi_0 ^ 2, \xi_0)}.
       \]
       According to Stirling's formula and the ultra-high dimension assumption, we have
       \[
       	|\mathbb{F}(\cS)|=\binom{s^*-s+t}{t} \lesssim \exp s^* \lesssim p
       \]
       and
       \[
       |\AA_t(s)| = \binom{p-s^*}{t}\binom{s^*}{s-t} \lesssim p^{t+1}.
       \]
A union bound over $\AA_t(s)$ yields that there exists a universal constant $C_1>0$ such that
       \[
	       \begin{aligned}
    	         \PP & \biggl( \exists \cS\in\AA_t(s)~\text{ s.t. }~ \psm(\cS) < \frac{n^{1/2}}{24}|\beta| \biggr) \le 12|\AA_t(s)|\cdot|\mathbb{F}(\cS)|e ^ {-c\min(\xi_0 ^ 2, \xi_0)} \\
	         & \le12C_1p^{t+2}e^{-c\min(\xi_0, \xi_0 ^ 2)}   \le 12C_1p^{3t}e^{-c\min(\xi_0, \xi_0 ^ 2)} = 12 C_1 e^{-\{c\min(\xi_0,\xi_0^2) - 3t\log p\}}.
    	   \end{aligned}
       \]
       Note that 
       \[
            \xi_0 = \frac{(n-s)t}{4M} \ge \frac{nt}{8M} = \frac{(nt)^{1/2}}{8C_0{s ^ * } ^ {1 / 2}}. 
       \]
       Assume that  $ n \ge \kappa s^* s (\log p)^2$ for some $\kappa >64C_0^2$. Then we have $\xi_0 \ge \kappa^{1/2}(8 C_0)^{-1}t \log p > 1$ and 
         \[
	       \begin{aligned}
    	         \PP & \biggl( \exists \cS\in\AA_t(s) ~ \text{ s.t. }  ~ \psm(\cS) < \frac{n^{1/2}}{24}|\beta| \biggr)\le 12 C_1 p^{-(c\kappa^{1/2}(8C_0)^{-1} - 3)t}.
    	   \end{aligned}
       \]
       Further applying a union bound over $t\in[s]$ yields that
       \[
              \PP \biggl(\psm_*(s) \ge \frac{n^{1/2}}{24}|\beta| \biggr)\ge 1 -  12 C_1 sp^{-(c\kappa^{1/2}C_0^{-1} - 3)}.
       \]
       Finally, let $C = \max\{12C_1, 8C_0/3, 8C_0/c, 2 ^ {1 / 2} / 3\}$. Then the conclusion follows for any $\kappa > 9 C^{2}$.

\subsection{Proof of Theorem \ref{thm:scrbss}}
By Theorem \ref{thm1}, we know that whenever
	\beq
		\psm_*(s) \ge \frac{8 \xi \sigma (\log p)^{1/2}}{1-\eta}, 
	\eeq	
	it holds that 
	\beq
		\PP\Bigl\{\fdp(\wh \cS) = 0,~\forall \wh\mS \in \mathbb{S}(s,\eta) \Bigr\} \ge 1-Csp^{-2(C^{-2}\xi^2-1)}.
	\eeq
	Consider the event $E_1 = \{ \fdp(\wh \cS) = 0,~\forall \wh\mS \in \mathbb{S}(s,\eta)\} =  \{\cS \subset \cS^*, \forall\cS \in \SS(s,\eta)\} $. The event indicates that $\cL_*$ is obtained within $\cS^*$. Consider the sure screening event $\cE = \{\cS^*\subset \wt{\cS}\}$.   If $E_1\cap \cE$ holds, we have $\wt{\cL}_* = \cL_*$ and thus
	\[
	     \wt{\SS}(s,\eta,\wt{\cS}) = \SS(s,\eta) ~~~\text{and} ~~~ \cS\subset \cS^*, \forall \cS\in \wt{\SS}(s,\eta,\wt{\cS}).
	\]
	So
	\[
	\begin{aligned}
	   \PP\bigl\{\fdp(\wh{\cS}   ) = 0 , \forall \wh{\cS} \in \wt{\SS}(s,\eta,\wt{\cS})
		\bigr\} & \ge \PP (E_1\cap \cE ) 
		 = 1 - \PP(E_1^c \cup \cE^c) \\
		& \ge 1 - \PP(E_1^c) - \PP(\cE^c)  = \PP(E_1) - \PP(\cE^c) \\
		& \ge  1-Csp^{-2(C^{-2}\xi^2-1)} - \PP(\cE^c).
	\end{aligned}
	\]

\section{Technical lemmas}
\begin{lemma}
	\label{lem:projected_gaussian}
	Consider $n$ independent and identically distributed observations $(Y_i, \bx_i)_{i \in [n]}$ of $(Y, \bx) \sim \cN(\bzero, \bSigma)$, where $Y$ is valued in $\RR$ and $\bx$ is valued in $\RR ^ p$. Write $\bX = (\bx_1, \ldots, \bx_n) ^ {\top}$ and $\by = (y_1, \ldots, y_n) ^ {\top}$. Define $\nu_1 := \bSigma_{Y\bx}\bSigma_{\bx\bx} ^ {-1} \bSigma_{\bx Y}$ and $\nu_2 := \Sigma_{YY} - \nu_1$. Then we have that
	\[
		\begin{aligned}
			\PP\biggl[\bigl|\by ^ {\top}\bP_{\bX}\by - \{n\nu_1 + p \nu_2\}\bigr| > \bigl\{3\nu_1n^{1 / 2} + 6(\nu_1\nu_2n) ^ {1 / 2} & + 3\nu_2p ^ {1 /2 }\bigr\}\xi\biggr] \le 6e ^ {-c\min(\xi ^ 2, \xi)}, 
		\end{aligned}
	\]
	where $c$ is a universal constant. 
\end{lemma}
\begin{proof}[Proof of Lemma \ref{lem:projected_gaussian}]
	Regressing $Y$ on $\bx$ yields that 
	\[
		Y = \bx ^ {\top}\btheta + Z, 
	\]
	where $\btheta := \bSigma_{\bx\bx} ^ {-1}\bSigma_{\bx Y}$, and where $Z$ is independent of $\bx$. Some algebra yields that 
	\[
		\nu_1 := \var(\bx ^ \top \btheta) = \bSigma_{Y\bx}\bSigma_{\bx\bx} ^ {-1} \bSigma_{\bx Y}
	\]
	and that 
	\[
		\nu_2 := \var(Z) = \bSigma_{YY} - \bSigma_{Y\bx}\bSigma_{\bx\bx} ^ {-1} \bSigma_{\bx Y}. 
	\]
	Now consider
	\[
		\by ^ {\top}\bP_{\bX}\by = \btheta ^ {\top}\bX ^ {\top}\bX\btheta + 2\bz ^ {\top}\bX\btheta + \bz ^ {\top} \bP_{\bX} \bz. 
	\]
	We bound the three terms on the right-hand side one by one. Note that 
	\[
		\|(\btheta ^ {\top} \bx) ^ 2\|_{\psi_1} \le 3\nu_1, ~\|(\ba ^ \top \bz) ^ 2\|_{\psi_1} \le 3 \nu_2, \forall \ba \in \cS ^ {n - 1}~\text{and}~\|Z(\btheta ^{\top} \bx)\|_{\psi_1} \le 6(\nu_1\nu_2) ^ {1 / 2}. 
	\]
	Given that $\bz$ is independent of $\bX$, applying Bernstein's inequality yields that for any $\xi > 0$, 
	\[
		\PP\biggl(\biggl|\frac{\ltwonorm{\bX\btheta} ^ 2}{n} - \nu_1 \biggr| \ge \xi \biggr) \le 2\exp\bigl[-cn\min((3\nu_1) ^ {-2}\xi ^ 2, (3\nu_1) ^ {-1}\xi)\bigr], 
	\]
	\[
		\PP\biggl(\biggl|\frac{\bz ^ {\top} \bX\btheta}{n} \biggr| \ge \xi \biggr) \le 2\exp\biggl[-cn\min\biggl\{\frac{\xi ^ 2}{36 \nu_1\nu_2}, \frac{\xi}{6(\nu_1\nu_2) ^ {1 /2 }}\biggr\}\biggr]
	\]
	and
	\[
		\PP\biggl(\biggl|\frac{\ltwonorm{\bz ^ {\top} \bP_{\bX} \bz} ^ 2}{ p} - \nu_2 \biggr| \ge \xi \biggr) \le 2\exp\bigl[-c p\min\{(3\nu_2) ^ {-2}\xi ^ 2, (3\nu_2) ^ {-1}\xi)\}\bigr], 
	\]	
	where $c$ is a universal constant. Combining the three bounds above yields the conclusion. 	
\end{proof}

\begin{lemma}
	\label{lem:var_of_max}
	Given two random variables $X_1$ and $X_2$ valued in $\RR$, $\var\{\max(X_1, X_2)\} \leq \var(X_1)+\var(X_2)$.
\end{lemma}
\begin{proof}
	$\var\{\max(X_1, X_2)\} =\var\{(X_1+X_2)/2+|X_1-X_2|/2 \}\leq \frac{1}{2}\var(X_1+X_2)+\frac{1}{2}\var(X_1-X_2)=\var(X_1)+\var(X_2)$.
\end{proof}

\begin{lemma}
	\label{lem:var_sup_process}
(\cite{van2014probability}, Lemma 6.12). Let $\{X_t \}_{t\in T}$ be a separable Gaussian process. Then sup$_{t\in T}X_t$ is sup$_{t\in T}$Var$(X_t)$-subgaussian.
\end{lemma}

\end{appendix}

\end{document}